\documentclass[11pt]{article}
\usepackage{tgpagella,eulervm}
\usepackage{setspace} 
\usepackage{framed} 
\usepackage{relsize}

\usepackage{color,soul}

\usepackage[framemethod=default]{mdframed} 

\usepackage{amssymb,mathtools,amsthm, amsmath, amsfonts}
\usepackage[refpage]{nomencl}

\usepackage{authblk}
\usepackage{fullpage}
\usepackage{enumitem}
\usepackage{graphicx} 
\usepackage{tikz-cd}
\usetikzlibrary{fit,arrows,calc,positioning}
\usepackage[hidelinks]{hyperref}
\hypersetup{
	colorlinks   = true,
	citecolor    = red
}
\hypersetup{linkcolor=blue}

\usepackage[capitalize,nameinlink]{cleveref}
\usepackage{array}
\usepackage{mathbbol}
\usepackage{listings}
\allowdisplaybreaks

\crefname{figure}{Figure}{Figures}

\usepackage[normalem]{ulem}
\usepackage{mathrsfs} 

\newtheorem{corollary}{Corollary}[section]
\newtheorem{theorem}{Theorem}
\newtheorem*{theoremnn}{Theorem}
\newtheorem{remark}[corollary]{Remark}
\newtheorem{lemma}[corollary]{Lemma}
\newtheorem{proposition}[corollary]{Proposition}
\newtheorem{definition}[corollary]{Definition}

\newtheorem{example}[corollary]{Example}

\newcommand{\Ffunc}          	{\mathsf{F}} 
\newcommand{\Gfunc}          	{\mathsf{G}} 
\newcommand{\fnc}          	    {\mathsf{Fnc}} 
\newcommand{\mpfil}          	    {\mathsf{MP\textnormal{-}Fil}} 
\newcommand{\Zfunc}          	{\mathsf{Z}} 
\newcommand{\Bfunc}          	{\mathsf{B}} 

\newcommand{\ZB}          	    {\mathsf{ZB}} 

\newcommand{\gr}          	    {\mathsf{Gr}} 
\newcommand{\Int}          	    {\mathsf{Int}} 
\newcommand{\Seg}          	    {\mathsf{Seg}} 
\newcommand{\diag}          	{\mathsf{diag}} 
\newcommand{\mobeq}          	{\simeq_{\mathsf{M\ddot{o}b}}} 
\newcommand{\mon}          	{\mathsf{Mon}} 
\newcommand{\subcx}          	{\mathsf{SubCx}} 
\newcommand{\goi}          	{\mathsf{OI}} 
\newcommand{\Ffrak}          	{\mathfrak{F}} 
\newcommand{\Gfrak}          	{\mathfrak{G}} 
\newcommand{\Mfrak}          	{\mathfrak{M}} 
\newcommand{\Nfrak}          	{\mathfrak{N}} 
\newcommand{\dtfnc}          	    {\mathsf{DT\textnormal{-}Fnc}} 
\newcommand{\lowersum}          	    {\mathsf{LowerSum}}  
\newcommand{\lp}                     {\mathbb{L}}
\newcommand{\dgr}                     {\rho}
\newcommand{\monfnc}                    {\mathsf{MonFnc}}

\newcommand\ladj[1]{#1_{\diamond}} 
\newcommand\radj[1]{#1^{\diamond}} 
\newcommand\mmi[2]{\partial_{#1}^\mathsf{Mon}\left(#2\right)} 

\newcommand{\R}{\mathbb{R}}
\newcommand{\N}{\mathbb{N}}

\newcommand{\Z}{\mathbb{Z}}

\renewcommand{\Vec}		{{{\mathsf{Vec}}}}

\newcommand{\incord}{\leq_{\mathsmaller{\supseteq}}}
\newcommand{\prodord}{\leq_{{\times}}}

\DeclareMathOperator{\proj}{proj}
\DeclareMathOperator{\Ima}{im}
\DeclareMathOperator{\spn}{span}
\DeclareMathOperator{\rank}{rank}

\newcommand{\dis}{\mathrm{dis}}

\title{Grassmannian Persistence Diagrams}

\author[1]{Aziz Burak G\"ulen\footnote{\href{mailto:aziz.burak.guelen@duke.edu}{aziz.burak.guelen@duke.edu}, 120 Science Dr, Durham, NC 27708}}
\author[2]{Facundo M\'emoli\footnote{\href{mailto:facundo.memoli@gmail.com}{facundo.memoli@gmail.com}, 110 Frelinghuysen Road, Piscataway, NJ 08854}}
\author[3]{Zhengchao Wan\footnote{\href{mailto:zwan@missouri.edu}{zwan@missouri.edu}, 810 Rollins St, Columbia, MO 65201}}
\affil[1]{Department of Mathematics, Duke University}

\affil[2]{Department of Mathematics, Rutgers University}

\affil[3]{Department of Mathematics, University of Missouri}

\date{}

\makenomenclature

\begin{document}

\maketitle

\begin{abstract}
    We introduce Orthogonal M\"obius Inversion, a concept analogous to M\"obius inversion on finite posets, which is applicable to order-preserving functions from a finite poset to the Grassmannian $\gr(V)$ of an inner product space $V$. This notion critically relies on the inner product structure on $V$ enabling it to capture much finer information than standard integer-valued persistence diagrams.
            
    Orthogonal Inversion is a special case of the broader concept of Orthomodular Inversion, where the target space is any orthomodular lattice, which we also identify.
            
    We apply Orthogonal Inversion in order to construct a "non-negative" persistence diagram for any given multiparameter filtration $\Ffunc$ of a finite simplicial complex $K$, indexed over an arbitrary finite poset $P$. This is done by applying it to the birth-death spaces of $\Ffunc$. Analogously to  $1$-parameter classical persistence diagrams, these multiparameter Grassmannian persistence diagrams offer straightforward interpretability. Specifically, to a segment $(b, d) \in \Seg(P)$, (1) the Grassmannian persistence diagram canonically assigns a vector subspace of $C_\dgr^K$ consisting of cycles that are born at $b$ and become boundaries at $d$ and (2) this assignment is exhaustive at the homology level. 

Finally, we relate our Grassmannian persistence diagrams to the recently introduced notion of Möbius homology, thus enhancing its interpretability through the lens of our framework.



\end{abstract}

\newpage
\printnomenclature[2cm]
\addcontentsline{toc}{section}{Nomenclature}

\tableofcontents

\section{Introduction}

In Topological Data Analysis (TDA) \cite{carlsson2009,edelsbrunner2010computational}, \emph{persistent homology} 
 is a tool for tracking and characterizing the evolution (i.e. the ``birth" and ``death") of homological features throughout a filtration. In the case of a 1-parameter filtration, this is achieved through an algebraic decomposition of the persistence module which is obtained by applying the homology functor to the given filtration. Every 1-parameter persistence module $M:\lp\to \Vec$, where $\lp := \{1<\cdots<n \}$ is any finite linear poset,  admits an interval decomposition \cite{zomorodian2005,crawley2015decomposition}: that is, $M$ can be written as a direct sum of \emph{indecomposables} which turn out to be \emph{interval} modules. Each interval in the decomposition of the module corresponds to the \emph{lifespan} of a topological feature in the filtration (i.e. the maximal region throughout which a topological feature is ``alive"). Interestingly, the multiplicities of intervals in this decomposition can be obtained through a notion of inclusion-exclusion, i.e. through \emph{M\"obius inversion}. More precisely, the function that assigns to each interval $(i,j)$  its multiplicity in the decomposition coincides with the \emph{M\"obius inverse} of the rank function $rk_M(i,j) := rk(M(i) \to M(j))$; see  \cite{landi1997new,cohen-steiner2007,Patel2018}.

In the case of multiparameter filtrations, i.e. filtrations defined on an arbitrary (say,  finite) poset $P$,  the  algebraic approach presents us with certain challenges. First off, the generalization of the notion of interval is substantially richer than its precursor in the 1-parameter case and,  more importantly, the indecomposables are no longer only interval modules. Furthermore, the space of indecomposables can be enormously complex \cite{carlsson2009,bauer2020cotorsion,kim2023persistence}.

Regarding the first point, given a poset $P$, and $b\leq d$ in $P$, the \emph{segment} $(b,d)$ will be the set of all $p\in P$ such that $b\leq p\leq d$. By $\Seg(P)$ we will denote the collection of all segments of $P$. The \emph{intervals}, $\Int(P)$, in $P$ are more general than segments;\footnote{When $P=\lp$, one has $\Int(P) = \Seg(P)$ but in general $\Int(P)$ is much larger; see \cite[Theorem 31]{asashiba2023approximation-2d}.}  we will not require their precise definition in this paper (see for example \cite[Definition 2.6]{Kim2021}).

Nevertheless, analogously to the 1-parameter setting, a notion of  persistence diagram can still be defined for a module $M : P \to \Vec$  over a poset $P$ (which could be a multidimensional grid) by applying M\"obius inversion to the rank function $rk_M (b,d) := rk(M(b) \to M(d))$, for $(b,d)\in\Seg(P)$. Since the resulting persistence diagram can no longer be guaranteed to always be non-negative, it is usually referred to as a \emph{signed} persistence diagram \cite{betthauser2019graded,Kim2021,botnan2022signed,edit,morozov-patel,gal-conn,kim2023persistence,metarank-clause}; see also \cite{asashiba2023approximation-2d,asashiba2023approximation}.

\begin{figure}
    \centering
    \tikzstyle{l} = [draw, -latex',thick]
\begin{tikzpicture}[auto,
 block/.style ={rectangle, thick, fill=blue!10, text width=5.0em,align=center, rounded corners},
 longblock/.style ={rectangle, thick, fill=green!10, text width=7em,align=center, rounded corners},
 boundblock/.style ={rectangle, thick, fill=red!10, text width=5.5em,align=center, rounded corners},
 longboundblock/.style ={rectangle, thick, fill=red!10, text width=7em,align=center, rounded corners},
 line/.style ={draw, -latex', shorten >=2pt},
  ]
    \node [block] (fil) {\scriptsize Simplicial Filtration};
    \node [right =6.5em of fil] (a) {};
    \node [longblock, below =1em of a] (chain) {\scriptsize Filtered Chain Complex (FCC)};
    \node [block, right =6em of a] (pmod) {\scriptsize Persistence Modules};
    \node [block, right =6em of pmod] (pdgm) {\scriptsize Persistence Diagrams};
    \path [line, dashed, gray] (fil) |- node [below] {\scriptsize \; \;\; \textcolor{black}{Chain Complex}} (chain);
    \path [line, black] (fil) -- node [above] {\scriptsize \; \;\; Homology} (pmod);
    \path [line, dashed, gray] (chain) -- node [text width=8em, above] {} (pmod);
    \path [line, black] (pmod) -- node [text width=8em, above] {\scriptsize \; \;\;\;\;M\"obius Inversion} (pdgm);
    \path [l, dashed] (chain) -| node [below] {} (pdgm);

\end{tikzpicture}
    \caption{The intermediate step of FCCs in the persistent homology pipeline.}
    \label{fig: pipeline with fccs}
\end{figure}
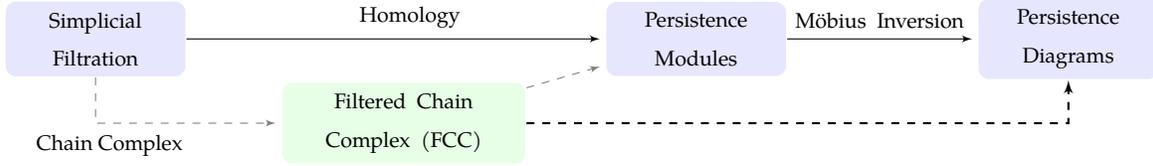

Recently, going beyond strict persistent \emph{homology}~\cite{zhou-thesis,thesis-guelen}, invariants arising from \emph{filtered chain complexes} (see \Cref{fig: pipeline with fccs}) such as \emph{birth-death functions} \cite{saecular, edit, morozov-patel,fasy2022persistent}, the \emph{persistent Laplacian} \cite{lieutier-pdf, persSpecGraph, persLap}, \emph{verbose barcodes}  \cite{usher, ephemral,fasy2019faithful,     landi2021invariants,chacholski2023decomposing}, as well as \emph{persistent cup-product} related functions \cite{contessoto_et_al,cup-persistent} at the the level of cohomology, and \emph{persistent homotopy groups}~\cite{memoli-zhou-pi1},
have been studied to track and characterize the evolution of topological features throughout a filtration.

In particular, it has been shown by McCleary and Patel \cite{edit} that, for a given  filtration $\Ffunc$ with domain a finite lattice $P$ and  for an integer $\dgr\geq 0$, a notion of persistence diagram can also arise as the M\"obius inverse of the \emph{birth-death function} $\dim\left(\ZB_\dgr^\Ffunc\right):\Seg(P)\to \mathbb{Z}$, where for each $(b,d)\in \Seg(P)$ the value $\dim\left(\ZB_\dgr^\Ffunc((b,d))\right)$ counts the number of $\dgr$-cycles that are born at or before $b$ and die (i.e. become boundaries) at or before $d$ along the filtration $\Ffunc$; see \cref{parag: poset of int}. This point of view was furthered by Morozov and Patel in \cite{morozov-patel}, where they consider the case of 2-parameter filtrations (and develop efficient algorithms for the computation of the resulting persistence diagrams), and by G\"ulen and McCleary in~\cite{gal-conn}, where they delve into  theoretical aspects.

The  persistence diagram induced by the birth-death function is stronger than the rank induced one. Indeed, it has recently  been shown by G\"ulen and McCleary that (on any finite poset) the (potentially signed) persistence diagram obtained from the rank function can be reconstructed from the signed persistence diagrams obtained from the birth-death function; see  \cite[Proposition 8.3]{gal-conn}. Furthermore, when derived from a filtration, birth-death function induced diagrams carry additional information compared to rank function induced diagrams in the form of \emph{ephemeral cycles}; see \cite{edit, gal-conn}.\footnote{The terminology is not used in those papers though. It first appeared in \cite{usher}.}

\begin{figure}
    \centering

\includegraphics[width=0.1\linewidth]{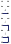}
\includegraphics[width=0.55\linewidth]{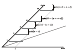}

    \caption{Grassmannian persistence diagram of the 1-parameter filtration depicted on the left. Grassmanian persistence diagrams  retain information about cycle spaces associated to different intervals. For example, for the interval $(1,2)$ the Grassmanian persistence diagram not only captures the multiplicity of that interval as the dimension of the space $\mathrm{span}\{a-c\}$ but also provides cycles that are precisely born at $1$ and die at $2$.}
    \label{fig:3d-gpd}
\end{figure}
\paragraph{The interpretability of signed persistence diagrams.} As mentioned above, whereas in the 1-parameter setting it is  straightforward to interpret persistence diagrams (arising either from rank or birth-death functions), in the multiparameter setting signed persistence diagrams are only partially interpretable:

\begin{enumerate}
\item When a multiparameter persistence module is segment/interval decomposable, its rank induced persistence diagram is non-negative and it captures the multiplicity  of each segment/interval appearing in is decomposition \cite[Theorem 3.14]{Kim2021}. However, in the topology induced by the interleaving distance on $\R^n$ modules (for $n\geq 2$), the collection of all almost-indecomposable modules is open and dense, i.e. the defining property is generic; see  \cite[Corollary 1]{bauer2022generic}. That is, the setting of decomposable multiparameter modules is indeed relatively restrictive which makes results such as \cite[Theorem 3.14]{Kim2021} rarely applicable in practise.

\item  Sometimes negative values are  attained by  rank induced persistence diagrams (see e.g. \cite[Example 18]{kim2023persistence}) and, by the previous item, these signify that the underlying persistence module is not segment/interval decomposable. 

\item  Even if a rank induced persistence diagram is non-negative, the underlying persistence module cannot be guaranteed to be segment/interval decomposable. Indeed, segments/intervals that are assigned positive values by the rank persistence diagram are not necessarily linked to a direct summand of the module; see  \cite[Example 3.21]{Kim2021}.

\item In the case of birth-death induced or rank induced signed persistence diagrams, even when non-negative, the integer number assigned to a segment  $(b,d)\in \Seg(P)$ does not necessarily correspond to the number of cycles that are born at $b$ and die at $d$; see~\cref{fig: filtrations for nonfully capturing} and \cref{ex: signed pd dont fully capture}. 
\end{enumerate}

In summary, what the Möbius inversion of the rank or birth-death function achieves in the $1$-parameter setting does not fully carry over to the multiparameter setting. The interpretation/meaning of the numbers determined by such persistence diagrams is unclear. Specifically, there is a lack of connection between the persistent diagrams obtained as M\"obius inverses of the rank and birth-death functions and the evolution of cycles throughout a filtration in the multiparameter context. Our work addresses these shortcomings by identifying a novel, and richer, variant of the notion of persistence diagram which admits direct interpretation; see \cref{fig:3d-gpd}.

\subsection{Main Contributions}

We introduce a novel variant of M\"obius inversion, termed \emph{Orthogonal Inversion}, which operates on functions $$\Ffrak:R\to \gr(V)$$ defined on any finite poset $R$ and valued in the Grasmannian of a finite dimensional inner product space $V$; see \cref{defn: goi}. When $\Ffrak$ is order preserving, its orthogonal inverse $\goi(\Ffrak):R\to \gr(V)$ behaves like a derivative in the sense that 
\[
    \sum_{r' \leq r} \goi (\Ffrak)(r') = \Ffrak(r)\,\,\text{for every $r\in R$;}
    \]
     see \Cref{prop: goi is monoidal mi}. When orthogonal inversion is applied to a $\gr(V)$-valued function $\Ffrak$ defined on the poset of segments (with the product order) of a finite poset $P$, that is when $R=\Seg(P)$, we obtain the \emph{Grassmannian persistence diagram} $$\goi(\Ffrak):\Seg(P)\to \gr(V)$$ of $\Ffrak$. \textbf{This notion critically relies on the inner product structure on $V$ enabling it to capture much finer information than integer-valued persistence diagrams.}\footnote{In our constructions, this inner product propagates from the simplicial structure onwards to the level of chain complexes.}

\begin{figure}[t]
    \centering
    \includegraphics[scale=13]{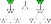}
    \caption{Hasse diagram of a poset $P$ with 4 elements (shown in the middle). Two different filtrations $\Ffunc$ (on the left) and $\Gfunc$ (on the right) indexed by the poset $P$. Note that the simplicial complexes $\Ffunc(p_1)=\Ffunc(p_2)$ have exactly one 1-cycle whereas  $\Gfunc(p_1)\neq \Gfunc(p_2)$ and each supports exactly one 1-cycle.  These two filtrations have identical degree-1 signed persistence diagrams (for both birth-death and rank functions) and, furthermore, both are non-negative. Indeed, both only produce the (multiplicity) values $0$ or $1$ and, hence, are unable to detect the $2$  different homology classes that are born at $p_3$ and die at $p_4$ in the filtration $\Ffunc$. In contrast, their Grassmannian persistence diagrams in degree-1 are different and, for each segment, they capture the precise number of degree-1 homology classes supported on exactly that segment. Neither the birth-death nor the rank function induced signed persistence diagrams achieve this in spite of their non-negativity. Beyond this example, as a general feature of Grassmanian persistence diagrams is that they are always guaranteed to capture all (persistent) homology classes (\cref{thm: completeness}). See \cref{ex: signed pd dont fully capture} for details. }
    \label{fig: filtrations for nonfully capturing}
\end{figure}

\paragraph{Interpretability.} Our method, when applied to the birth-death space valued function $\ZB_\dgr^\Ffunc$ of a filtration $\Ffunc$ indexed by a finite poset $P$, is able to completely capture the evolution of topological features along this filtration without imposing any limitation on the domain poset $P$:

\begin{itemize}
\item For any segment $I = (b,d)$ of the poset $P$, the vector space $\goi (\ZB_\dgr^\Ffunc)(I)$ consists of \emph{cycles} that are born precisely at b and die precisely at d; see \cref{prop: born and dead multiparameter}.   Furthermore, such a vector space is canonically associated to every segment $I$; see~\cref{prop: canonicality}. 

\item As a refinement of the previous item, this assignment is exhaustive at the level of homology.  Indeed, for any segment $I=(b,d)$, \emph{every homological feature born at $b$ and dying at $d$ is captured by $\goi \left(\ZB_\dgr^\Ffunc\right)(I)$}; see \cref{thm: completeness}. In particular, we have that  $\dim \left( \goi \left(\ZB_\dgr^\Ffunc\right) ((b,d)) \right)$ is always at least the number of independent homological features that are born at $b$ and die at $d$. As we mentioned above, this exhaustiveness property is not enjoyed by currently  available notions of persistence diagrams over general posets; see \cref{fig: filtrations for nonfully capturing} and~\cref{ex: signed pd dont fully capture}.

\end{itemize}

\paragraph{Discriminative power.} In addition to admitting a more direct interpretation than both standard birth-death and rank-induced persistence diagrams, the Grassmannian persistence diagram $\goi\left(\ZB_\dgr^\Ffunc\right)$ is also stronger than these notions. Indeed, the birth-death induced persistence diagram can be constructed from $\goi\left(\ZB_\dgr^\Ffunc\right)$ (see \cref{prop: goi to signed pd}), whereas it is also already known that the rank induced persistence diagram can be recovered from the birth-death induced one; see \cite[Proposition 8.3]{gal-conn}. Furthermore, there are cases in which (both birth-death and rank induced) signed persistence diagrams confuse filtrations that are distinguished by Grassmannian persistence diagrams; one such example is described in \cref{fig: filtrations for nonfully capturing} and elaborated on in \cref{ex: signed pd dont fully capture}. Other manifestations of the added strength are discussed below.

\paragraph{Functoriality and Stability.} A desirable property of any reasonable notion of persistence diagram is its stability. By building upon \cite{edit}, where the authors construct a functorial pipeline for persistence,\footnote{Their functoriality  critically rely on the use of the product order on segments as opposed to the more commonly used inclusion order.} 
we show that our Grassmannian persistence diagrams are stable under a suitable notion of edit distance. Our stability result is, in fact, a byproduct of the functorial nature of our construction; see \cref{thm: stability of mp gpd}.

\paragraph{Enhancing interpretability of persistent M\"obius Homology.} 
Unlike the conventional approach of associating a \emph{single} persistence diagram to a given persistence module, Patel and Skraba’s recent work~\cite{patel-skraba-mobius} introduces a novel perspective: a framework that assigns to a persistence module an entire \emph{family of persistence diagrams}, indexed by a nonnegative integer, which we refer to as the \emph{stratum} and denote by $s$.\footnote{For a given integer  $s\geq 0$, we will call the corresponding element in this family  the stratum-$s$ persistent M\"obius homology of the given module.}

A key insight from \cite{patel-skraba-mobius} is that the Euler characteristic (i.e., alternating sum) of this family of persistence diagrams precisely recovers the birth-death induced signed persistence diagram~\cite[Theorem 3.13]{patel-skraba-mobius}. Since the birth-death induced signed persistence diagram itself is obtained through Möbius inversion, this result can be interpreted as Möbius homology providing a \emph{categorification} of Möbius inversion—hence the name Möbius homology. Moreover, this perspective can also be seen as unraveling the intricate structure of the birth-death induced signed persistence diagram of a module by breaking it into simpler, more manageable pieces (the different strata).  Furthermore, similarly to what we do in this paper in terms of Grassmanian persistence diagrams, Patel and Skraba demonstrate that this family consisting of different strata is strictly more informative than the birth-death induced signed persistence diagram for a given persistence module~\cite[Example 5.10]{patel-skraba-mobius}. 

The M\"obius Homology perspective is not only novel and intriguing but also represents a conceptual breakthrough in our understanding of persistence modules. Despite this richness, the precise meaning and implications of these stratified diagrams remain largely unexplored. In this work, we contribute to a deeper understanding of these diagrams by leveraging the interpretability of Grassmannian persistence diagrams established here and uncovering a fundamental connection to Patel and Skraba’s work—one that may suggest a future higher-order construction of Grassmannian persistence diagrams. 

In particular, we clarify the interpretation of the $0$-th stratum of persistent Möbius homology of a module by leveraging the interpretability of Grassmannian persistence diagrams established in this paper. Specifically, when a persistence module arises from a filtration via the homology functor, we show in \cref{prop: grade 0 pers mob hom isomorphic to gpd} that, for the $0$-th stratum, Patel and Skraba's construction assigns to each segment $(b, d)$ a vector space whose dimension coincides with that of the vector space specified by the Grassmannian persistence diagram of the filtration at $(b, d)$. As a consequence, and in light of \cref{prop: born and dead multiparameter}, we are able to prove that the stratum-$0$ persistent Möbius homology also encodes the number of independent cycles born at $b$ and dying at $d$; see \cref{coro:MH-ex}.

\subsection{Organization of the Paper}

In~\cref{sec: monoidal mobious inversion}, we introduce the notion of \emph{Monoidal M\"obius inverse} (\cref{defn: monoidal mobius inversion}). The common theme across the applications of M\"obius inversion in TDA \cite{Patel2018} is that the functions to which M\"obius inversion is applied are valued in an abelian group. This abelian group often arises as the group completion of the monoid consisting of the isomorphism classes of objects in an abelian category $\mathcal{A}$. The group completion process is employed in order to make sense of the ``minus signs'' in the defining formula for the M\"obius inversion. The concept of monoidal M\"obius inverse  (see~\cref{defn: monoidal mobius inversion}) directly applies  to functions valued in a monoid (i.e. it does not rely on its  group completion). This notion permits identifying a non-trivial notion of M\"obius inversion in situations when the group completion of the target monoid is trivial, which is the case for $\gr(V)$ as demonstrated in~\cref{appendix:details}. Furthermore, we show that monoidal M\"obius inverses satisfy a monoidal version of Rota's Galois Connection Theorem; see~\cref{thm: monoidal rgct}. As a result, the process of assigning a monoidal M\"obius inverse to a monotone function through Orthogonal Inversion inherits the functoriality of classical M\"obius inversion~\cite{edit, gal-conn}. 

In~\cref{sec:generalized orthogonal inversion}, we introduce the concepts of Orthogonal Inversion and Orthomodular Inversion. Orthogonal Inversion applies to order-preserving functions defined from a finite poset to a Grassmannian. When applied to such a function, Orthogonal Inversion produces a monoidal Möbius inverse of the function, as shown in~\cref{prop: goi is monoidal mi}. This provides us with a functorial and stable way (see~\cref{thm: goi functor and stable,thm: isometry of goi}) to obtain a notion of persistence diagram, which we refer to as Grassmannian persistence diagrams.

In~\cref{subsec: gpd for multiparameter}, we introduce the notion of the Grassmannian persistence diagram as the Orthogonal Inverse of the birth-death spaces (\cref{defn: gpd of multiparameter filtrations}). In~\cref{subsec: inter of gpd canonical cyyles and exhaustiveness}, we show that the Grassmannian persistence diagram of a filtration $\Ffunc : P \to \subcx(K)$ canonically assigns a cycle space to every segment $(b,d) \in \Seg(P)$, consisting of cycles that are born at $b$ and die at $d$ (see~\cref{prop: born and dead multiparameter,prop: canonicality}). Moreover, this assignment is exhaustive at the homology level, as established in~\cref{thm: completeness}. In~\cref{subsec: stability of grassmannian pd}, we prove the stability of Grassmannian persistence diagrams in~\cref{thm: stability of mp gpd}. Finally, in~\cref{subsec: compare gpd to signed}, we compare the strength and discriminative power of Grassmannian persistence diagrams to that of birth-death induced signed persistence diagrams and provide a lower bound for the edit distance between Grassmannian persistence diagrams of two filtrations by the edit distance between their birth-death induced signed persistence diagrams (\cref{thm: edit between signed is lower bound for edit between grassmannian}).

In~\cref{sec: pers mob hom}, we consider a notion of persistent M\"obius homology  specifically directly applicable to filtrations and establish \cref{prop: grade 0 pers mob hom isomorphic to gpd}, which provides precise   relationship between the Grassmannian persistence diagram and the 0-th stratum of persistent M\"obius homology of a filtration.   In~\cref{subsec: pers mob hom of modules}, we prove that the our filtration-based version of persistent M\"obius homology is isomorphic to the module-based persistent M\"obius homology of \cite{patel-skraba-mobius} when their input module is the one obtained by applying the homology functor to the given filtration; see~\cref{thm: mob hom ker is mob hom fil}.

Finally, in \cref{sec:disc} we collect a few questions that might motivate further research.

\subsection{Related Work}
Our work is inspired by and builds upon several recently explored directions. First, our functoriality and stability results extend those established in~\cite{edit}. Second, the idea of equipping chain spaces with an inner product to obtain vector spaces corresponding to intervals in the persistence diagram has been previously explored in~\cite{harmonicph}, and we further develop this approach.\footnote{See also~\cite{ogle-sultan-inner} for inner product structures considered at the level of modules.}  Additionally, our homological exhaustiveness result (\cref{thm: completeness}) generalizes and leverages a key construction from~\cite{harmonicph}. In this way,  together with \cite{1parameterGrassmannian}, our  work  serves as a bridge between the ideas presented in~\cite{edit} and~\cite{harmonicph}.

\paragraph{Generalized persistence diagrams.}
 Via M\"obius inversion,  in \cite{Patel2018} Patel introduced \emph{generalized}  1-parameter persistence modules. A generalized 1-parameter persistence module is a functor  $H : \lp \to \mathcal{A}$ 
where $\mathcal{A}$ is an abelian group and a certain constructibility/finiteness condition is satisfied by $H$. In this generalization, M\"obius inversion is applied to the rank invariant~\cite[Section 7]{Patel2018} of the generalized persistence module. 

Patel's work motivated the use of M\"obius inversion in settings where the underlying poset is  more general than a linear poset (such as in multiparameter persistence, or persistence over general posets) leading to the application of M\"obius inversion to different invariants such as the generalized rank invariant~\cite{Kim2021,botnan2022signed,dynamic-kim2023,kim2023persistence}, birth-death function~\cite{edit,fasy2022persistent}, kernel function~\cite{gal-conn}, meta-rank~\cite{metarank-clause}, graded rank function~\cite{betthauser2019graded}, cup-product related invariants~\cite{cup-persistent}, etc. See also~\cite{krishnan-ogle-invert}. As mentioned above, in contrast with the case of 1-parameter filtrations (or persistence modules), when applying M\"obius inversion for defining a generalized persistence diagram for filtrations (or persistence modules) defined over a general poset, a drawback arises: the resulting (generalized) persistence diagram might attain \emph{negative} values ~\cite{Kim2021,botnan2022signed, edit, gal-conn,kim2023persistence} thus  complicating its interpretation as multiplicities.


\paragraph{Beyond Persistence Modules: Invariants of Filtered Chain Complexes.}


The standard persistent homology pipeline applies the homology functor to a simplicial filtration to produce a persistence module. However, this procedure inherently involves an intermediate step: the \emph{filtered chain complex} (FCC).  Our work is part of a broader effort to develop invariants for filtered chain complexes. Indeed, rather than focusing solely on the resulting persistence module, one can analyze this intermediate structure to extract richer invariants (see Figure~\ref{fig: pipeline with fccs}).
Indeed, various invariants that  arise from a simplicial filtration at the chain level have been investigated, such as: 
\begin{enumerate}
    \item Algebraic decomposition of FCCs and verbose barcodes~\cite{usher, ephemral,landi2021invariants,chacholski2023decomposing},
    \item Birth-death functions/spaces~\cite{saecular, edit, morozov-patel}, 
    \item Persistent Laplacians~\cite{lieutier-pdf, persSpecGraph, persLap}.
\end{enumerate}

Below we expand upon concepts outlined in items 2 and 3.

\medskip

In \cite{edit} the authors consider the notion of \emph{birth-space} spaces and functions, with the latter arising as the dimension of the former. In that article, the authors consider a notion of persistence diagram induced by birth-death functions through M\"obius inversion (using the product order on the poset of intervals; see.~\cref{parag: poset of int}). In contrast to our work, neither their work nor  \cite{saecular, morozov-patel} explore the possibility of utilizing more information than just the dimension of  birth-death spaces. This possibility is interesting since one expects  birth-death spaces to carry information not only about birth and death parameter value, but also about cycle representatives of all persistent homological features. 

The persistent Laplacian construction, which relies on a fixed inner product on the chain groups, \cite{lieutier-pdf, persSpecGraph,persLap}, through its 0-eigenspace (i.e. its kernel), provides us with an interesting inner product space assignment to every interval in the underlying linear poset of a filtration. In related settings, the presence of inner products offers an enriched and more versatile framework to operate within (see for example \cite{ogle-sultan-inner}), allowing the selection of canonical choices of representatives, as demonstrated in~\cite{harmonicph}. Nevertheless, no successful attempts have been made in order to  identify a notion of persistence diagram for  persistent Laplacians which is able to capture more information than what is available through the dimension of its kernel (which reduces to the usual persistent Betti number of the filtration; see \cite[Theorem 2.7]{persLap}). One such attempt is described in the companion paper \cite{1parameterGrassmannian}.



\paragraph{Edit Distance Stability of Persistence Diagrams.}
The notions of edit distance between filtrations and edit distance between persistence diagrams were introduced by  McCleary and Patel in \cite{edit}, where they also established the edit distance stability of persistence diagrams. Their stability result leverages the fact that the \emph{algebraic}\footnote{We use the term `algebraic' to stress the fact that their notion relies on a step consisting of taking  the group completion.} M\"obius inversion is a functor when the  order on segments is crucially chosen to be the product order. Similarly, our stability results (\cref{thm: goi functor and stable}) also make use of a notion of functoriality for monoidal M\"obius inverses. To obtain persistence diagrams, McCleary and Patel apply (algebraic) M\"obius inversion to the function given as the dimension of the birth-death spaces. In contrast, we apply Orthogonal Inversion to the function assigning to an interval the actual birth-death vector space associated to that interval (endowed with an inner product structure). In particular, we prove that our stability result improves upon that of \cite{edit}; see   \cref{thm: edit between signed is lower bound for edit between grassmannian} and~\cref{rmk:kracht}.

\paragraph{Other aspects.}
Additional aspects related to Grassmanian persistence diagrams for 1-parameter filtrations are explored in the companion paper \cite{1parameterGrassmannian}. These include: the connection to standard persistence diagrams (see \cref{fig:3d-gpd}), connections to the persistent Laplacian, their distinguishing power for ultrametric and tree-like spaces, canonical representatives for bars  as well as algorithmic aspects.  The algorithmic aspects  of the $2$-parameter case is discussed in \cref{rem: computability of 2-parameter}, where we highlight how the computational aspects extend from the $1$-parameter setting.

\subsection*{Acknowledgements}

This work was partially supported by NSF DMS \#2301359, NSF CCF \#2310412, NSF CCF \#2112665, NSF CCF \#2217058, NSF RI \#1901360. We would like to thank Woojin Kim and Alex McCleary for their suggestions and feedback which helped us improve the exposition of the material. Additionally, we would like to express our gratitude to Sanjeevi Krishnan for suggesting that we explore orthomodular lattices as a means of generalizing our approach. We also thank Amit Patel for encouraging us to generalize a previous version of our construction.

\section{Preliminaries}\label{sec: prelim}

In this section, we recall the background concepts that will be used in the upcoming sections of the paper.

\paragraph{Grassmannian.} Let $V$ be a finite-dimensional vector space. The set of all $d$-dimensional linear subspaces of $V$ is a well-studied topological space, called the~\emph{$d$-Grassmannian of $V$} and denoted $\gr(d,V)$. As we will be working with linear subspaces with varying dimensions, we consider the disjoint union
\[
\gr(V) := \coprod_{0\leq d\leq \dim(V)} \gr(d,V)
\]
\noindent which is called the~\emph{Grassmannian of $V$}. Note that $\gr(V)$ is closed under the sum of subspaces. That is, for two linear subspaces $W_1, W_2 \subseteq V$, their sum $W_1 + W_2 := \{w_1 + w_2 \mid w_1 \in W_1, w_2 \in W_2 \} \subseteq V$ is also a linear subspace. Thus, the triple $(\gr(V), +, \{ 0\})$ forms a commutative monoid. Moreover, $(\gr(V), \subseteq)$ is a poset. 
\nomenclature[01]{$\gr(V)$}{Grassmanian of the vector space $V$}

\begin{framed}
In this paper, we only utilize the monoidal structure and  the natural partial order on $\gr(V)$, without referring to its topology.
\end{framed}

\begin{definition}[Transversity]
    Let $V$ be any finite-dimensional vector space.
    A family $\{ W_i\}_{i=1}^m$ of subspaces of $V$ is said to be a \emph{transversal family} if 
    \[
    \dim \left ( \sum_{i=1}^n W_i \right ) = \sum_{i=1}^n \dim (W_i).
    \]

\paragraph{Classical Algebraic Structures.} We denote by $\Vec$ the category of finite dimensional vector spaces over $\R$ and by $\mathrm{Ch}(\Vec)$ the category of chain complexes in $\Vec$.

\nomenclature[02]{$\Vec$}{Category of finite dimensional vector spaces over $\R$}

\nomenclature[02]{$\mathrm{Ch}(\Vec)$}{Category of chain complexes in Vec.}

\medskip
The main object of study in this paper is that of a persistence modules over posets.
\begin{definition}[Persistence modules]
    Let $P$ be any finite poset. A functor $M: P \to \Vec$ is called a \emph{persistence module}.
\end{definition}

\end{definition}

\subsection{Edit Distance} 
The notion of \emph{edit distance} employed in this paper is closely related to the one considered in~\cite{edit}, where the authors introduce their version as the categorification of the Reeb graph edit distance discussed in~\cite{DiFabio2012, DiFabio2016, bauer-edit-2016, Bauer2020}. Let $\mathcal{C}$ be a category and assume that for every morphism $f : A \to B$ in $\mathcal{C}$, there is a cost $c_\mathcal{C} (f) \in \R^{\geq 0}$ associated to it. For two objects $A$ and $B$ in this category, a~\emph{path}, $\mathcal{P}$, is a finite sequence of morphisms
\[
\mathcal{P} : \; \; A \xleftrightarrow{\makebox[1cm]{$f_1$}} D_1 \xleftrightarrow{\makebox[1cm]{$f_2$}} \cdots \xleftrightarrow{\makebox[1cm]{$f_{k-1}$}} D_{k-1} \xleftrightarrow{\makebox[1cm]{$f_k$}} B
\]
where $D_i$s are objects in $\mathcal{C}$ and $\leftrightarrow$ indicates a morphism in either direction. The~\emph{cost of a path} $\mathcal{P}$, denoted $c_\mathcal{C} (\mathcal{P})$, is the sum of the cost of all morphisms in the path.
\[
c_\mathcal{C} (\mathcal{P}) := \sum_{i=1}^k c_\mathcal{C} (f_i).
\]

\begin{definition}[Edit distance]\label{defn: edit dist}
    The~\emph{edit distance} $d_{\mathcal{C}}^E (A,B)$ between two objects $A$ and $B$ in $\mathcal{C}$ is the infimum, over all paths between $A$ and $B$, of the cost of such paths.
    \[
    d_\mathcal{C}^E(A,B) := \inf_{\mathcal{P}} c_\mathcal{C}(\mathcal{P}).
    \]
\end{definition}
\nomenclature[02]{$d_{\mathcal{C}}^E(A,B)$}{Edit distance between objects $A$ and $B$ in a category $\mathcal{C}$}

While the definition of edit distance may appear rather abstract, it is worth noticing that in~\cite[Theorem 9.1]{edit}, the authors constructed a category of persistence diagrams  and established that the edit distance within this context is bi-Lipschitz equivalent to the well-known bottleneck distance~\cite{cohen-steiner2007}.

\subsection{Posets and (Algebraic) M\"obius Inversion}

The notion of M\"obius inversion can regarded as a combinatorial analog to the notion of derivative. It is one of the most fundamental tools in the algebraic-combinatorial approach to the study of persistence modules developed in~\cite{Patel2018}. Here, we recall the essential background on M\"obius inversion required for this paper and direct readers to~\cite{rota64} for a more comprehensive overview.

\paragraph{Poset(s) of Segments.}\label{parag: poset of int} 
Let $(P, \leq)$ be a poset. For any $a\leq b$, we refer to pair $(a,b) \in P\times P$ as a \emph{bounded segment}. For every $a\in P$, we introduce a distinguished pair, denoted $(a,\infty)$, and we refer to these pairs as \emph{unbounded segments}. Unbounded segments, as defined here, enable us to conceptualize the absence of a maximum element in a segment, thus allowing the identification of cycles in a filtration that never die (or, die at infinity). We denote the collection of bounded and unbounded segments in $P$ by $\Seg(P)$, which we refer to as the \emph{set of segments} of $P$. We denote by $\diag(P):= \{ (a,a) \in \Seg(P) \mid a \in P \}$ the~\emph{diagonal} of $\Seg(P)$. 
\nomenclature[03]{$\Seg(P)$}{The set of segments of a poset $P$}
\nomenclature[04]{$\diag(P)$}{The diagonal of $\Seg(P)$}

\begin{framed}
    In this paper, we will assume that all posets are finite unless otherwise specified.
\end{framed}

The \emph{product order} on $\Seg (P)$, $\prodord$, is given by the restriction of the product order on $P\times P$ to $\Seg(P)$. More precisely,  
\[
    (b_1,d_1) \prodord (b_2,d_2) \iff b_1\leq b_2 \text{ and } d_1\leq d_2 
\]
where we assume $p < \infty$ for all $p \in P$. We will also use another order on $\Seg (P)$, called the \emph{reverse inclusion order}. The reverse inclusion order on $\Seg (P)$, denoted $\incord$, is given by
\[
    (b_1,d_1) \incord (b_2,d_2) \iff b_1\leq b_2 \text{ and } d_1\geq d_2. 
\]


We denote by
\begin{itemize}
\item $\overline{P}^\times := (\Seg (P), \prodord)$ the poset of segments with the product order;
\item  $\overline{P}^\supseteq := (\Seg(P), \incord)$ the poset of non-diagonal segments with the reverse inclusion order. 
\end{itemize}
\nomenclature[05]{$\prodord$}{The product order on the segments of a poset}
\nomenclature[06]{$\incord$}{The reserve inclusion order on the segments of a poset}
\nomenclature[07]{$\overline{P}^\times$}{The poset of segments of $P$ with the product order}
\nomenclature[08]{$\overline{P}^\supseteq$}{The poset of segments of $P$ with the reverse inclusion order}

\begin{remark}[About the choice of order on $\Seg(P)$]
    The reverse inclusion order $\incord$ is the one that has been traditionally used in the applied algebraic community, in particular it is the one  used to induce signed barcodes from the rank invariant in \cite{botnan2022signed}. The product order $\prodord$ has been advocated by McCleary and Patel due to its favorable interaction with M\"obius inversion of birth-death functions; see \cite{edit,morozov-patel}. In this paper, we will will mostly utilize the product order on $\Seg(P).$
\end{remark}

Notice that if $f: P \to Q$ is an order-preserving map between two posets $P$ and $Q$, then $f$ induces an order-preserving map between the posets of segments $\overline{P}^\times$ and $\overline{Q}^\times$. We denote by $\overline{f} : \overline{P}^\times \to \overline{Q}^\times$ the map induced by $f$ that acts to segments component-wisely. That is, $\overline{f} ((b,d)) := (f(b), f(d))$ and $\overline{f} ((b,\infty)) := (f(b), \infty)$ for every $b\leq d \in P$. Note that with the convention described above, this condition means that $\overline{f}$ preserve the type of segments, i.e., it maps (un)bounded segments to (un)bounded segments. 

\paragraph{Metric Posets.} 
A finite~\emph{(extended) metric poset} is a pair $(P, d_P)$ where $P$ is a finite poset and $d_P : P \times P \to \R \cup \{ \infty \}$ is an (extended) metric
 such that for every $p_1\leq p_2 \leq p_3 \in P$, $d_P (p_1,p_2) \leq d_P (p_1,p_3)$ and $d_P(p_2,p_3) \leq d_P (p_1,p_3)$. 
A \emph{morphism of finite metric posets} $\alpha: (P, d_P) \to (Q, d_Q)$ is an order-preserving map $\alpha: P \to Q$. The \emph{distortion} of a morphism $\alpha : (P, d_P) \to (Q,d_Q)$, denoted $\dis(\alpha)$, is
\[
\dis(\alpha) := \max_{p_1,p_2 \in P} \big|d_P (p_1,p_2) - d_Q (\alpha(p_1), \alpha(p_2))\big|.
\]

\nomenclature[09]{$\dis(\alpha)$}{Distortion of the morphism $\alpha$ between two metric posets}

For every finite metric poset $(P, d_P)$, its poset of segments $\overline{P}^\times$ is also a metric poset with 
$$d_{\overline{P}^\times} ((b_1,d_1), (b_2,d_2)) := \max \{d_P(b_1,b_2), d_P(d_1,d_2) \}.$$ 
A morphism of finite metric posets $\alpha : (P, d_P) \to (Q, d_Q)$ induces a morphism of finite metric posets $\overline{\alpha} : \left(\overline{P}^\times ,d_{\overline{P}^\times}\right) \to \left(\overline{Q}^\times, d_{\overline{Q}^\times}\right)$ via $\overline{\alpha} ((b,d)) = (\alpha(b), \alpha(d))$, with $\dis(\alpha) = \dis(\overline{\alpha})$; see  \cite[Proposition 3.4]{edit}.

\paragraph{M\"obius Inversion.} \label{pg:mi} Let $P$ be a poset and $\mathcal{M}$ be a commutative monoid. Let $ \kappa (\mathcal{M})$ be the Grothendieck group completion of $\mathcal{M}$ (see~\cref{appendix:details}). The abelian group $\kappa (\mathcal{M})$ consists of equivalence classes of pairs $(m,n) \in \mathcal{M} \times \mathcal{M}$ under the equivalence relation described in~\cref{appendix:details}. Let $\varphi_\mathcal{M} : \mathcal{M} \to \kappa (\mathcal{M})$ denote the canonical morphism that maps $m\in \mathcal{M}$ to the equivalence class of $(m,0)$ in $\kappa (\mathcal{M})$. Let $m :P \to \mathcal{M}$ be a function. Whenever it exists, we define the~\emph{algebraic M\"obius inverse} of $m$ to be the unique function $\partial_P (m) : P \to \kappa (\mathcal{M}) $ satisfying, for all $p\in P$,
    \[
    \sum_{p'\leq p} \partial_P (m) (p') = \varphi_\mathcal{M} (m(p)).
    \]
 
\nomenclature[091]{$\kappa(\cdot)$}{Grothendieck group completion of a commutative monoid}

Let $\mathcal{G}$ be an abelian group. Then, the group completion of $\mathcal{G}$ is isomorphic to $\mathcal{G}$. That is, $\mathcal{G}$ and $\kappa (\mathcal{G})$ can be identified and the canonical map $\varphi_\mathcal{G}$ can be taken as the identity map. In this case, if $g : P \to \mathcal{G}$ is a function, then its algebraic M\"obius inverse, whenever it exists, is the \emph{unique} function $\partial_P g : P \to \mathcal{G}$ satisfying, for all $p\in P$,
    \begin{equation}\label{eqn: defining property of mi}
        \sum_{p'\leq p} \partial_P (g) (p') = g(p).
    \end{equation}

\nomenclature[10]{$\partial_P(\cdot)$}{M\"obius inverse of a function defined on a poset $P$}

\begin{remark}[Computing M\"obius inverses]\label{rem: computing mob inverse}
    Observe that for a function $g: P \to \mathcal{G}$, where $P$ is a finite poset and $\mathcal{G}$ is any abelian group, the M\"obius inverse $\partial_P (g) : P \to \mathcal{G}$ can be computed inductively along the poset structure by using~\cref{eqn: defining property of mi}. Explicitly:
    \begin{enumerate}
        \item \underline{Base case:} For every minimal element in $p_{\min}\in P$, we must have $\partial_P (g)(p_{\min}) = g(p_{\min})$ since there are no $p' < p_{\min}$.
        \item \underline{Inductive step:} For any non-minimal $p\in P$, we compute 
        \[
        \partial_P (g) (p) = g(p) - \sum_{p' < p} \partial_P (g)(p'),
        \]
        assuming that the values $\partial_P (g)(p')$ for $p' < p$ have already been determined.
    \end{enumerate}
    
    This inductive procedure applies to any finite poset and provides a general method for computing M\"obius inverses. More explicit inversion formulas can be derived in certain structured settings—such as the poset of segments of a finite linear poset or of a $2$-D grid. For the case of segments of a linear poset, we provide such a formula in~\cref{prop: algebraic mobius inversion formulas}; for segments of a $2$-D grid, explicit formulas appear in~\cite[Lemma 2.7 and Corollary 2.8]{morozov-patel} and~\cite[Equation 5.1]{botnan2022signed}.
\end{remark}

\begin{proposition}\label{prop: algebraic mobius inversion formulas}
    Let $\lp = \{ \ell_1 < \cdots < \ell_n \}$ be a finite linearly ordered set, let $m : \overline{\lp}^\times \to \mathcal{G}$ be any function. Then, the algebraic M\"obius inverse of $m$, $\partial_{\overline{\lp}^\times}(m) : \overline{\lp}^\times \to \mathcal{G}$, is given by
\begin{align}
    \partial_{\overline{\lp}^\times}(m) ((\ell_i, \ell_j)) &= m((\ell_i, \ell_{j})) - m((\ell_i, \ell_{j-1})) + m((\ell_{i-1}, \ell_{j-1})) - m((\ell_{i-1}, \ell_{j})), \label{eqn: times algebraic mob inv1} \\
    \partial_{\overline{\lp}^\times}(m) ((\ell_i, \infty)) &= m((\ell_i, \infty)) - m((\ell_i, \ell_{n})) + m((\ell_{i-1}, \ell_{n})) - m((\ell_{i-1}, \infty)), \label{eqn: times algebraic mob inv2} \\
    \partial_{\overline{\lp}^\times}(m) ((\ell_i, \ell_i)) &= m ((\ell_i, \ell_i)) - m ((\ell_{i-1}, \ell_{i})), \label{eqn: times algebraic mob inv3} 
\end{align}
for $1\leq i  < j \leq n$, where we follow the convention that the expressions of the form $m((\ell_0, \ell_j))$ and $m((\ell_0 , \infty))$ are assumed to be $0$. 
\end{proposition}

\begin{remark}\label{remark: convention edge cases}
    Breaking down the conventions of the proposition above, we have:
    \begin{align*}
        \partial_{\overline{\lp}^\times}(m) ((\ell_1, \ell_1)) &= m((\ell_1, \ell_{1})), \\
        \partial_{\overline{\lp}^\times}(m) ((\ell_1, \infty)) &= m((\ell_1, \infty)) - m((\ell_1 , \ell_n)), \\
        \partial_{\overline{\lp}^\times}(m) ((\ell_1, \ell_j)) &= m((\ell_1, \ell_{j})) - m((\ell_1, \ell_{j-1})) \text{ for } 1 < j \leq n.
    \end{align*}
\end{remark} 

\subsection{Galois Connections}

\begin{definition}[Galois connections]\label{defn: galois connections}
    Let $P$ and $Q$ be any two posets.\footnote{Finiteness is not required.} A pair, $(\ladj{f}, \radj{f})$, of order-preserving maps, $\ladj{f} : P \to Q$ and $\radj{f} : Q \to P$, is called a~\emph{Galois connection} if they satisfy
    \[
    \ladj{f}(p) \leq q \iff p \leq \radj{f} (q)
    \]
    for every $p\in P$, $q\in Q$. We refer to $\ladj{f}$ as the~\emph{left adjoint} and refer to $\radj{f}$ as the~\emph{right adjoint}. We will also use the notation $\ladj{f} : P \leftrightarrows Q : \radj{f}$ to denote a Galois connection.
\end{definition}

The left and right adjoints of a Galois connection can be expressed in terms of each
other as follows:
\begin{align*}
    \ladj{f} (p) &= \min \{ q\in Q \mid p \leq \radj{f}(q) \} \\
    \radj{f} (q) &= \max \{ p\in P \mid \ladj{f}(p) \leq q \}.
\end{align*}

\begin{example}
    Consider the inclusion $\iota: \Z \hookrightarrow \R$, the ceiling function $\lceil \cdot \rceil : \R \to \Z$, and the floor function $\lfloor \cdot \rfloor : \R \to \Z$. The pairs $\left(\lceil \cdot \rceil, \iota\right)$ and $\left(\iota, \lfloor \cdot \rfloor\right)$ are Galois connections.
\end{example}

\begin{remark}
    Notice that the composition of Galois connections $\ladj{f} : P \leftrightarrows Q : \radj{f}$ and $\ladj{g} : Q \leftrightarrows R : \radj{g}$ is also a Galois connection $$\ladj{g} \circ \ladj{f} : P \leftrightarrows R : \radj{f} \circ \radj{g}.$$ Also, note that a Galois connection $\ladj{f} : P \leftrightarrows Q : \radj{f}$ induces a Galois connection on the poset of segments $\overline{\ladj{f}} : \overline{P}^\times \leftrightarrows \overline{Q}^\times :\overline{\radj{f}}$. These properties of Galois connections will later be utilized in defining morphisms in certain categories that we will introduce in and~\cref{sec:generalized orthogonal inversion}.
\end{remark}

\begin{definition}[Pushforward and pullback]
Let $f: P \to Q$ be any order-preserving map between two posets, and let $m : P \to \mathcal{G}$ be any function. The~\emph{pushforward} of $m$ along $f$ is the function $f_\sharp m : Q \to \mathcal{G}$ given by
$$f_\sharp m (q) := \sum_{p \in f^{-1}(q)} m(p).$$
Let $h : Q \to \mathcal{G}$ be any function. The~\emph{pullback} of $h$ along $f$ is the function $f^\sharp h : P \to \mathcal{G}$ given by
$$(f^\sharp h )(p) := h( f(p)).$$
\end{definition}
\nomenclature[11]{$(\cdot)_\sharp$}{Pushforward along a map}
\nomenclature[12]{$(\cdot)^\sharp$}{Pullback along a map}

The following theorem, Rota's Galois Connection Theorem (RGCT), describes how M\"obius inversion behaves when a Galois connection exists between two posets.

\begin{theorem}[RGCT~{\cite[Theorem 3.1]{gal-conn}}]\label{thm:rgct}
    Let $P$ and $Q$ be finite posets and $(\ladj{f}, \radj{f})$ be a Galois connection. Then,
    \begin{equation}\label{eqn: rgct push vs pull}
        (\ladj{f})_\sharp \circ \partial_P = \partial_Q \circ (\radj{f})^\sharp.
    \end{equation}
\end{theorem}

\begin{figure}
    \centering
    \includegraphics[scale=15]{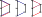}
    \caption{An illustration of RGCT.}
    \label{fig:RGCTdemo}
\end{figure}

\begin{example}\label{ex:gal conn demo}
    Let $P = \{ p_1 < p_2 < p_3 \}$ and $Q = \{ q_1 < q_2\}$ be two posets. Let
    \begin{align*}
        \ladj{f} : P &\to Q & \radj{f}: Q &\to P \\
                   p_1 &\mapsto q_1 & q_1 &\mapsto p_1\\
                   p_2 &\mapsto q_2 & q_2 &\mapsto p_3\\
                   p_3 &\mapsto q_2
    \end{align*}  
    Then, $(\ladj{f}, \radj{f})$ is a Galois connection. In~\cref{fig:RGCTdemo}, $\ladj{f}$ is depicted with the red arrows and $\radj{f}$ is depicted with the blue arrows. In the middle of~\cref{fig:RGCTdemo}, we illustrate two functions. First, $m : P \to \Z$ is a function on $P$, whose values are given by $m(p_1) =1$, $m(p_2) =2$ and $m(p_3) =5$. Second, $(\radj{f})^\sharp m : Q \to \Z$ is a function on $Q$, whose values are given by $(\radj{f})^\sharp m (q_1) := m(\radj{f}(q_1)) = m(p_1) =1$ and similarly $(\radj{f})^\sharp m (q_2):= m(\radj{f}(q_2)) = m(p_3) =5$. On the right of~\cref{fig:RGCTdemo}, we illustrate the M\"obius inverses of these functions, namely, $\partial_P (m)$ and $\partial_Q ((\radj{f})^\sharp m)$. Notice that the pushforwad of $\partial_P (m)$ along $\ladj{f}$ is equal to $\partial_Q ((\radj{f})^\sharp m)$. That is, $\partial_Q ((\radj{f})^\sharp m) = (\ladj{f})_\sharp (\partial_P (m))$ as stated in~\cref{thm:rgct}.
\end{example}

\subsection{Simplicial Complexes and Filtrations}

In this section, we introduce fundamental concepts and definitions, including simplicial complexes, filtrations, persistent Betti numbers, the concepts of birth and death of cycles, and birth-death spaces.

\paragraph{Simplicial Complexes and Chain Spaces.} An (abstract) \emph{finite simplicial complex} $K$ over a finite ordered vertex set $V$ is a non-empty collection of non-empty subsets of $V$ with the property that for every $\sigma \in K$, if $\tau\subseteq \sigma$, then $\tau \in K$. An element $\sigma \in K$ is called a \emph{$\dgr$-simplex} if the cardinality of $\sigma$ is $\dgr+1$. An \emph{oriented simplex}, denoted $[\sigma]$, is a simplex $\sigma \in K$ whose vertices are ordered. We always assume that ordering on simplices is inherited from the ordering on $V$. Let $\mathfrak{s}_\dgr^K$ denote the set of all oriented $\dgr$-simplices of $K$. 
\nomenclature[13]{$\mathfrak{s}_\dgr^K$}{set of all oriented $\dgr$-simplices of a simplicial complex $K$}

The \emph{$\dgr$-th chain space of $K$}, denoted $C_\dgr^K$, is the vector space over $\R$ with basis $\mathfrak{s}_\dgr^K$. Let $n_\dgr^K := |\mathfrak{s}_\dgr^K| = \dim_\R (C_\dgr^K)$. The \emph{$\dgr$-th boundary operator} $\partial_\dgr^K : C_\dgr^K \to C_{\dgr-1}^K$ is defined by
\[
    \partial_\dgr^K ([v_0,\ldots,v_\dgr]) := \sum_{i=0}^{\dgr}(-1)^i[v_0,\ldots,\hat{v}_i,\ldots,v_\dgr]
\]
for every oriented $\dgr$-simplex $[\sigma] = [v_0,\ldots,v_\dgr]\in \mathfrak{s}_\dgr^K$, where $[v_0,\ldots,\hat{v}_i,\ldots,v_\dgr]$ denotes the omission of the $i$-th vertex, and extended linearly to $C_\dgr^K$. We denote by $\Zfunc_\dgr(K)$ the space of $\dgr$-cycles of $K$, that is $$\Zfunc_\dgr(K) := \ker \left(\partial_\dgr^K\right),$$ and we denote by $\Bfunc_\dgr(K)$ the space of $\dgr$-boundaries of $K$, that is $$\Bfunc_\dgr(K):= \Ima \left(\partial_{\dgr+1}^K\right).$$ Additionally, we denote by $H_\dgr(K)$ the $\dgr$-th homology group of $K$, that is $$H_\dgr(K) := \frac{\Zfunc_\dgr(K)}{ \Bfunc_\dgr(K)}.$$
\nomenclature[14]{$C_\dgr^K$}{$\dgr$-th chain space of of a simplicial complex $K$}
\nomenclature[15]{$\partial_\dgr^K$}{$\dgr$-th boundary operator of a simplicial complex $K$}
\nomenclature[16]{$\Zfunc_\dgr(\cdot)$}{Space of $\dgr$-cycles}
\nomenclature[17]{$\Bfunc_\dgr(\cdot)$}{Space of $\dgr$-boundaries}
\nomenclature[18]{$H_\dgr(\cdot)$}{$\dgr$-th homology group}

For each integer $\dgr\geq 0$, we define an inner product, $\langle \cdot , \cdot \rangle_{C_\dgr^K}$, on $C_\dgr^K$ as follows:
\[
\langle [\sigma] , [\sigma'] \rangle_{C_\dgr^K} := \delta_{[\sigma], [\sigma']} \text{, for all } [\sigma],[\sigma'] \in \mathfrak{s}_\dgr^K,
\]
where $\delta_{\bullet, \bullet}$ is the Kronecker delta.
That is, we declare that $\mathfrak{s}_\dgr^K$ is an orthonormal basis for $C_\dgr^K$. We will refer to $\langle \cdot , \cdot \rangle_{C_\dgr^K}$ as the~\emph{standard} inner product on $C_\dgr^K$. We will omit the subscript from the notation $\langle \cdot , \cdot \rangle_{C_\dgr^K}$ when the context is clear. We denote by $\left(\partial_\dgr^K\right)^* : C_{\dgr-1}^K \to C_\dgr^K$ the adjoint of $\partial_\dgr^K$ with respect to the standard inner products  on $C_\dgr^K$ and $C_{\dgr-1}^K$.
\nomenclature[19]{$\langle \cdot , \cdot \rangle_{C_\dgr^K}$}{Standard inner product on $C_\dgr^K$}

\begin{framed}
    In this paper, we work with a fixed finite simplicial complex $K$ and primarily use real coefficients. However, some of our results, definitions, and remarks still hold when the coefficients are taken from any field. Specifically, whenever a result, definition, example, or remark involves the symbols $\goi$, or $\ominus$, real coefficients are required. In all other cases, the results hold with coefficients from any field.
\end{framed}

\paragraph{Simplicial Filtrations.}\label{parag: simplicial filtration} For a finite simplicial complex $K$, let $\subcx(K)$ denote the poset of subcomplexes of $K$, ordered by inclusion. A \emph{simplicial filtration of $K$} is an order-preserving map $\Ffunc : P \to \subcx (K)$, where $P$ is a finite poset. 
A \emph{$1$-parameter} filtration of $K$ is a filtration $\Ffunc : \lp \to \subcx(K)$ where $\lp = \{\ell_1<\cdots<\ell_n \}$ is a finite linearly ordered set and $\Ffunc(\ell_n) = K$. When $\Ffunc : \{ \ell_1<\cdots<\ell_n\} \to \subcx (K)$ is a $1$-parameter filtration, we use the notation $K_i$ for the simplicial complex $\Ffunc(\ell_i)$ and succinctly write $\Ffunc = \{K_i \}_{i=1}^n$ to denote the simplicial filtration. Note that $K_n = K$. 
\nomenclature[20]{$\subcx(\cdot)$}{Poset of subcomplexes of a simplicial complex, ordered by inclusion}

\begin{definition}[Persistent Betti numbers / rank invariant \cite{Edelsbrunner2002}]\label{defn: persistent betti rank inv}
    Let $\Ffunc : P \to \subcx (K)$ be a filtration. For $(p,p')\in \Seg(P)$, let $\iota_\dgr^{p,p'} : H_\dgr (\Ffunc(p)) \to H_\dgr (\Ffunc(p'))$ denote the homomorpshim induced by the inclusion $\Ffunc(p) \hookrightarrow \Ffunc(p')$. We define the \emph{$\dgr$-th persistent Betti number} for the segment $(p,p')$ as
    \[
    \beta_\dgr^\Ffunc((p,p')) := \rank \left(\iota_\dgr^{p,p'}\right).
    \]
    As is customary in applied algebraic topology, we also use the term \emph{rank invariant} to refer to persistent Betti numbers. 
\end{definition}
\nomenclature[21]{$\beta_\dgr^\Ffunc((\cdot, \cdot))$}{$\dgr$-th persistent Betti numbers}


Let $\Ffunc = \{K_i \}_{i=1}^n$ be a filtration of $K$. Observe that, for any dimension $\dgr\geq 0$, the inclusion of simplicial complexes $K_i \subseteq K_j$, for $i \leq j$, induces canonical inclusions on the cycle and boundary spaces. In particular, for any $i = 1,\ldots,n$, the $\dgr$-th cycle and boundary spaces of $K_i$ can be identified with subspaces of $C_\dgr^{K_n} = C_\dgr^K$:

\begin{center}
    \begin{tikzcd}
\Zfunc_\dgr(K_i) \arrow[r, hook]                 & \Zfunc_\dgr(K_j) \arrow[r, hook]                 & \Zfunc_\dgr(K) \arrow[r, hook] & C_\dgr^K \\
\Bfunc_\dgr(K_i) \arrow[r, hook] \arrow[u, hook] & \Bfunc_\dgr(K_j) \arrow[r, hook] \arrow[u, hook] & \Bfunc_\dgr(K) \arrow[u, hook] &      
    \end{tikzcd}
\end{center}

\begin{definition}[Birth-death spaces]\label{defn: bd space}
    Let $\Ffunc : P \to \subcx (K)$ be a filtration. For any degree $\dgr\geq 0$, the $\dgr$-th~\emph{birth-death spaces} associated to $\Ffunc$ is defined as the function $\ZB_\dgr^\Ffunc : \overline{P}^\times\to \gr\left(C_\dgr^K\right)$ given by
    \begin{align*}
        \ZB_\dgr^\Ffunc ((b, d)) &:= \Zfunc_\dgr \big(\Ffunc(b)\big)\cap \Bfunc_\dgr \big(\Ffunc(d)\big), \\
        \ZB_\dgr^\Ffunc ((b,\infty)) &:= \Zfunc_\dgr \big(\Ffunc(b)\big).
    \end{align*} 

\end{definition}
\nomenclature[22]{$\ZB_\dgr^\Ffunc$}{$\dgr$-th birth-death spaces associated to a filtration $\Ffunc$}

Informally, when $b\leq d$, for a cycle $z$ to be in the birth-death space $\ZB_\dgr^\Ffunc((b,d))$ means that $z$ becomes ``alive" at or before $b$ and that it ``dies''  (i.e. it becomes a boundary) at or before $d$.

\begin{remark}
The classical definition of persistence diagrams, as in~\cite{cohen-steiner2007}, utilizes persistent Betti numbers. Birth-death spaces were introduced in~\cite{edit, saecular} as an alternative way to define persistence diagrams. For a $1$-parameter filtration, classical persistence diagrams and persistence diagrams obtained through utilizing the dimension of birth-death spaces coincide as shown in~\cite[Section 9.1]{edit}. McCleary and Patel showed that the use of birth-death spaces for defining persistence diagrams enables us to organize the persistent homology pipeline in a functorial way~\cite{edit}. Additionally, this functoriality leads to the edit distance stability of persistence diagrams as in~\cite[Theorem 8.4]{edit}.

\end{remark}

We now discuss how to determine from $\ZB_\dgr^\Ffunc$ the exact number of independent $\dgr$-cycles that ``\emph{are born at precisely $b$ and die precisely at $d$}". This will pave the way towards the precise notion of birth and death of cycles which we give in \Cref{def:bd-cycles}. Note that in TDA, the death time is typically used in reference to  homology classes as opposed to cycles and 
    the death time of a homology class refers to the first time when the class merges with an ``older" one, following the ``elder rule" \cite{edelsbrunner2010computational,curry2018fiber}.

\begin{remark}[Number of independent cycles that ``are born at $b$ and die at $d$"]\label{remark: number of cycles that are born at b and die at d}
Let $\Ffunc : P \to \subcx(K)$ be a filtration.  First, we observe that the definition of birth-death spaces directly implies that $\dim \left( \ZB_\dgr^\Ffunc ((b,d)) \right)$ equals the number of linearly independent cycles that 

\begin{center}``(are born at or before $b$) and (die at or before $d$)."
\end{center}

Towards our goal, we need to refine the information contained in the birth-death space. This will be accomplished by computing the quotient of $\ZB_\dgr^\Ffunc((b,d))$ by a certain subspace $W_\dgr^\Ffunc((b,d))$ that we now describe.  Consider the subspaces
\begin{align*}
    W_{]b,d)} &:= \sum_{\substack{a\leq b \\ c<d}} \ZB_\dgr^\Ffunc ((a,c)) \subseteq C_\dgr^K,  \\
    W_{)b, d]} &:= \sum_{\substack{a<b \\ c\leq d}} \ZB_\dgr^\Ffunc ((a,c)) \subseteq C_\dgr^K.    
\end{align*}
Note that $\dim \left( W_{]b, d)} \right)$  is exactly the number of linearly independent cycles that
\begin{center}
    ``(are born at or before $b$) and (die strictly before $d$).''
\end{center}
Equivalently, this can also be expressed as the number of linearly independent cycles that
\begin{center}
    ``((are born strictly before $b$) and (die strictly before $d$)) \\ or \\ ((are born at $b$) and (die strictly before $d$)).''
\end{center}
Similarly, $\dim \left( W_{)b, d]} \right)$ is exactly the number of linearly independent cycles that
\begin{center}
    ``((are born strictly before $b$) and (die strictly before $d$)) \\ or \\ ((are born strictly before $b$) and (die at $d$)).''    
\end{center}
Let 
\[
W_\dgr^\Ffunc((b,d)) := W_{]b, d)} + W_{)b, d]} = \sum_{(a,c)<_\times (b,d)} \ZB_\dgr^\Ffunc ((a,c)).
\]
Combining the two interpretations above, we  see that $\dim \left( W_{(b,d)} \right)$ is exactly the number of linearly independent cycles that 
\begin{center}
    ``((are born strictly before $b$) and (die strictly before $d$)) \\
    or \\
    ((are born at $b$) and (die strictly before $d$) \\
    or \\
    ((are born strictly before $b$) and (die at $d$)).''
\end{center}
Hence, with the above interpretations, it follows that
\begin{align*}
    \dim\left( \frac{\ZB_\dgr^\Ffunc ((b,d))}{W_\dgr^\Ffunc((b,d))} \right) &= \dim \left( \ZB_\dgr^\Ffunc ((b,d)) \right) - \dim \left( W_\dgr^\Ffunc((b,d)) \right) \\
    &=\dim \left( \ZB_\dgr^\Ffunc ((b,d)) \right) - \dim \left( \sum_{(a,c)<_\times (b,d)} \ZB_\dgr^\Ffunc ((a,c)) \right)
\end{align*}
equals the number of linearly independent cycles that 
\begin{center}
``are born precisely at $b$ and die precisely at $d$". 
\end{center}

\end{remark}

The considerations in \cref{remark: number of cycles that are born at b and die at d} lead to the following definition of \emph{lifespan} of a cycle.

\begin{definition}[Lifespan of cycles / ephemeral cycles]\label{def:bd-cycles} 
    Let $P$ be any finite poset and let $\Ffunc : P \to \subcx(K)$ be a filtration. Let $(b,d) \in \Seg(P)$. We say that a nonzero cycle $z \in C_\dgr^K$ has \emph{lifespan $(b,d)$} if the following two conditions are met:

    \begin{itemize}
        \item $z \in \ZB_\dgr^\Ffunc((b,d))$, and
        \item $z \notin W_\dgr^\Ffunc((b,d)) = \sum_{(a,c)<_\times (b,d)} \ZB_\dgr^\Ffunc((a,c))$. 
    \end{itemize}

    When a cycle $z$ is born at $b$ and dies at $b$ (i.e., $b=d$), we say that $z$ is an \emph{ephemeral} cycle.
    
\end{definition}

\begin{remark}
    Note that when $P$ is a linear poset, the number of linearly independent $\dgr$-cyles that are born at $b$ and die at $d$ coincides with the multiplicity of the interval $(b,d)$ in the degree-$\dgr$ persistence diagram of the filtration $\Ffunc : P \to \subcx(K)$. To be specific, when $P =\lp = \{ \ell_1<\ell_2<\cdots<\ell_n\}$, we have that the number
    \[
    \dim \left( \ZB_\dgr^\Ffunc ((\ell_i,\ell_j)) \right) - \dim \left( \sum_{(\ell_k,\ell_l)<_\times (\ell_i,\ell_j)} \ZB_\dgr^\Ffunc ((\ell_k,\ell_l)) \right) 
    \]
    boils down to the following expression
    \[
    \dim \left( \ZB_\dgr^\Ffunc ((\ell_i,\ell_j)) \right) - \dim \left( \ZB_\dgr^\Ffunc ((\ell_{i-1},\ell_j)) \right) + \dim \left( \ZB_\dgr^\Ffunc ((\ell_{i-1},\ell_{j-1})) \right) - \dim \left( \ZB_\dgr^\Ffunc ((\ell_i,\ell_{j-1})) \right),
    \]
    which, by~\cref{prop: algebraic mobius inversion formulas}, is the M\"obius inverse of the birth-death function evaluated at $(b,d)$. And, as shown in~\cite[Section 9.1]{edit}, this number is exactly the multiplicity of the interval $(b,d)$ in the persistence diagram of the filtration $\Ffunc$.
\end{remark}

\begin{remark}(Non-uniqueness of lifespans in multiparameter filtrations)
    Note that, while a cycle can only have a unique lifespan in a 1-parameter filtration, it can exhibit multiple lifespans in a multiparameter filtration, as illustrated in~\cref{fig: 3x3 filtration} and~\cref{ex: multiparameter gpd}.

\end{remark}

\subsection{Monoidal M\"obius Inverses}\label{sec: monoidal mobious inversion}

The ``algebraic" M\"obius inverse of a function $m : P \to \mathcal{M}$, where $\mathcal{M}$ is a commutative monoid, involves the group completion of $\mathcal{M}$, denoted $\kappa(\mathcal{M})$. This is required in order to ``make sense'' of the minus operations that may appear in M\"obius inversion formulas such as the one in~\cref{prop: algebraic mobius inversion formulas}. However, the group completion of a commutative monoid could be the trivial group, yielding a trivial algebraic M\"obius inverse. In particular, for any vector space $V$, the group completion of $\gr(V)$ is the trivial group; see~\cref{appendix:details}. In this case, the algebraic M\"obius inverse of any map $m : P \to \gr(V)$ is the trivial map $ \partial_P (m) : P \to \{ 0 \} = \kappa (\gr(V))$. This suggests considering a notion of M\"obius inverse that does not involve group completion.

\begin{definition}[Monoidal M\"obius inverses]\label{defn: monoidal mobius inversion}
    Let $\mathcal{M}$ be a commutative monoid. Let $m : P \to \mathcal{M}$ be a function, then a function $m' : P \to \mathcal{M}$ is called a~\emph{monoidal M\"obius inverse} of $m$ if it satisfies 
    \[
    \sum_{p'\leq p} m'(p') = m(p)
    \]
    for all $p \in P$. We denote by $\mmi{P}{m}$ the set of all monoidal M\"obius inverses of $m$.
\end{definition}
\nomenclature[23]{$\mmi{P}{\cdot}$}{Set of all monoidal inverses of a function defined on a poset $P$}

Notice that if $\mathcal{G}$ is an abelian group and $g : P \to \mathcal{G}$ is any function, then the algebraic M\"obius inverse of $g$ is a monoidal M\"obius inverse of $g$. Indeed, the algebraic M\"obius inverse of $g$ is the unique monoidal M\"obius inverse of $g$ in this case. However, if $\mathcal{M}$ is a commutative monoid that is not an abelian group and $m : P \to \mathcal{M}$ is a function, then the algebraic M\"obius inverse of $m$ and a monoidal M\"obius inverse of $m$ have different codomains as functions. 

While the algebraic M\"obius inverse of $m$ is always guaranteed to exist when $P$ is finite, a monoidal M\"obius inverse of $m$ might not exist even in this case. Moreover, in the case when both inverses exist,  there might be more than one monoidal M\"obius inverse of $m$ whereas the algebraic M\"obius inverse is necessarily unique. 

\begin{example}[Non-existence] 
    Let $P = \{ 1 < 2< \cdots < n \}$ be the linear poset on $n$ elements. Any function $m : P \to \N$ has an algebraic M\"obius inverse given by 
    \begin{align*}
        \partial_P (m) : P &\to \Z \\
                       1 &\mapsto m(1) \\
                       i &\mapsto m(i)-m(i-1) \text{ for } i>1.
    \end{align*}
   However, $m : P \to \N$ has a monoidal M\"obius inverse if and only if $m$ is non-decreasing.
\end{example}

\begin{example}[Non-uniqueness]\label{ex: nonUniqueMI}
    Let $P = \{ 1 < 2 \}$ and let $\mathcal{M} = \gr(\R^2)$. Let $e_1, e_2 \in \R^2$ denote the standard basis elements of $\R^2$. Consider the function $m : \{ 1<2 \} \to \gr(\R^2)$ given by $m(1) = \spn \{ e_1 \}$, $m(2) = \spn \{ e_1, e_2 \}$. For any angle $\theta \in (0, \pi)$ consider the function $\ell_\theta : \{ 1 < 2 \} \to \gr(\R^2)$ defined by $\ell_\theta(1) = \spn \{ e_1\}$ and $\ell_\theta(2) = \spn \{ \cos (\theta) e_1 + \sin (\theta) e_2 \}$, see~\cref{fig:nonUniqueMI} for a visual illustration. Then, for any $\theta \in (0,\pi)$, $\ell_\theta$ is a monoidal M\"obius inverse of $m$ as it holds that 
    \[
    \sum_{p'\leq p} \ell_\theta (p') = m(p)
    \]
    for all $p \in \{ 1<2\}$.
\end{example}

\begin{figure}
    \centering
    \includegraphics[scale=28]{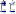}
    \caption{A function $m : \{ 1<2\} \to \gr(\R^2)$ and a one-parameter family of  monoidal M\"obius inverses for it, $\ell_\theta : \{ 1< 2\} \to \gr(\R^2)$ for $\theta\in(0,\pi)$, are depicted.}
    \label{fig:nonUniqueMI}
\end{figure}

The fact that the monoidal M\"obius inverse may not be unique, as demonstrated in~\cref{ex: nonUniqueMI}, motivates the following definition, which introduces an equivalence relation linking all functions serving as M\"obius inverses of the same function.

\begin{definition}[M\"obius equivalence]
    Two functions $m_1, m_2 : P \to \mathcal{M}$ are said to be \emph{M\"obius equivalent} if
    \[
    \sum_{p'\leq p} m_1 (p') = \sum_{p'\leq p} m_2 (p')
    \]
    for all $p\in P$. In this case, we write $m_1 \mobeq m_2$.
\end{definition}
\nomenclature[24]{$\mobeq$}{M\"obius equivalence}

\subsection{A Monoidal Rota's Galois Connection Theorem}

Rota's Galois Connection Theorem (RGCT)~\cite{gal-conn} describes how algebraic
M\"obius inversion behaves when a Galois connection exists between two posets. The functoriality of algebraic M\"obius inversion~\cite{edit, gal-conn} is indeed a direct consequence of the RGCT.
Now, we introduce a monoidal analog of the RGCT. The Monoidal RGCT will allow us to establish the functoriality of our construction in~\cref{sec:generalized orthogonal inversion}.

\begin{theorem}[Monoidal RGCT]\label{thm: monoidal rgct}
    Let $P$ and $Q$ be finite posets, $\ladj{f} : P \leftrightarrows Q : \radj{f}$ be a Galois connection, and $m: P \to \mathcal{M}$ be any function. Assume that $m' : P \to \mathcal{M}$ is a monoidal M\"obius inverse of $m$. Then, $(\ladj{f})_\sharp (m')$ is a monoidal M\"obius inverse of $(\radj{f})^\sharp m$, i.e., $(\ladj{f})_\sharp (m') \in \mmi{Q}{(\radj{f})^\sharp m}$.

\end{theorem}

\begin{proof}
    It suffices to prove that the following equality
    \[
    \sum_{q' \leq q} (\ladj{f})_\sharp \left(m'\right) (q') = (\radj{f})^\sharp m (q)
    \]
    holds for every $q \in Q$. Let $q\in Q$ be any element, starting from the left-hand side, we have
    \begin{align*}
        \sum_{q' \leq q} (\ladj{f})_\sharp \left( m' \right) (q') &= \sum_{q' \leq q} \sum_{p \in \ladj{f}^{-1}(q')} m' (p) \\
        &= \sum_{\substack{p\in P \\ \ladj{f}(p) \leq q}} m' (p) \\
        &= \sum_{\substack{p\in P \\ p \leq \radj{f}(q)}} m' (p) \\
        &= m (\radj{f}(q)) \\
        &= (\radj{f})^\sharp m (q).
    \end{align*}

\end{proof}

\begin{remark}
    Let $\ladj{f} : P \leftrightarrows Q : \radj{f}$ be a Galois connection between two finite posets, and $m: P \to \mathcal{M}$ be any function. Let $(\ladj{f})_\sharp \left( \mmi{P}{m}\right) := \left\{ (\ladj{f})_\sharp(g) \mid g \in \mmi{P}{m} \right\}$. Then,~\cref{thm: monoidal rgct} can be restated as
    \begin{equation}\label{eqn: monoidal rgct - push vs pull}
        (\ladj{f})_\sharp \left(\mmi{P}{m}\right) \subseteq \mmi{Q}{(\radj{f})^\sharp m}.
    \end{equation} 
    Note that the inclusion in~\cref{eqn: monoidal rgct - push vs pull} can be viewed as a generalization of the equality in~\cref{eqn: rgct push vs pull} from~\cref{thm:rgct} in the sense that it extends that result to the monoidal setting.
\end{remark}

\begin{example}[Monoidal RGCT]
    Let $P = \{p_1 < p_2 <p_3\}$, $Q = \{ q_1 <q_2\}$ and $\ladj{f} : P \leftrightarrows Q : \radj{f}$ be as in~\cref{ex:gal conn demo} (which are illustrated in~\cref{fig:RGCTdemo}). Let $m$ be the function defined by
    \begin{align*}
        m : P &\to \gr(\R^3) \\
            p_1 &\mapsto \spn \{e_1 \} \\
            p_2 &\mapsto \spn \{e_1, e_2 \} \\
            p_3 &\mapsto \spn \{e_1, e_2, e_3 \}.
    \end{align*}
    Observe that the function defined by 
    \begin{align*}
        m' : P &\to \gr(\R^3) \\
                                    p_1 &\mapsto \spn \{e_1 \} \\
                                     p_2 &\mapsto \spn \{ e_2 \} \\
                                     p_3 &\mapsto \spn \{e_3 \}.
    \end{align*}
    is a monoidal M\"obius inverse of $m$, i.e. $m' \in \mmi{P}{m}$. Also, the function defined by
    \begin{align*}
        n : Q &\to \gr(\R^3) \\
                                    q_1 &\mapsto \spn \{e_1 \} \\
                                     q_2 &\mapsto \spn \{ e_2+e_1, e_3+e_1 \}
    \end{align*}
    is a monoidal M\"obius inverse of $(\radj{f})^\sharp m$, i.e., $n\in \mmi{Q}{(\radj{f})^\sharp m}$. Observe that $(\ladj{f})_\sharp m' $ and $ n$ do not coincide as functions because
    \[
    (\ladj{f})_\sharp (\radj{f})^\sharp m (q_2) = \spn \{ e_2, e_3 \} \neq \spn \{ e_2+e_1, e_3+e_1 \} = n (q_2).
    \]
    Nevertheless, it holds that 
    \[
    (\ladj{f})_\sharp m \in \mmi{Q}{(\radj{f})^\sharp m},
    \]
    as stated in~\cref{thm: monoidal rgct}.
\end{example}

\section{Orthogonal Inversion and Orthomodular Inversion}\label{sec:generalized orthogonal inversion}

In this section, we introduce the key construction used throughout the paper: Orthogonal Inversion. We also extend this concept by introducing a generalization, which we call Orthomodular Inversion, that applies to order-preserving functions from a finite poset to an orthomodular lattice. The results regarding Orthogonal Inversion are presented in~\cref{subsec: orthogonal inversion section}, while the discussion of Orthomodular Inversion is covered in~\cref{sec: orthomodular inversion sec}.

\subsection{Orthogonal Inversion}\label{subsec: orthogonal inversion section}
In this section, we introduce the notion of \emph{Orthogonal Inversion}. This concept applies to order-preserving functions from a finite poset to a Grassmannian. We utilize Orthogonal Inversion later in~\cref{subsec: gpd for multiparameter} by exploring the applications of this notion to persistence over finite posets, i.e., the case when the underlying poset $P$ of a filtration $\Ffunc : P \to \subcx(K)$ can be any finite poset.

\begin{definition}[Monotone space function]\label{defn: monotone space functions}
    Let $V$ be a finite-dimensional inner product space and $R$ be any finite poset. An order-preserving function $\Ffrak : R\to \gr(V)$ is called a \emph{monotone space function}.
\end{definition}

\begin{definition}[Difference of subspaces]
    Let $V$ be an inner product space and let $W_1, W_2 \subseteq V$ be subspaces. We define the~\emph{difference} of two subspaces as
    \[
    W_1 \ominus W_2 := W_1 \cap W_2^\perp.
    \]
\end{definition}
\nomenclature[27]{$\ominus$}{Difference of subspaces}

\begin{remark}\label{remark: dimension difference}
    In the definition above, if $W_2 \subseteq W_1$, then $W_1 \ominus W_2 = W_1 \cap W_2^\perp$ is the orthogonal complement of $W_2$ inside of $W_1$. In this case, $\dim (W_1 \ominus W_2) = \dim W_1 - \dim W_2$. 
    In general (i.e. when $W_2 \nsubseteq W_1$), we have that 
    \[
    W_1 \ominus W_2 = W_1 \ominus \proj_{W_1} (W_2),
    \]
    where $\proj_{W_1} : V \to W_1$ is the orthogonal projection. This can be informally interpreted as expressing that $\proj_{W_1} (W_2)$ and $W_2$ are treated as being ``quasi-isomorphic'' with respect to $W_1$. 
\end{remark}

\begin{definition}[Orthogonal Inversion]\label{defn: goi}
    Let $R$ be a finite poset and let $\Ffrak: R \to\gr(V)$ be a monotone space function. We define the \emph{Orthogonal Inverse of $\Ffrak$} to be the function $\goi( \Ffrak) : R \to \gr(V)$ given by
    \[
    \goi (\Ffrak) (r) := \Ffrak(r) \ominus \left(\sum_{r'< r} \Ffrak(r')\right).
    \]
\end{definition}
\nomenclature[34]{$\goi$}{Orthogonal Inversion}

\begin{remark}[$\goi$ compresses information]\label{rmk: smaller support of goi}
        For a monotone space function $\Ffrak$, its Orthogonal Inverse $\goi (\Ffrak)$ has a smaller footprint than $\Ffrak$. To be precise, let $\Ffrak : R \to \gr(V)$ be a monotone space function. As $\Ffrak (r') \subseteq \Ffrak(r)$ for every $r'<r\in R$, we have $$\sum_{r'<r}\Ffrak (r') \subseteq \Ffrak(r).$$ Thus, $$\goi (\Ffrak)(r) = \Ffrak(r) \ominus \left( \sum_{r'<r} \Ffrak(r') \right)$$ is the orthogonal complement of $\sum_{r'<r}\Ffrak (r')$ inside $\Ffrak (r)$. In particular, $\goi(\Ffrak)(r) \subseteq \Ffrak(r)$ for every $r\in R$. This can be expressed as: ``$\goi(\Ffrak)$ has a smaller support than $\Ffrak$''. In other words, storing $\goi(\Ffrak)$ requires no more memory than storing $\Ffrak$. Indeed, it  always holds that
        \[
        \sum_{r\in R} \dim \left( \goi \left( \Ffrak \right)\right) \leq \sum_{r\in R} \dim \left( \Ffrak \right)
        \]
        and the equality seemingly takes place only in pathological cases such as when the underlying poset is an antichain or when $\Ffrak$ is identically zero. 
\end{remark}

In the following proposition, we show that the Orthogonal Inverse of a monotone space function produces a monoidal M\"obius inverse.

\begin{proposition}\label{prop: goi is monoidal mi}
    Let $R$ be a finite poset and let $\Ffrak : R \to \gr(V)$ be a monotone space function. Then, $\goi (\Ffrak)$ is a monoidal M\"obius inverse of $\Ffrak$, i.e., $\goi (\Ffrak) \in \mmi{R}{\Ffrak}$. That is,
    \[
    \sum_{r' \leq r} \goi (\Ffrak)(r') = \Ffrak(r)
    \]
    for every $r\in R$.
\end{proposition}

\begin{proof}
    Let $R$ be a finite poset and let $\Ffrak : R \to \gr(V)$ be a monotone space function. 
    We will proceed by strong induction on the poset $R$.

  \medskip
\noindent  
    \noindent \underline{Base cases:} Let $0_R$ be a minimal element of $R$ (note that there could be more than one minimal element of $R$). By~\cref{defn: goi}, $\goi (\Ffrak)\left(0_R\right) = \Ffrak \left(0_R\right)$. As $0_R$ is a minimal element of $R$, it follows that
    $$\sum_{\substack{r' \in R \\ r' \leq R}} \goi (\Ffrak) (r') = \Ffrak (r).$$
    Thus, we have the result for the base case. 
    
   \medskip
\noindent
    \noindent \underline{Inductive step:} Let $r \in R$ be any element and assume that for every $s < r$ it holds that
    $$\sum_{\substack{s' \in R \\ s' \leq s}} \goi (\Ffrak) (s') = \Ffrak (s).$$
    Then, due to the fact that $A + B = A + B + B$ for any $A, B \in \gr(V)$, it follows that
    \begin{align*}
        \sum_{\substack{r' \in R \\ r' \leq r}} \goi (\Ffrak) (r') &= \goi (\Ffrak) (r) + \sum_{\substack{r' \in R \\ r' < r}} \goi (\Ffrak) (r')  \\
        &= \goi (\Ffrak) (r) + \sum_{\substack{s \in R \\ s < r}} \sum_{s'\leq s} \goi (\Ffrak) (s') \\
        &= \goi (\Ffrak) (r) + \sum_{\substack{s \in R \\ s < r}} \Ffrak (s)
    \end{align*}
    By~\cref{defn: goi}, we have
    \[
    \goi (\Ffrak) (r)  =  \Ffrak(r) \ominus \sum_{\substack{s \in R \\ s < r}} \Ffrak (s) 
    \]
    Then, it follows that
    \begin{align*}
    \sum_{\substack{r' \in R \\ r' \leq r}} \goi (\Ffrak) (r') &= \goi (\Ffrak) (r) + \sum_{\substack{s \in R \\ s < r}} \Ffrak (s) \\
        &= \left ( \Ffrak(r) \ominus \sum_{\substack{s \in R \\ s < r}} \Ffrak (s) \right ) + \sum_{\substack{s \in R \\ s < r}} \Ffrak (s) \\
        &= \Ffrak (r)
    \end{align*}
    Thus, $\goi (\Ffrak)$ is a monoidal M\"obius inverse of $\Ffrak$.  
\end{proof}

The proposition above implies that any monotone space function $\Ffrak$ can be reconstructed from its Orthogonal Inverse, $\goi(\Ffrak)$. Additionally, considering that $\goi(\Ffrak)$ has a smaller support than $\Ffrak$, as detailed in~\cref{rmk: smaller support of goi}, one can interpret $\goi$ as serving as a form of lossless compression for monotone space functions.

\begin{example}[Orthogonal Inversion]
    Let $R = \{ r_1, r_2, r_3, r_4 \}$ be the poset given by the relations $r_1 < r_2$, $r_1 < r_3$, $r_2 < r_4$, and $r_3 < r_4$ as depicted in~\cref{fig: 4element poset} on the left. Let $V = \R^3$ and let $\{e_1, e_2, e_3 \}$ be the standard orthonormal basis of $\R^3$. Consider the monotone space function depicted in~\cref{fig: 4element poset} in the middle, defined as follows
    \begin{align*}
        \Ffrak : R &\to \gr\left(\R^3\right) \\
                 r_1 &\mapsto \spn \{e_1 \} \\
                 r_2 &\mapsto \spn \{e_1, e_2 \} \\
                 r_3 &\mapsto \spn \{e_1, e_2 \} \\
                 r_4 &\mapsto \R^3.
    \end{align*}    
    The Orthogonal Inverse of $\Ffrak$ is given by
    \begin{align*}
        \goi(\Ffrak) : R &\to \gr\left(\R^3\right) \\
                 r_1 &\mapsto \spn \{e_1 \} \\
                 r_2 &\mapsto \spn \{e_2 \} \\
                 r_3 &\mapsto \spn \{e_2 \} \\
                 r_4 &\mapsto \spn \{e_3 \},
    \end{align*}
    as illustrated in~\cref{fig: 4element poset} on the right. Observe that although both $\Ffrak$ and $\goi (\Ffrak)$ are supported everywhere in $P$, we have that $\goi (\Ffrak)(r) \subsetneq \Ffrak(r)$ for $r\in \{r_2, r_3, r_4 \}$. 
    Note that the family $\{ \goi(\Ffrak)(r) \}_{r\in R}$ is \emph{not} a transversal family as $\goi(\Ffrak)(r_2) = \goi(\Ffrak)(r_3) = \spn \{e_2 \}$.
\end{example}

\begin{figure}[t]
    \centering
    \includegraphics[scale=22]{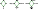}
    \caption{Poset R, order-preserving space function $\Ffrak$ and its Orthogonal Inverse $\goi(\Ffrak)$}
    \label{fig: 4element poset}
\end{figure}

As demonstrated in the example above, the Orthogonal Inverse of a monotone space function may not necessarily yield a transversal family. Nevertheless, the Orthogonal Inverse of monotone space functions still adheres to a weaker notion of transversality, which we will now describe in the following.

\begin{definition}[Downward transversal functions]
    A function $\Mfrak : R \to \gr(V)$ is called a \emph{downward transversal function} if the family $$\left\{ \sum_{r'<r} \Mfrak(r'), \Mfrak(r) \right\}$$ is a transversal family for every $r\in R$.
\end{definition}

\begin{proposition}\label{prop: goi gives downward transversel}
    Let $\Ffrak : R \to \gr(V)$ be a monotone space function. Then, $\goi (\Ffrak)$ is a downward transversal function.
\end{proposition}

\begin{proof}
    Let $r \in R$ be any element. By~\cref{defn: goi}, we have that $\goi (\Ffrak) (r) = \Ffrak (r) \ominus \sum_{r' < r} \Ffrak(r')$. Due to the fact that $A + B = A + B + B$ for any $A, B \in \gr(V)$ and by~\cref{prop: goi is monoidal mi}, it follows that
    \[
    \sum_{r'<r} \goi (\Ffrak)(r') = \sum_{r'<r}\sum_{s\leq r'} \goi (\Ffrak)(s) = \sum_{r' < r} \Ffrak(r').
    \]
    Therefore, we obtain that
    \[
    \goi(\Ffrak)(r) = \Ffrak(r) \ominus \sum_{r'<r}\Ffrak(r') = \Ffrak(r) \ominus \sum_{r'<r} \goi(\Ffrak)(r').
    \]
    Hence, the subspaces $\goi(\Ffrak)(r)$ and $\sum_{r'<r}\goi(\Ffrak)(r')$ are orthogonal to each other. Thus, the family
    \[
    \left\{\goi(\Ffrak)(r),  \sum_{r'<r}\goi(\Ffrak)(r')\right\}
    \]
    is a transversal family. So, we conclude that $\goi (\Ffrak)$ is a downward transversal function.
\end{proof}

We now proceed by organizing the collection of monotone space functions into a category, and similarly, we organize the collection of downward transversal functions into another category. We then establish that the Orthogonal Inversion operates as a functor between these categories. As a result of this functoriality, we also obtain edit distance stability.

\begin{definition}[Category of monotone space functions]
    We define $\mon(V)$ to be the category where
    \begin{itemize}
        \item Objects are monotone space functions defined over any finite metric poset,
        \item A morphism from a monotone space function $\Ffrak : R \to \gr(V)$ to another monotone space function $\Gfrak : Q\to \gr(V)$ is given by a Galois connection $\ladj{f} : R \leftrightarrows Q : \radj{f}$ such that $\Ffrak \circ \radj{f} = \Gfrak$.
    \end{itemize}
\end{definition}
\nomenclature[35]{$\mon(\cdot)$}{Category of monotone space functions}

\begin{definition}[Cost of a morphism in $\mon(V)$]
    The cost of a morphism in $\mon(V)$ is given by $\dis(\ladj{f})$, the distortion of the left adjoint.
\end{definition}

\begin{definition}[Category of downward transversal functions]
    We define $\dtfnc(V)$ to be the category where
    \begin{itemize}
        \item Objects are downward transversal functions defined over any finite metric poset,
        \item A morphism from an object $\Mfrak : R \to \gr(V)$ to another object $\Nfrak : Q \to \gr(V)$ is given by a Galois connection $\ladj{f} : R \leftrightarrows Q : \radj{f}$ such that $\left(\overline{\ladj{f}}\right)_\sharp \left(\Mfrak \right) \mobeq \Nfrak$.
    \end{itemize}   
\end{definition}
\nomenclature[36]{$\dtfnc(\cdot)$}{Category of downward transversal functions}

\begin{definition}[Cost of a morphism in $\dtfnc(V)$]
        The cost of a morphism in $\dtfnc(V)$ is given by $\dis(\ladj{f})$, the distortion of the left adjoint. 
\end{definition}

As a result of the fact that $\goi$ produces a monoidal M\"obius inverse,~\cref{prop: goi is monoidal mi}, and by the monoidal RGCT,~\cref{thm: monoidal rgct}, we obtain the following.

\begin{theorem}[Functoriality and stability of $\goi$]\label{thm: goi functor and stable}
    The assignment 
    \begin{align*}
        \goi : \mon(V) &\to \dtfnc(V) \\
                   \Ffrak &\mapsto \goi (\Ffrak) 
    \end{align*}
    is a functor. Moreover, for any two objects $\Ffrak$ and $\Gfrak$ in $\mon(V)$, we have
    \[
    d_{\dtfnc(V)}^E (\goi (\Ffrak), \goi (\Gfrak)) \leq d_{\mon(V)}^E (\Ffrak, \Gfrak).
    \]
\end{theorem}

\begin{proof}
    Let $ \Ffrak : R \to \gr(V)$ be an object in $\mon(V)$. By~\cref{prop: goi gives downward transversel}, we have that $\goi (\Ffrak)$ is an object in $\dtfnc(V)$. Now, let $\Gfrak : S \to \gr(V)$ be another object in $\mon(V)$ and let $\ladj{f} : R \leftrightarrows S : \radj{f}$ a Galois connection that determines a morphism from $\Ffrak$ to $\Gfrak$. This means that $ \left(\overline{\radj{f}}\right)^\sharp \Ffrak=  \Ffrak \circ \overline{\radj{f}} = \Gfrak$. By~\cref{prop: goi is monoidal mi}, we have that $\goi (\Ffrak)$ and $\goi (\Gfrak)$ are monoidal M\"obius inverses of $\Ffrak$ and $\Gfrak$ respectively. Then, by the monoidal RGCT,~\cref{thm: monoidal rgct}, we have that 
    \[
    \left(\overline{\ladj{f}}\right)_\sharp \goi (\Ffrak) \mobeq \goi (\Gfrak).
    \]
    Therefore, the Galois connection $(\ladj{f}, \radj{f})$ is a morphism from $\goi (\Ffrak)$ to $\goi (\Gfrak)$. Thus $\goi$ is a functor from $\mon(V)$ to $\dtfnc(V)$. The edit distance stability
    \[
    d_{\dtfnc(V)}^E (\goi (\Ffrak), \goi (\Gfrak)) \leq d_{\mon(V)}^E (\Ffrak, \Gfrak)
    \]
    follows from the fact that $\goi$ is a functor from $\mon(V)$ to $\dtfnc(V)$.
\end{proof}

In the following, we prove that we not only achieve stability but also obtain isometry.

\begin{theorem}[Isometry]\label{thm: isometry of goi}
    For any two objects $\Ffrak$ and $\Gfrak$ in $\mon(V)$, we have
    \[
    d_{\dtfnc(V)}^E (\goi (\Ffrak), \goi (\Gfrak)) = d_{\mon(V)}^E (\Ffrak, \Gfrak).
    \]
\end{theorem}

\begin{proof}
    We have the inequality
    \[
    d_{\dtfnc(V)}^E (\goi (\Ffrak), \goi (\Gfrak)) \leq d_{\mon(V)}^E (\Ffrak, \Gfrak)
    \]
    by~\cref{thm: goi functor and stable}. Thus, it suffices to prove the following inequality
    \[
    d_{\dtfnc(V)}^E (\goi (\Ffrak), \goi (\Gfrak)) \geq d_{\mon(V)}^E (\Ffrak, \Gfrak).
    \]
    We prove the above inequality by constructing a functor $ \lowersum : \dtfnc(V) \to \mon(V)$ that repects the cost of morphisms. For any $\Mfrak : R \to \gr(V)$ downward transversal function in $\dtfnc(V)$, let 
    \[
    \lowersum (\Mfrak) (r) := \sum_{r'\leq r} \Mfrak(r').
    \]
    For any $r'\leq r \in R$, $\lowersum (\Mfrak) (r') \subseteq \lowersum (\Mfrak) (r)$. Thus, $\lowersum (\Mfrak) : R \to \gr(V)$ is an order-preserving space function, i.e., $\lowersum (\Mfrak)$ is in $\mon(V)$. To see that $\lowersum$ is a functor, let $\Mfrak : R \to \gr(V)$ and $\Nfrak : S \to \gr(V)$ be two downward transversal functions and let $\ladj{f} : R \leftrightarrows S : \radj{f}$ be a Galois connection that determines a morphism from $\Mfrak$ to $\Nfrak$. This means that
    \[
    \left(\overline{\ladj{f}}\right)_\sharp \left(\Mfrak \right) \mobeq \Nfrak.
    \]
    That is, for any $s\in S$, we have
    \begin{align*}
        \sum_{s'\leq s}\Nfrak (s') &= \sum_{s'\leq s} \left(\overline{\ladj{f}}\right)_\sharp \left(\Mfrak \right) (s') \\
        &= \sum_{s'\leq s} \sum_{\substack{r\in R \\ \ladj{f}(r) = s'}} \Mfrak(r) \\
        &= \sum_{\substack{r\in R \\ \ladj{f}(r)\leq s}} \Mfrak (r) \\
        &= \sum_{\substack{r\in R \\ r\leq \radj{f}(s)}} \Mfrak (r).
    \end{align*}
    Thus, we obtained that 
    \begin{align*}
        \lowersum (\Nfrak)(s) &= \sum_{s'\leq s}\Nfrak (s') \\
        &= \sum_{\substack{r\in R \\ r\leq \radj{f}(s)}} \Mfrak (r) \\
        &= \lowersum (\Mfrak)(\radj{f}(s)).
    \end{align*}
    That is, $\lowersum (\Nfrak) = \left( \radj{f} \right)^\sharp (\lowersum(\Mfrak))$, i.e., the Galois connection $\ladj{f} : R \leftrightarrows S : \radj{f}$ is a morphism from $\lowersum(\Mfrak)$ to $\lowersum(\Nfrak)$. Therefore, $\lowersum : \dtfnc(V) \to \mon(V)$ is a functor. 

    Now, let $\Ffrak$ and $\Gfrak$ be two order-preserving space functions. Obserbe that, by~\cref{prop: goi is monoidal mi}, we have that $\lowersum (\goi (\Ffrak)) = \Ffrak$ and $\lowersum (\goi (\Gfrak)) = \Gfrak$. Therefore, a path between $\goi(\Ffrak)$ and $\goi (\Gfrak)$ in $\dtfnc(V)$ induces a path between $\Ffrak$ and $\Gfrak$ in $\mon(V)$ with the same cost. Thus, we obtain
    \[
    d_{\dtfnc(V)}^E (\goi (\Ffrak), \goi (\Gfrak)) \geq d_{\mon(V)}^E (\Ffrak, \Gfrak).
    \]

\end{proof}

\subsection{Orthomodular Inversion}\label{sec: orthomodular inversion sec}
In this section, we replace Grassmannians with  more general objects: \emph{orthomodular lattices}; see~\cref{defn: orthomodular lattices}. We prove that order-preserving functions from a finite poset to an orthomodular lattice can be M\"obius inverted through a construction that generalizes Orthogonal Inversion.

A tuple $(L, \leq, \vee, \wedge)$ is called a \emph{lattice} if $(L,\leq)$ is a poset and every two elements $a,b\in L$
has a least upper bound, denoted $a\vee b$, and a greatest lower bound, denoted $a \wedge b$. A tuple $(L, \leq, \vee, \wedge, 0_L, 1_L)$ is called a \emph{bounded lattice} if $L$ has smallest and greatest elements, denoted $0_L$ and $1_L$ respectively. 

\begin{definition}[Orthomodular lattices~{\cite[page 53]{Birkhoff1940}}]\label{defn: orthomodular lattices}
    A tuple $(L, \wedge, \vee, 0_L, 1_L, (\cdot)^\perp)$ is called an~\emph{orthomodular lattice} if $(L, \wedge, \vee, 0_L, 1_L,)$ is a bounded lattice and $(\cdot)^\perp : L \to L$ is function, called \emph{orthocomplementation}, such that the following conditions hold for all $a,b\in L$:
    \begin{itemize}
        \item If $a\leq b$, then $b^\perp \leq a^\perp$,
        \item $a \vee a^\perp = 1_L$, $a\wedge a^\perp = 0_L$, $(a^\perp)^\perp = a$,
        \item If $a\leq b$, then $b = a \vee (a^\perp \wedge b)$.
    \end{itemize}
\end{definition}

Observe that the Grassmannian of a finite-dimensional inner product space $V$, $\gr(V)$, is an orthomodular lattice where:
\begin{itemize}
\item[--] $\vee$ is taken as the sum of subspaces,
\item[--] $\wedge$ is taken as the intersection of subspaces, 
\item[--] $0_{\gr(V)} = \{ 0 \}$, 
\item[--] $1_{\gr(V)} = V$, and 
\item[--] $(\cdot)^\perp$ is taken as the orthogonal complement operation. 
\end{itemize}

\begin{remark} 
Note that the two structures on Grassmannians that we have been utilizing in this paper, the monoid structure (e.g. \cref{prop: goi is monoidal mi}) and the poset structure (e.g. \cref{defn: monotone space functions}), are unified when a Grassmannian is viewed as an orthomodular lattice.  The notion of orthomodular lattice was first studied by Birkhoff and Von Neumann in~\cite[Section 11]{Birkhoff1936} to formalize a logic system for quantum mechanics. Building upon this rich mathematical landscape, our paper extends the application of orthomodular lattices to the theory of persistence.
\end{remark}

We now describe a notion of M\"obius inversion for lattice-valued functions. The following definition closely resembles the classical/monoidal Möbius inversion, with the sole modification being the replacement of the group (or monoid) operation ``$+$'' with the join operation ``$\vee$'' of the lattice.

\begin{definition}[Join M\"obius inversion]\label{defn: join mi}
    Let $R$ be poset and $L$ be a lattice, and $\Ffrak : R \to L$ be an order-preserving function. A function $\Ffrak' : R \to L$ is called a~\emph{join M\"obius inverse} of $\Ffrak$ if it satisfies
    \[
    \Ffrak(r) = \bigvee_{r'\leq r} \Ffrak'(r')
    \]
    for every $r \in R$.
\end{definition}

The following proposition establishes that the Rota's Galois Connection Theorem still holds for functions valued in lattices.

\begin{proposition}[Join RGCT]\label{prop: join rgct}
    Let $R$ and $Q$ be finite posets and $\ladj{f} : R \leftrightarrows Q : \radj{f}$ be a Galois connection. Let $L$ be a lattice, $\Ffrak : R \to L$ be an order-preserving map, and $\Ffrak' : R \to L$ be a join M\"obius inverse of $\Ffrak$. Then, the function
    \begin{align*}
        \left( \ladj{f}\right)_\sharp(\Ffrak') : Q &\to L \\
                                            q &\mapsto \bigvee_{\substack{r\in R \\ \ladj{f}(r)=q}} \Ffrak'(r)
    \end{align*}        
    is a join M\"obius inverse of $\left(\radj{f} \right)^\sharp(\Ffrak) = \Ffrak\circ \radj{f}$.
\end{proposition}

\begin{proof}
    It suffices to prove the following equality
    \[
    \bigvee_{q' \leq q} (\ladj{f})_\sharp (\Ffrak') (q') = \Ffrak \left((\radj{f}(q)\right)
    \]
    for every $q \in Q$. Let $q\in Q$ be any element, then, we have
    \begin{align*}
    \bigvee_{q' \leq q} (\ladj{f})_\sharp (\Ffrak') (q') &= \bigvee_{q' \leq q} \bigvee_{r \in \ladj{f}^{-1}(q')} \Ffrak' (r) \\
    &= \bigvee_{\substack{r\in R \\ \ladj{f}(r) \leq q}} \Ffrak' (r) \\
    &= \bigvee_{\substack{r\in R \\ r \leq \radj{f}(q)}} \Ffrak' (r) \\
    &= \Ffrak(\radj{f}(q)).
    \end{align*}
\end{proof}

For order-preserving functions valued in an orthomodular lattice, it is always possible to obtain a join M\"obius inverse. We will now describe this procedure and prove that it indeed yields  a join M\"obius inverse.

\begin{definition}[Orthomodular Inversion]\label{defn: modular inv}
    Let $R$ be a finite poset and $L$ be an orthomodular lattice. Let $\Ffrak: R \to L$ be an order-preserving function. We define the~\emph{Orthomodular Inverse} of $\Ffrak$ to be the function defined by
    \begin{align*}
        \partial^\perp_R (\Ffrak) : R &\to L \\
                    r &\mapsto \Ffrak(r) \wedge \left(\bigvee_{r'< r} \Ffrak(r')\right)^\perp.    
    \end{align*}
\end{definition}
\nomenclature[37]{$\partial_R^\perp(\cdot)$}{Orthomodular inversion of a function defined on a poset $R$}

\begin{proposition}[Orthomodular Inversion is a join M\"obius inversion]\label{prop: perp inv is join mobius inv}
    Let $R$ be a finite poset and $L$ be an orthomodular lattice. Let $\Ffrak: R\to L$ be an order-preserving function. Then, the Orthomodular Inverse of $\Ffrak$, $\partial^\perp_R (\Ffrak) : R\to L$, is a join M\"obius inverse of $\Ffrak$.
\end{proposition}

\begin{proof}
    We will proceed by strong induction on the poset $R$.
    
 \medskip
\noindent
    \noindent \underline{Base cases:} Let $0_R$ be a minimal element of $R$ (note that there could be more than one minimal element of $R$). By~\cref{defn: modular inv}, $\partial^\perp_R (\Ffrak)\left(0_R\right) = \Ffrak \left(0_R\right)$. As $0_R$ is a minimal element of $R$, it follows that
    $$\bigvee_{\substack{r' \in R \\ r' \leq r}} \partial^\perp_R (\Ffrak) (r') = \Ffrak (r).$$
    Thus, we have the result for the base case. 
    
 \medskip
\noindent
    \noindent \underline{Inductive step:} Let $r \in R$ be any element and assume that for every $s < r$ it holds that
    $$\bigvee_{\substack{s' \in R \\ s' \leq s}} \partial^\perp_R (\Ffrak) (s') = \Ffrak (s).$$
    Then, due to the fact that $a\vee b = a\vee b\vee b$ for any $a,b \in L$, it follows that
    \begin{align*}
        \bigvee_{\substack{r' \in R \\ r' \leq r}} \partial^\perp_R (\Ffrak) (r') &= \partial^\perp_R (\Ffrak) (r) \vee \left( \bigvee_{\substack{r' \in R \\ r' < r}} \partial^\perp_R (\Ffrak) (r') \right)  \\
        &= \partial^\perp_R (\Ffrak) (r) \vee \left( \bigvee_{\substack{s \in R \\ s < r}} \bigvee_{s'\leq s} \partial^\perp_R (\Ffrak) (s') \right) \\
        &= \partial^\perp_R (\Ffrak) (r) \vee \left( \bigvee_{\substack{s \in R \\ s < r}} \Ffrak (s) \right).
    \end{align*}
    By~\cref{defn: modular inv}, we have
    \[
    \partial^\perp_R (\Ffrak) (r)  =  \Ffrak(r) \wedge \left( \bigvee_{\substack{s \in R \\ s < r}} \Ffrak (s) \right)^\perp
    \]
    Then, it follows that
    \begin{align*}
    \bigvee_{\substack{r' \in R \\ r' \leq r}} \partial^\perp_R (\Ffrak) (r') &= \partial^\perp_R (\Ffrak) (r) \vee \left( \bigvee_{\substack{s \in R \\ s < r}} \Ffrak (s) \right) \\
        &= \left ( \Ffrak(r) \wedge \left( \bigvee_{\substack{s \in R \\ s < r}} \Ffrak (s)\right)^\perp \right ) \vee \bigvee_{\substack{s \in R \\ s < r}} \Ffrak (s) \\
        &= \Ffrak (r)
    \end{align*}
    Thus, $\partial^\perp_R (\Ffrak)$ is a join M\"obius inverse of $\Ffrak$.  
\end{proof}

\begin{remark}
 Orthomodular Inversion (\cref{defn: modular inv}) generalizes Orthogonal Inversion (\cref{defn: goi}). Indeed, they agree when applied to an order-preserving function $\Ffrak: R \to \gr(V)$. 
\end{remark}

\section{Grassmannian Persistence Diagrams}
\label{subsec: gpd for multiparameter}

In this section, we apply our Orthogonal Inversion framework from \cref{sec:generalized orthogonal inversion} to analyze simplicial filtrations over finite posets. 

Let $K$ be a finite simplicial complex. Recall from~\cref{parag: simplicial filtration} that a simplicial filtration of $K$ is an order-preserving function, $F : P \to \subcx(K)$, where $P$ is a finite poset and $\subcx(K)$ is the poset of subcomplexes of $K$. Note that we do not require $P$ to be linearly ordered. To indicate this, we refer to filtrations $\Ffunc : P \to \subcx(K)$ as \emph{filtrations over general posets}. Observe that, for any degree $\dgr\geq 0$, the $\dgr$-th birth-death spaces associated to a filtration $\Ffunc : P \to \subcx(K)$, $\ZB_\dgr^\Ffunc : \overline{P}^\times \to \gr\left( C_\dgr^K \right)$, is a monotone space function, where $\overline{P}^\times= (\Seg(P),\prodord)$ is the poset of segments of $P$ with the product order. This motivates the following definition.

\begin{definition}[Grassmannian persistence diagrams of filtrations over a general posets]\label{defn: gpd of multiparameter filtrations}
    Let $P$ be any finite poset and let $\Ffunc : P\to \subcx(K)$ be a filtration. For any degree $\dgr\geq 0$, we define the \emph{degree-$\dgr$ Grassmannian persistence diagram} of $\Ffunc$ as the Orthogonal Inverse of $\ZB_\dgr^\Ffunc$, i.e., $$\goi \left(\ZB_\dgr^\Ffunc\right) : \Seg(P) \to \gr\left( C_\dgr^K \right).$$
\end{definition}

\begin{figure}[t]
    \centering
    \includegraphics[scale=15]{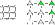}
    \caption{On the left is the product poset $\{ 1<2<3 \} \times \{1<2<3 \}$, refered as the $3 \times 3$ grid. The ordering is given by the product order, i.e., $(i,j) \leq (k,l) \iff i \leq k$ and $ j\leq l$. On the right is a multiparameter filtration over the $3\times 3$ grid.}
    \label{fig: 3x3 filtration}
\end{figure}

\begin{example}\label{ex: multiparameter gpd}
    Let $P = \{1<2<3 \}\times \{1<2<3 \}$ be the $3$ by $3$ grid with the product order, as depicted in~\cref{fig: 3x3 filtration}. Let $\Ffunc$ be the filtration over $P$ depicted in~\cref{fig: 3x3 filtration}. For $\dgr=1$, $\goi \left(\ZB_1^\Ffunc\right) : \Seg(P) \to \gr\left(C_1^K\right)$ is given by
    \[
        \goi \left(\ZB_1^\Ffunc\right)(I) =\begin{cases}
            \spn \{ ab - ac + bc \} & \text{ if } I \in \left\{ \begin{multlined}
                \left((1,2), (2,3)\right), \left((1,2), (3,2)\right), \vspace{-1em} \\  \left((2,1), (2,3)\right), \left((2,1), (3,2)\right)
            \end{multlined} \right\}\\
            \{ 0 \} &\text { otherwise.}
        \end{cases}
    \]

    Note that the Grassmannian persistence diagram captures the birth and death of some obvious degree 1 cycles in the filtration. This phenomenon which will be extensively explored next.   Note also that $\goi \left(\ZB_1^\Ffunc\right)$ is supported on four segments, whereas $\ZB_1^\Ffunc$ has a support of size $16$. This indicates that $\goi \left(\ZB_1^\Ffunc\right)$ has significantly smaller support compared to $\ZB_1^\Ffunc$ as discussed in~\cref{rmk: smaller support of goi}.
\end{example}

\begin{remark}[Computational complexity for $2$-parameter filtrations]\label{rem: computability of 2-parameter}
In \cite[Appendix C]{1parameterGrassmannian}, we describe an algorithm for computing Grassmannian persistence diagrams in the one-parameter setting and provided an analysis of its time complexity. The algorithm operates in a brute-force manner by computing the Grassmannian persistence diagram for all segments. This algorithm, together with its complexity analysis, can be naturally extended to $2$-parameter (and even $n$-parameter) filtrations which we illustrate below.

Specifically, let
\[
\Ffunc \colon P := \{1 < 2 < \cdots < m\} \times \{1 < 2 < \cdots < m\} \;\to\; \subcx(K)
\]
be a $2$-parameter filtration. Then in a brute-force manner, we compute $\goi \bigl(\ZB_\dgr^\Ffunc\bigr)$ for each of the $O(m^4)$ segments in the $m \times m$ grid as follows: For each segment $((i,j),(k,l)) \in \Seg(P)$, we construct a basis for the space:
        \begin{align*}
            \goi \left( \ZB_\dgr^\Ffunc \right)((i,j),(k,l)) &= \ZB_\dgr^\Ffunc ((i,j),(k,l)) \ominus \left(\sum_{I<_\times ((i,j),(k,l))} \ZB_\dgr^\Ffunc(I) \right) \\
            &=  \\
            \ZB_\dgr^\Ffunc ((i,j),(k,l)) &\ominus \left( \begin{multlined} \ZB_\dgr^\Ffunc ((i-1,j),(k,l)) + \ZB_\dgr^\Ffunc ((i,j-1),(k,l)) \\ + \ZB_\dgr^\Ffunc ((i,j),(k-1,l)) + \ZB_\dgr^\Ffunc ((i,j),(k,l-1)) \end{multlined}\right).
            \end{align*}
This computation boils down to computing bases for subspaces obtained through intersection, sum as well as the $\ominus$ operation. All these linear algebra operations have been analyzed in \cite[Appendix C]{1parameterGrassmannian} (see \cite[Lemmas C.2, C.3 and C.4]{1parameterGrassmannian}) and the net complexity of the above computation is
        $$O\left(  n_\dgr^K \cdot n_{\dgr-1}^K \cdot \min\left(n_\dgr^K, n_{\dgr-1}^K\right) + n_{\dgr+1}^K \cdot n_{\dgr}^K \cdot \min\left(n_{\dgr+1}^K, n_{\dgr}^K\right) +    \left(n_\dgr^K\right)^3 \right),   $$
where, for each degree $\rho,$ $n_\dgr^K$ denotes the number of $\dgr$-simplices in $K$. 

Hence, the total time complexity of the algorithm described above for computing the degree-$\dgr$ Grassmannian persistence diagram of $\Ffunc$ is
\begin{equation}\label{eq: complexity}
O\!\Bigl( m^4  \,\Bigl( 
n_\dgr^K\,n_{\dgr-1}^K\,\min\bigl(n_\dgr^K,\,n_{\dgr-1}^K\bigr)\;+\; 
n_{\dgr+1}^K\,n_\dgr^K\,\min\bigl(n_{\dgr+1}^K,\,n_\dgr^K\bigr)\;+\; 
\bigl(n_\dgr^K\bigr)^3 
\Bigr)\Bigr).
\end{equation}
Treating the quantities $n_{\dgr-1}^K$, $n_{\dgr}^K$, and $n_{\dgr+1}^K$ as constants we see that the complexity of this procedure is $O(m^4)$, which matches the complexity of computing rank-induced persistence diagrams via a brute-force algorithm~\cite{botnan2022rectangle}; see also \cite{morozov-patel} for an output-sensitive algorithm for computing birth-death induced persistence diagrams of 2-parameter filtrations.
\end{remark}

In the remainder of this section, we demonstrate that the Grassmannian persistence diagram of a filtration over a general poset (1) offers a clear and interpretable description of the evolution of topological features; (2) is stable with respect to a specific edit distance, which we define below; and (3) subsumes and strictly strengthens the information captured by birth-death induced signed persistence diagrams.

\subsection{Interpretability of Grassmannian Persistence Diagrams: Canonical Cycles and Homological Exhaustiveness}\label{subsec: inter of gpd canonical cyyles and exhaustiveness}

As observed in~\cref{ex: multiparameter gpd}, for any segment $(b,d) \in \Seg(P)$, the cycles determined by $\goi \left(\ZB_\dgr^\Ffunc \right) ((b,d))$ are born exactly at $b$ and die exactly at $d$. We now show that this is not a coincendence, but rather a fact.

\begin{theorem}[$\goi\left(\ZB_\dgr^\Ffunc\right)$ produces fully supported canonical cycles]\label{prop: born and dead multiparameter}
    Let $\Ffunc : P \to \subcx(K)$ be a filtration. Let $(b,d)\in \Seg(P)$ and let $z\in \goi \left( \ZB_\dgr^\Ffunc\right) ((b,d)) $ be a nonzero cycle. Then, $z$ is born precisely at $b$ and dies precisely at $d$ (see \cref{def:bd-cycles}).
\end{theorem}

\begin{proof}

Let $z \in \goi\left( \ZB_\dgr^\Ffunc \right)((b, d))$ be a nonzero cycle. By~\cref{defn: goi}, we have that 
\begin{align*}
    \goi\left( \ZB_\dgr^\Ffunc \right)((b, d)) &= \ZB_\dgr^\Ffunc ((b,d)) \ominus  W_\dgr^\Ffunc((b,d))  \\
    &= \ZB_\dgr^\Ffunc ((b,d)) \cap \left( W_\dgr^\Ffunc((b,d)) \right)^\perp.
\end{align*}
As $z$ is nonzero, we conclude that $z\notin W_\dgr^\Ffunc((b,d)) = \sum_{(a,c)<_\times (b,d)} \ZB_\dgr^\Ffunc ((a,c))$ and $z \in \ZB_\dgr^\Ffunc ((b,d))$. Therefore, $z$ is born at $b$ and dies at $d$.
\end{proof}

\subsubsection{Canonicality}

Note that, according to \cref{defn: goi,defn: gpd of multiparameter filtrations}, $\goi$ provides an assignment of intervals to  subspaces of $C_\dgr^K$ that does not depend on any superfluous choices. As such we regard this assignment as \emph{canonical}.

Indeed,  let $K$ be a simplicial complex over a vertex set $V = \{ v_1,\ldots v_m \}$ and $\xi(K)$ be the simplicial complex obtained from $K$ by relabelling (i.e. permuting) its vertices. That is, $\xi(K)$ is a simplicial complex over a vertex set $W = \{ w_1, \ldots w_m \}$ such that there exists a permutation $\xi : \{1,\ldots,m \} \to \{ 1,\ldots,m\}$ with 
\[
\{ v_{i_0},\ldots,v_{i_\dgr} \} \in K \iff \{ w_{\xi\left(i_0\right)},\ldots, w_{\xi\left(i_\dgr\right)} \} \in \xi(K).
\]
Let $\Ffunc : P \to \subcx(K)$ be a filtration filtration of $K$. Then, $\xi$ determines a filtration $\xi(\Ffunc) : P \to \subcx(\xi(K))$ given by
\[
\{ v_{i_0} ,\ldots ,v_{i_\dgr} \} \in \Ffunc(a) \iff \{ w_{\xi\left(i_0\right)},\ldots w_{\xi\left(i_\dgr\right)} \} \in \xi(\Ffunc)(a),
\]
for any $a\in P$.
In the following proposition, we show that for each segment $(b, d) \in \Seg(P)$, the subspaces of $C_\dgr^K$ and $C_\dgr^{\xi(K)}$ determined by $\goi$ are mapped to each other through the  \emph{canonical isomorphism} given by 
\begin{align*}
    \Omega_\dgr^{\xi} : C_\dgr^K &\to C_\dgr^{\xi(K)} \\
                    [v_{i_0},\ldots,v_{i_\dgr}] &\mapsto [w_{\xi(i_0)},\ldots w_{\xi(i_\dgr)}].
\end{align*}

\begin{proposition}[Canonicality]\label{prop: canonicality} 
     For any permutation $\xi : \{1,\ldots,m\} \to \{1,\ldots,m \}$ and for any segment $(b,d)\in \Seg(P)$, we have that $$\Omega_\dgr^{\xi} \left(\goi\left(\ZB_\dgr^\Ffunc\right)((b,d))\right) = \goi\left(\ZB_\dgr^{\xi(\Ffunc)}\right)((b,d)).$$
In other words, the following diagram commutes:
 \begin{center}
           \begin{tikzcd}
           F \arrow[d, "\ZB_\rho^\bullet"] \arrow[r, "\xi"]  & \xi(F) \arrow[d, "\ZB_\rho^\bullet"]\\
      \ZB_\dgr^F \arrow[d, "\goi"]  & \ZB_\dgr^{\xi(F)} \arrow[d, "\goi"] \\
        \goi\big(\ZB_\dgr^F\big) \arrow[r, "\Omega_{\dgr}^{\xi}"]                   & \goi\big(\ZB_\dgr^{\xi(F)}\big)                                
        \end{tikzcd} 
    \end{center}
     
\end{proposition}

\begin{proof}
    Recall that $C_\dgr^K$ and $C_\dgr^{\xi(K)}$ are endowed with their respective standard inner products as described in~\cref{sec: prelim}, with respect to which we have that 
    \begin{align*}
    \Omega_{\xi} : C_\dgr^K &\to C_\dgr^{\xi(K)} \\
                    [v_{i_0},\ldots,v_{i_\dgr}] &\mapsto [w_{\xi(i_0)},\ldots w_{\xi(i_\dgr)}].
    \end{align*}
    is an isometry. Moreover, $\Omega_\dgr^{\xi}$ preserves the cycle and boundary spaces. To see this, observe that the following diagram
    \begin{center}
           \begin{tikzcd}
      C_\dgr^K \arrow[r, "\Omega_\dgr^{\xi}"] \arrow[d, "\partial_\dgr^K"] & C_\dgr^{\xi(K)} \arrow[d, "\partial_\dgr^{\xi(K)}"] \\
        C_{\dgr-1}^K \arrow[r, "\Omega_{\dgr-1}^{\xi}"]                   & C_{\dgr-1}^{\xi(K)}                                
        \end{tikzcd} 
    \end{center}
    commutes as
    \begin{align*}
        \partial_\dgr^{\xi(K)} \left( \Omega_\dgr^{\xi} \left( [v_{i_0},\ldots,v_{i_\dgr}] \right)\right) &= \partial_\dgr^{\xi(K)} \left( [w_{\xi(i_0)},\ldots,w_{\xi(i_\dgr)}] \right) \\
        &=  \sum_{j=0}^\dgr (-1)^j [w_{\xi(i_0)},\ldots,\widehat{w_{\xi(i_j)}},\ldots,w_{\xi(i_\dgr)}] \\
        &= \sum_{j=0}^\dgr (-1)^j \Omega_{\dgr-1}^{\xi} \left([v_{i_0},\ldots,\widehat{v_{i_j}},\ldots,v_{i_\dgr}] \right) \\
        &= \Omega_{\dgr-1}^{\xi} \left(\sum_{j=0}^\dgr (-1)^j \left([v_{i_0},\ldots,\widehat{v_{i_j}},\ldots,v_{i_\dgr}] \right)\right) \\
        &= \Omega_{\dgr-1}^{\xi} \left( \partial_\dgr^K \left( [v_{i_0}, \ldots,v_{i_\dgr}] \right) \right).
    \end{align*}
    Hence, $\Omega_\dgr^{\xi} \left( \ZB_\dgr^\Ffunc((b,d)) \right) = \ZB_\dgr^{\xi(\Ffunc)} ((b,d)) $ for every $(b,d) \in \Seg (P)$. Combining this with the fact that $\Omega_\dgr^{\xi}$ is a isometry, we obtain that 
    \begin{align*}
        \Omega_\dgr^{\xi} \left( \goi\left( \ZB_\dgr^\Ffunc\right)((b,d)) \right) &= \Omega_\dgr^{\xi} \left( \ZB_\dgr^\Ffunc ((b,d)) \ominus \left( W_\dgr^\Ffunc ((b,d)) \right) \right) \\
        &= \Omega_\dgr^{\xi} \left( \ZB_\dgr^\Ffunc ((b,d)) \cap \left( W_\dgr^\Ffunc ((b,d)) \right)^\perp \right) \\ 
        &=  \ZB_\dgr^{\xi(\Ffunc)} ((b,d)) \cap \left( W_\dgr^{\xi(\Ffunc)} ((b,d)) \right)^\perp \\
        &=\ZB_\dgr^{\xi(\Ffunc)} ((b,d)) \ominus \left( W_\dgr^{\xi(\Ffunc)} ((b,d)) \right) \\
        &= \goi \left(\ZB_\dgr^{\xi(\Ffunc)}\right) ((b,d)).
    \end{align*}
\end{proof}

\subsubsection{Homological Exhaustiveness}

Next, in \Cref{thm: completeness} we  refine \Cref{prop: born and dead multiparameter} in order  establish the \emph{exhaustiveness} of Grassmannian persistence diagrams at the level of homology. We first give an informal statement of this theorem; the precise formulation is given later in this section after introducing necessary definitions.

\begin{theoremnn}[Homological exhaustiveness: informal statement of \Cref{thm: completeness}]
    Let a filtration $\Ffunc : P \to \subcx(K)$ be given. Then, for $(b,d)\in \Seg(P)$, every degree-$\dgr$ homology class that  is born precisely at $b$ and  dies precisely at $d$ is captured by $\goi \left( \ZB_\dgr^\Ffunc \right)((b,d))$.
\end{theoremnn}

We start by recalling the definition of \emph{persistent homology group} from~\cite{robins1999towards,edelsbrunner2010computational} and some related definitions from~\cite{harmonicph}.
\begin{definition}[Persistent homology group]\label{defn: m n and p spaces}
    Let $\Ffunc : P \to \subcx(K)$ be a filtration. For any $(b,d) \in \Seg(P)$, let $\iota_\dgr^{b,d} : H_\dgr (\Ffunc(b)) \to H_\dgr (\Ffunc(d))$ denote the homomorpshim induced by the inclusion $\Ffunc(b) \hookrightarrow \Ffunc(d)$. The~\emph{persistent homology group}, $H_\dgr^{b,d}(\Ffunc)$, of $\Ffunc$ is defined by 
    \[
    H_\dgr^{b,d} (\Ffunc) := \Ima \left(\iota_\dgr^{b,d}\right).
    \]
For $(b,d)\in \Seg(P)$, also define
    \begin{align*}
        M_\dgr^{b,d}(\Ffunc) &:= \sum_{b' < b} \left(\iota_\dgr^{b,d}\right)^{-1} \left(H_\dgr^{b', d} (\Ffunc)\right)\subseteq H_\dgr(\Ffunc(b)).
    \end{align*}
For $(b,d) \in \Seg(P) \setminus \diag(P)$, i.e., $b < d$, define
    \begin{align*}
        N_\dgr^{b,d}(\Ffunc) &:= \sum_{ b \leq d' < d} M_\dgr^{b,d'}(\Ffunc) \subseteq H_\dgr(\Ffunc(b))\\
        &\,\,\text{and}\\
        P_\dgr^{b,d}(\Ffunc) &:= \frac{M_\dgr^{b,d}(\Ffunc)}{ N_\dgr^{b,d}(\Ffunc)}.
    \end{align*}
    When $b\in P$ is a minimal element, we follow the convention that $H_\dgr^{b',d}(\Ffunc) = \{0 \}$ in the definitions of $M_\dgr^{b,d}(\Ffunc)$ and $N_\dgr^{b,d}(\Ffunc)$ above. That is, in the case that $b \in P$ is a minimal element, we have that $M_\dgr^{b,d}(\Ffunc) = \ker \left(\iota_\dgr^{b,d}\right)$ and $N_\dgr^{b,d}(\Ffunc) = \sum_{b\leq d'<d}\ker \left(\iota_\dgr^{b,d'} \right)$.
    
\end{definition}

\begin{remark}
Note that
\begin{enumerate}
    \item The spaces $M_\dgr^{b,d}(\Ffunc)$, $N_\dgr^{b,d}(\Ffunc)$, and $P_\dgr^{b,d}(\Ffunc)$ were originally defined for  $1$-parameter filtration in~\cite[Section 3.1, Definitions 3.6 and 3.9]{harmonicph}. In~\cref{defn: m n and p spaces}, we generalize these definitions without restricting $\Ffunc$ to be a $1$-parameter filtration. In fact, $\Ffunc$ can be a filtration over any finite poset $P$.
    \item The defining formula of $N_\dgr^{b,d}(\Ffunc)$ in~\cite[Definition 3.6]{harmonicph} differs from ours, specifically, it is given by:
    \[
    \sum_{b' < b \leq d' < d} \left(\iota_\dgr^{b, d'}\right)^{-1} \left(H_\dgr^{b', d'} (\Ffunc)\right)
    \]
    However, it is straightforward to verify that their formulation is equivalent to ours. That is,
    \[
    \sum_{ b \leq d' < d} M_\dgr^{b,d'}(\Ffunc) = \sum_{b' < b \leq d' < d} \left(\iota_\dgr^{b, d'}\right)^{-1} \left(H_\dgr^{b', d'} (\Ffunc)\right).
    \]
\end{enumerate}

\end{remark}

\begin{example}[$M_\dgr^{b,d}$, $N_\dgr^{b,d}$ and $P_\dgr^{b,d}$]\label{ex: M N and P}
    Consider the filtration $\Ffunc$ depicted in~\cref{fig: synthesisvsbirth}. We have
    \begin{align*}
        M_1^{p_1, p_4}(\Ffunc) &= \spn \{ [v_1 v_2 + v_2 v_4 - v_1 v_4] \} \\  
        N_1^{p_1, p_4}(\Ffunc) &= \{ 0 \} \\
        \dim \left( P_1^{p_1, p_4}(\Ffunc) \right) &= 1 \\
        M_1^{p_2, p_4}(\Ffunc) &= \spn \{ [v_2 v_4 + v_4 v_5 - v_2 v_5] \} \\
        N_1^{p_2, p_4}(\Ffunc) &= \{ 0 \} \\
        \dim \left( P_1^{p_2, p_4}(\Ffunc) \right) &= 1 \\
        M_1^{p_3, p_4}(\Ffunc) &= H_1(\Ffunc(p_3)) \simeq \R^3 \\
        N_1^{p_3, p_4}(\Ffunc) &= \spn \{ [v_1 v_2 + v_2 v_4 - v_1 v_4], [v_2 v_4 + v_4 v_5 - v_2 v_5] \} \\
        \dim \left( P_1^{p_3, p_4}(\Ffunc) \right) &= 1.
    \end{align*}
\end{example}

    As shown in~\cite[Proposition 3.8]{harmonicph}, the interpretation of the subspaces $M^{b,d}_\dgr(F),N^{b,d}_\dgr(F)$ and $P^{b,d}_\dgr(F)$, in the case of a $1$-parameter filtration, is as follows.
    \begin{itemize}
        \item $M_\dgr^{b,d}(\Ffunc)$ is a subspace of $H_\dgr(K_b)$ consisting of homology classes in $H_\dgr(K_b)$ which  
        \begin{center}
            ``$\big($are born before $b \big)$ or $\big($(born at $b$) and (die at $d$ or earlier)$\big)$.''
        \end{center}
        \item $N_\dgr^{b,d}(\Ffunc)$ is a subspace of $H_\dgr(K_b)$ consisting of homology classes in $H_\dgr(K_b)$ which  
        \begin{center}
            ``$\big($are born before $b\big)$ or $\big($(born at $b$) and (die strictly earlier than $d$)$\big)$.''
        \end{center}
        \item $P_\dgr^{b,d}(\Ffunc)$ is the space of equivalence classes of $\dgr$-dimensional cycles\footnote{More precisely, equivalence classes of homology classes.}  which 
        \begin{center}
            ``are born exactly at $b$ and die exactly at $d$.''
        \end{center}
    \end{itemize}

    We now investigate the interpretation of these spaces $M_\dgr^{b,d}(\Ffunc)$, $N_\dgr^{b,d}(\Ffunc)$, and $P_\dgr^{b,d}(\Ffunc)$ in the case when $\Ffunc$ is a filtration over an arbitrary finite poset $P$. In the 1-parameter case we have the notion that a cycle is born before or exactly at a given index. In the case of a general poset, this situation is more subtle and requires a new definition.  \Cref{rem:interp-def} provides such an interpretation.

\begin{figure}[!t]
    \centering
    \includegraphics[scale=16]{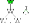}
    \caption{A poset $P$ and a filtration $\Ffunc$ indexed over $P$ is depicted. This filtration serves two purposes: (1) it provides an example for the computation of $M_\dgr^{b,d}(\Ffunc)$, $N_\dgr^{b,d}(\Ffunc)$ and $P_\dgr^{b,d}(\Ffunc)$ (see~\cref{ex: M N and P} for details), and (2) it illustrates the difference between notions of synthesis and birth of a class, as defined in~\cref{defn: synthesis and birth of classes} (see~\cref{ex: synthesisvsbirth} for further explanation).}
    \label{fig: synthesisvsbirth}
\end{figure}

\begin{definition}\label{defn: synthesis and birth of classes}
    Let $\Ffunc : P \to \subcx(K)$ be a filtration. For $b \in P$ and $\dgr\geq 0$ be an integer.
    \begin{enumerate}
        \item We say that a homology class $\gamma \in H_\dgr(\Ffunc(b))$ is \emph{born} at $b$ if \footnote{$\gamma$ cannot be written as a sum of homology classes that were alive earlier than $b$}
        \[
        \gamma \notin \sum_{b' < b} H_\dgr^{b',b} (\Ffunc).
        \]
        \item We say that a nonzero homology class $\gamma \in H_\dgr(\Ffunc(b))$ is \emph{synthesized} at or before $b$ if
        \[
        \gamma \in \sum_{b' < b} H_\dgr^{b',b} (\Ffunc).
        \]
        \item Suppose $\gamma \in H_\dgr(\Ffunc(b))$ is \emph{born} at $b$ and let $d>b$. We say that $\gamma $  \emph{dies} at or before $d$ if \footnote{$\gamma$ merges with a class that was born or synthesized before $b$}
        \[
        \iota_\dgr^{b,d}(\gamma) \in \sum_{b'<b} H_\dgr^{b',d}(\Ffunc).
        \]
    \end{enumerate}
    When $b\in P$ is a minimal element, we follow the convention that $H_\dgr^{b',b}(\Ffunc) = \{ 0 \} = H_\dgr^{b',d}(\Ffunc)$ in the conditions above. Thus, when $b\in P$ is minimal, every nonzero homology class in $H_\dgr(\Ffunc(b))$ is born at $b$. 
\end{definition}

\begin{remark}\label{rem:interp-def}

    The definition above attempts to formally capture the notions of birth and death (and also synthesis) of homology classes. Note that items $1.$ (being born) and $2.$ (being synthesized) in the above definition are mutually exclusive, distinguishing between two distinct ways in which a ``new'' homology class can ``appear'' at $b$. While synthesis allows for combining ``older'' classes (in the sense of belonging to $H_\dgr^{b',b}(\Ffunc)$ for $b'<b$), the notion of being born at $b$ represents a more restrictive way of ``appearing'' at $b$.

\end{remark}

\begin{example}[Difference between synthesis and birth]\label{ex: synthesisvsbirth}
    Consider the filtration $\Ffunc$ depicted in~\cref{fig: synthesisvsbirth}. The homology class $[v_1 v_2 +v_2 v_5 - v_4 v_5 -v_1 v_4 ] \in H_1(\Ffunc(p_3))$ is synthesized at or before $p_3$ as $$[v_1 v_2 +v_2 v_5 - v_4 v_5 -v_1 v_4 ] = \underbrace{[v_1 v_2 + v_2 v_4 - v_1 v_4]}_{\in H_1^{p_1, p_3}(\Ffunc)} - \underbrace{[v_2 v_4 + v_4 v_5 - v_2 v_5]}_{\in H_1^{p_2, p_3}(\Ffunc)} \in \sum_{p<p_3} H_1^{p,p_3} (F).$$ The homology class $[v_2 v_3 + v_3 v_5 - v_2 v_5] \in H_1(\Ffunc(p_3))$ is born at $p_3$ as $$[v_2 v_3 + v_3 v_5 - v_2 v_5] \notin \sum_{p<p_3} H_1^{p,p_3} (F).$$
\end{example}

We now formally interpret the three subspaces of $H_\rho(F(b))$ introduced in \Cref{defn: m n and p spaces} using our new terminology in the definition above. The proof of the following proposition is  postponed to the end of this section.

\begin{proposition}\label{prop: meaning of m n and p in general case}
    For every $b<d\in P$ and integers $\dgr\geq 0$,
    \begin{enumerate}
        \item $M_\dgr^{b,d}(\Ffunc)$ is the subspace of $H_\dgr(\Ffunc(b))$ consisting of those homology classes that 
        \begin{center}
            ``$\big($are synthesized at or before $b\big)$ or $\big($(are born at $b$)  and (die at or before $d$)$\big)$. ''
        \end{center}
        \item $N_\dgr^{b,d}(\Ffunc)$ is the subspace of $H_\dgr(\Ffunc(b))$ consisting of those homology classes that
        \begin{center}
             ``$\big($are synthesized at or before $b\big)$ or $\big($(are born at $b$) and (die strictly before $d$)$\big)$.''
        \end{center}
        \end{enumerate}
        As a consequence,  $P_\dgr^{b,d} (\Ffunc) = \frac{M_\dgr^{b,d}(\Ffunc)}{N_\dgr^{b,d}(\Ffunc)}$ is the vector space of  degree-$\dgr$ homology classes that 
        \begin{center}
            ``are born exactly at $b$ and die exactly at $d$.''
        \end{center}
        
\end{proposition}

We now present the main result of this section, which identifies a canonical surjection  mapping the cycles produced by the Grassmannian persistence diagram $\goi \left( \ZB_\dgr^\Ffunc \right) ((b,d))$ evaluated at the segment $(b,d)$ to the \emph{equivalence classes} of homology classes in $P_\dgr^{b,d} (\Ffunc)$.

\begin{theorem}[Homological exhaustiveness]\label{thm: completeness}
    Let $\Ffunc : P \to \subcx(K)$ be a filtration and $\dgr\geq 0$ be an integer. Then, for any $b<d \in P$, we have that the map
    \begin{align*}
        \goi \left( \ZB_\dgr^\Ffunc\right) ((b,d))  &\to P_\dgr^{b,d}(\Ffunc) = \frac{M_\dgr^{b,d}(\Ffunc)}{N_\dgr^{b,d}(\Ffunc)} \\
        z &\mapsto [z]+ N_\dgr^{b,d}(\Ffunc)
    \end{align*}
    is a surjection, where $[z] = z+\Bfunc_\dgr(\Ffunc(b))$ is the homology class of $z$ in $H_\dgr(\Ffunc(b))$.
\end{theorem}

 We will need the following definition and lemmas when proving~\cref{thm: completeness}. The proofs of the lemmas are postponed to the end of this section.

\begin{definition}
    For a simplicial complex $K$ and any degree $ \dgr \geq 0$, let $$\phi_\dgr^K : \Zfunc_\dgr(K) \to \frac{\Zfunc_\dgr(K)}{  \Bfunc_\dgr(K)} = H_\dgr(K)$$ denote the canonical quotient map. For a filtration $\Ffunc : P \to \subcx(K)$, we define, for $b<d\in P$,
    \begin{align*}
        \tilde{M}_\dgr^{b,d}(\Ffunc) &:= \left(\phi_\dgr^{\Ffunc(b)}\right)^{-1} \left(M_\dgr^{b,d}(\Ffunc)\right) \subseteq \Zfunc_\dgr (K), \\
        \tilde{N}_\dgr^{b,d}(\Ffunc) &:= \left(\phi_\dgr^{\Ffunc(b)}\right)^{-1} \left(N_\dgr^{b,d}(\Ffunc)\right) \subseteq \Zfunc_\dgr (K).
    \end{align*}
    Observe that $\tfrac{\tilde{M}_\dgr^{b,d}(\Ffunc)}{ \Bfunc_\dgr (K)} = M_\dgr^{b,d}(\Ffunc)$ and $\tfrac{\tilde{N}_\dgr^{b,d}(\Ffunc)}{ \Bfunc_\dgr (K)} = N_\dgr^{b,d}(\Ffunc)$.
\end{definition}

\begin{lemma}\label{lem: m tilde n tilde explicit}
    \begin{align*}
        \tilde{M}_\dgr^{b,d}(\Ffunc) &=\sum_{b'<b} \left\{ z \in \Zfunc_\dgr (\Ffunc(b)) \mid \exists z' \in \Zfunc_\dgr(\Ffunc(b'))  \text{ such that } z-z' \in \Bfunc_\dgr(\Ffunc(d)) \right\} \\
        \tilde{N}_\dgr^{b,d}(\Ffunc) &= \sum_{b'<b\leq d' < d} \{ z \in \Zfunc_\dgr (\Ffunc(b)) \mid \exists z' \in \Zfunc_\dgr(\Ffunc(b'))  \text{ such that } z-z' \in \Bfunc_\dgr(\Ffunc(d')) \}
    \end{align*}
\end{lemma}

Recall from~\cref{remark: number of cycles that are born at b and die at d} that the subspace $\sum_{(a,c)<_\times (b,d)} \ZB_\dgr^\Ffunc ((a,c)) \subseteq \ZB_\dgr^\Ffunc ((b,d))$ is denoted by $W_{(b,d)} $. We will continue to use this notation throughout the rest of this section.
\begin{lemma}\label{lem: varphi isom}
    The map 
    \begin{align*}
        \varphi : \frac{\ZB_\dgr^\Ffunc ((b,d))}{W_{(b,d)}} &\to \frac{\tilde{M}_\dgr^{b,d}(\Ffunc)}{\tilde{N}_\dgr^{b,d}(\Ffunc)} \\
                                                           z + W_{(b,d)}     &\mapsto z + \tilde{N}_\dgr^{b,d}(\Ffunc)
    \end{align*}
    is a surjection.
\end{lemma}

\begin{proof}[Proof of~\cref{thm: completeness}]
    We will obtain the claim that the map
    \begin{align*}
        \goi \left( \ZB_\dgr^\Ffunc\right) ((b,d))  &\to P_\dgr^{b,d}(\Ffunc) = \frac{M_\dgr^{b,d}(\Ffunc)}{N_\dgr^{b,d}(\Ffunc)} \\
        z &\mapsto [z]+ N_\dgr^{b,d}(\Ffunc)
    \end{align*}
    is a surjection by writing it as a composition of surjections. Note that, the map
    \begin{align*}
        [\cdot]_{\ZB_\dgr} : \goi\left( \ZB_\dgr^\Ffunc\right) ((b,d)) &\to \frac{\ZB_\dgr^\Ffunc((b,d))}{W_{(b,d)}} \\
        z &\mapsto [z]_{\ZB_\dgr} := z + W_{(b,d)}
    \end{align*}
    is an isomorphism because $\goi\left( \ZB_\dgr^\Ffunc\right) ((b,d)) = \ZB_\dgr^\Ffunc((b,d)) \cap W_{(b,d)}^\perp$. Let
    \[
    \varphi : \frac{\ZB_\dgr^\Ffunc ((b,d))}{W_{(b,d)}} \to \frac{\tilde{M}_\dgr^{b,d}(\Ffunc)}{\tilde{N}_\dgr^{b,d}(\Ffunc)}
    \]
    be the surjection described in~\cref{lem: varphi isom}. Let 
    \[
    \theta : \frac{\tilde{M}_\dgr^{b,d}(\Ffunc)}{\tilde{N}_\dgr^{b,d}(\Ffunc)} \to \frac{\tilde{M}_\dgr^{b,d}(\Ffunc) / \Bfunc_\dgr (\Ffunc(b))}{\tilde{N}_\dgr^{b,d}(\Ffunc) / \Bfunc_\dgr(\Ffunc(b))} = \frac{M_\dgr^{b,d}(\Ffunc)}{N_\dgr^{b,d}(\Ffunc)}
    \]
    be the canonical isomorphism. Then, the composition $\theta \circ \varphi \circ [\cdot]_{\ZB_\dgr}$ 
    \begin{center}
    \begin{tikzcd}
        {\goi\left(\ZB_\dgr^\Ffunc\right)((b,d))} \arrow[r, "{[\cdot]_{\ZB_\dgr}}"] & {\frac{\ZB_\dgr^\Ffunc((b,d))}{W_{(b,d)}}} \arrow[r, "\varphi"] & {\frac{\tilde{M}_\dgr^{b,d}(\Ffunc)}{\tilde{N}_\dgr^{b,d}(\Ffunc)}} \arrow[r, "\theta"] & {\frac{\tilde{M}_\dgr^{b,d}(\Ffunc) / \Bfunc_\dgr (\Ffunc(b))}{\tilde{N}_\dgr^{b,d}(\Ffunc) / \Bfunc_\dgr(\Ffunc(b))} = \frac{M_\dgr^{b,d}(\Ffunc)}{N_\dgr^{b,d}(\Ffunc)}}
    \end{tikzcd}
    \end{center}
    agrees with the candidate surjection since
    \begin{align*}
        \theta \circ \varphi \circ [\cdot]_{\ZB_\dgr} : \goi\left( \ZB_\dgr^\Ffunc \right)((b,d)) &\to P_\dgr^{b,d}(\Ffunc)=\frac{M_\dgr^{b,d}(\Ffunc)}{N_\dgr^{b,d}(\Ffunc)} \\
        z &\mapsto [z] + N_\dgr^{b,d}(\Ffunc).
    \end{align*}
\end{proof}

\paragraph{Proofs of auxiliary results.}

\begin{proof}[Proof of \Cref{prop: meaning of m n and p in general case}]
Observe that Claim $2$ in the statement follows from Claim $1$ as $N_\dgr^{b,d}(\Ffunc) = \sum_{b\leq d' < d} M_\dgr^{b,d'}(\Ffunc)$. Thus, it suffices to prove that $M_\dgr^{b,d}(\Ffunc)$ consists of those degree-$\dgr$ homology classes in $H_\dgr (\Ffunc(b))$ that
        \begin{center}
            ``$\big($are synthesized at or before $b\big)$ or $\big($(are born at $b$)  and (die at or before $d$)$\big)$. ''
        \end{center}
Let $U_\dgr^{b,d}(\Ffunc)$ denote the subset of $H_\dgr (\Ffunc(b))$ consisting of those degree-$\dgr$ homology classes that $\big($are synthesized at or before $b\big)$ or $\big($(are born at $b$) and (die at or before $d$)$\big)$. 

\smallskip
We first prove that $U_\dgr^{b,d}(\Ffunc) \subseteq M_\dgr^{b,d}(\Ffunc)$. Let $\gamma \in U_\dgr^{b,d}(\Ffunc)$, then there are two cases: 
        \begin{itemize}
            \item \underline{Case 1}: $\gamma$ is synthesized at or before $b$, or
            \item \underline{Case 2}: $\gamma$ is born at $b$ and dies at or before $d$,
        \end{itemize}
        We now show that in each of these cases $\gamma$ must be in $M_\dgr^{b,d}(\Ffunc)$.
        \smallskip
        
\noindent        
\underline{Case 1}: Assume $\gamma$ is synthesized at or before $b$. That is,
        \[
        \gamma \in \sum_{b'<b} H_\dgr^{b',b}(\Ffunc).
        \]
        In this case, we have that 
        \[
        \gamma = \sum_{i=1}^k \gamma_i,
        \]
        where $\gamma_i \in H_\dgr^{b_i, b}(\Ffunc)$ for $b_i<b$ and $i=1,\ldots ,k$. Since $\iota_\dgr^{b,d}\circ \iota_\dgr^{b_i, b} = \iota_\dgr^{b_i, d}$, we have that $\iota_\dgr^{b,d}(\gamma_i) \in H_\dgr^{b_i, d}(\Ffunc)$. Therefore, $\gamma_i \in \left(\iota_\dgr^{b,d} \right)^{-1} \left(H_\dgr^{b_i, d} (\Ffunc) \right)$. Thus, 
        \[
        \gamma =\sum_{i=1}^k \gamma_i \in \sum_{b'<b} \left(\iota_\dgr^{b,d} \right)^{-1} \left(H_\dgr^{b', d}(\Ffunc) \right)= M_\dgr^{b,d}(\Ffunc).
        \]

        \noindent
        \underline{Case 2:} Assume $\gamma$ is born at $b$ and dies at or before $d$. In this case, we have that $\gamma \notin \sum_{b'<b} H_\dgr^{b',b}(\Ffunc)$ and $\iota_\dgr^{b,d} (\gamma) \in \sum_{b'<b}H_\dgr^{b',d}(\Ffunc)$. That is, we have that
        \[
        \iota_\dgr^{b,d}(\gamma) = \sum_{i=1}^k \iota_\dgr^{b_i,d} (\gamma_i),
        \]
        where $\gamma_i \in H_\dgr(\Ffunc(b_i))$ for $b_i< b$ and $i=1,\ldots ,k$. Since $\iota_\dgr^{b_i, d} = \iota_\dgr^{b,d}\circ \iota_\dgr^{b_i,b}$, we have
        \begin{align*}
            \iota_\dgr^{b,d}(\gamma) &= \sum_{i=1}^k \iota_\dgr^{b_i,d} (\gamma_i) \\
                                  &= \sum_{i=1}^k \iota_\dgr^{b,d} \circ \iota_\dgr^{b_i,b} (\gamma_i)\\
                                  &= \iota_\dgr^{b,d} \left(\sum_{i=1}^k \iota_\dgr^{b_i,d}(\gamma_i) \right).
        \end{align*}
        Hence, we obtain that $\tilde{\gamma}:= \gamma - \sum_{i=1}^k \iota_\dgr^{b_i,b}(\gamma_i) \in \ker(\iota_\dgr^{b,d})$. Therefore, $\tilde{\gamma} \in M_\dgr^{b,d}(\Ffunc)$. To conclude that $\gamma \in M_\dgr^{b,d}(\Ffunc)$, it is enough to verify that $\sum_{i=1}^k \iota_\dgr^{b_i,b}(\gamma_i) \in M_\dgr^{b,d}(\Ffunc)$. Since $\iota_\dgr^{b,d}(\iota_\dgr^{b_i,b}(\gamma_i)) \in H_\dgr^{b_i,d}(\Ffunc)$, we have that $\sum_{i=1}^k \iota_\dgr^{b_i,b}(\gamma_i) \in M_\dgr^{b,d}(\Ffunc)$. Therefore, $\gamma = \tilde{\gamma} + \sum_{i=1}^k \iota_\dgr^{b_i,b}(\gamma_i) \in M_\dgr^{b,d}(\Ffunc)$. Hence, we conclude that $U_\dgr^{b,d}(\Ffunc) \subseteq M_\dgr^{b,d}(\Ffunc)$

\smallskip
        We now show that $M_\dgr^{b,d}(\Ffunc) \subseteq U_\dgr^{b,d}(\Ffunc)$. Let $\gamma \in M_\dgr^{b,d}(\Ffunc)$. There is nothing to show if $\gamma$ is synthesized at or before $b$. Assume that it is not the case that $\gamma$ is synthesized at or before $b$. Then, $\gamma$ must be born at $b$. In order to conclude that $\gamma \in U_\dgr^{b,d}(\Ffunc)$, we need to show that $\gamma$ dies at or before $d$. That is, we need to verify that 
        \[
        \iota_\dgr^{b,d}(\gamma) \in \sum_{b'<b}H_\dgr^{b',d}(\Ffunc).
        \]
        Since $\gamma \in M_\dgr^{b,d}(\Ffunc)$, we have that 
        \[
        \gamma = \sum_{i=1}^k \gamma_i,
        \]
        where $\gamma_i \in \left(\iota_\dgr^{b,d}\right)^{-1}\left(H_\dgr^{b_i,d} (\Ffunc) \right)$ for $b_i<b$ and $i=1,\ldots ,k$. Therefore, we have that $\iota_\dgr^{b,d}(\gamma_i) \in H_\dgr^{b_i, b}(\Ffunc)$ and
        \[
        \iota_\dgr^{b,d}(\gamma) = \sum_{i=1}^k \iota_\dgr^{b,d}(\gamma_i) \subseteq \sum_{b'<b}H_\dgr^{b', d}(\Ffunc).
        \]
        Thus, $\gamma$ dies at or before $d$. Therefore, we conclude that $M_\dgr^{b,d}(\Ffunc) \subseteq U_\dgr^{b,d}(\Ffunc)$. Hence, we obtain that $M_\dgr^{b,d}(\Ffunc) = U_\dgr^{b,d}(\Ffunc).$ 
\end{proof}

\begin{proof}[Proof of \Cref{lem: m tilde n tilde explicit}]
    As $\tilde{M}_\dgr^{b,d}(\Ffunc) = \left(\phi_\dgr^{\Ffunc(b)}\right)^{-1} \left(M_\dgr^{b,d}(\Ffunc)\right)$, we start by computing $M_\dgr^{b,d}(\Ffunc)$ as follows
    \begin{align*}
        M_\dgr^{b,d}(\Ffunc) &=\sum_{b'<b} \left(\iota_\dgr^{b,d}\right)^{-1} \left(H_\dgr^{b', d} (\Ffunc)\right) \\
                          &= \sum_{b'<b} \left(\iota_\dgr^{b,d}\right)^{-1} \left( \frac{\Zfunc_\dgr(\Ffunc(b')) + \Bfunc_\dgr(\Ffunc(d) )}{\Bfunc_\dgr(\Ffunc(d))} \right) \\
                          &= \sum_{b'<b} \left\{ z + \Bfunc_\dgr(\Ffunc(b)) \in H_\dgr(\Ffunc(b)) \mid z +\Bfunc_\dgr(\Ffunc(d)) \in  \frac{\Zfunc_\dgr(\Ffunc(b')) + \Bfunc_\dgr(\Ffunc(d) )}{\Bfunc_\dgr(\Ffunc(d))}  \right\} \\
                          &= \sum_{b'<b} \{ z + \Bfunc_\dgr(\Ffunc(b)) \in H_\dgr(\Ffunc(b)) \mid \exists z' \in \Zfunc_\dgr (\Ffunc(b')) \text{ such that } z- z' \in \Bfunc_\dgr (\Ffunc(d)) \}
    \end{align*}
    Therefore, as $\phi_\dgr^{\Ffunc(b)}$ is surjective, we have that
    \begin{align*}
        \tilde{M}_\dgr^{b,d}(\Ffunc) &= \left(\phi_\dgr^{\Ffunc(b)}\right)^{-1} \left(M_\dgr^{b,d}(\Ffunc)\right) \\
                                  &= \left(\phi_\dgr^{\Ffunc(b)}\right)^{-1} \left( \sum_{b'<d} \{ z + \Bfunc_\dgr(\Ffunc(b)) \in H_\dgr(\Ffunc(b)) \mid \exists z' \in \Zfunc_\dgr (\Ffunc(b')) \text{ such that } z- z' \in \Bfunc_\dgr (\Ffunc(d)) \}\right) \\
                                  &= \sum_{b'<b} \left( \left(\phi_\dgr^{\Ffunc(b)}\right)^{-1} \{ z + \Bfunc_\dgr(\Ffunc(b)) \in H_\dgr(\Ffunc(b)) \mid \exists z' \in \Zfunc_\dgr (\Ffunc(b')) \text{ such that } z- z' \in \Bfunc_\dgr (\Ffunc(d)) \} \right) \\
                                  &= \sum_{b'<b} \{ z \in \Zfunc_\dgr (\Ffunc(b)) \mid \exists z' \in \Zfunc_\dgr(\Ffunc(b'))  \text{ such that } z-z' \in \Bfunc_\dgr(\Ffunc(d)) \}
    \end{align*}        
    The result for $\tilde{N}_\dgr^{b,d}(\Ffunc)$ follows from the fact that $N_\dgr^{b,d}(\Ffunc) = \sum_{b\leq d'<d} M_\dgr^{b,d'}(\Ffunc)$ as follows:
    \begin{align*}
        \tilde{N}_\dgr^{b,d}(\Ffunc) &= \left(\phi_\dgr^{\Ffunc(b)}\right)^{-1} N_\dgr^{b,d}(\Ffunc) \\
                                  &= \left(\phi_\dgr^{\Ffunc(b)}\right)^{-1} \left( \sum_{b\leq d'<d} M_\dgr^{b,d'}(\Ffunc) \right) \\
                                  &= \sum_{b\leq d'< d} \left( \left(\phi_\dgr^{\Ffunc(b)}\right)^{-1}  M_\dgr^{b,d'}(\Ffunc) \right) \;\;\;\;\;\;\;\;\text{as $\phi_\dgr^{\Ffunc(b)}$ is surjective} \\
                                  &= \sum_{b\leq d'< d} \tilde{M}_\dgr^{b,d'}(\Ffunc) \\
                                  &= \sum_{b\leq d'< d} \sum_{b'<b} \{ z \in \Zfunc_\dgr (\Ffunc(b)) \mid \exists z' \in \Zfunc_\dgr(\Ffunc(b'))  \text{ such that } z-z' \in \Bfunc_\dgr(\Ffunc(d')) \} \\
                                  &= \sum_{b'<b\leq d' < d} \{ z \in \Zfunc_\dgr (\Ffunc(b)) \mid \exists z' \in \Zfunc_\dgr(\Ffunc(b'))  \text{ such that } z-z' \in \Bfunc_\dgr(\Ffunc(d')) \}.
    \end{align*}
\end{proof}

\begin{proof}[Proof of \Cref{lem: varphi isom}]
    Observe that $\ZB_\dgr^\Ffunc ((b,d)) \subseteq \tilde{M}_\dgr^{b,d}(\Ffunc)$. This suggests defining the map
    \begin{align*}
        \psi : \ZB_\dgr^\Ffunc ((b,d)) &\to \frac{\tilde{M}_\dgr^{b,d}(\Ffunc)}{\tilde{N}_\dgr^{b,d}(\Ffunc)} \\
                        z  &\mapsto z + \tilde{N}_\dgr^{b,d}(\Ffunc).
    \end{align*}
    We first show that $\psi$ is surjective. Let $z\in \tilde{M}_\dgr^{b,d}(\Ffunc)$. By~\cref{lem: m tilde n tilde explicit}, we can write $z= z_1 + \cdots + z_m$ where $z_i \in \Zfunc(\Ffunc(b))$ such that for each $i=1,\ldots,m$ there is $z_i' \in \Zfunc_\dgr(\Ffunc(b_i))$ for some $b_i < b$ with $z_i - z_i' \in \Bfunc(\Ffunc(d))$. Let 
    \[
        x := z - \left(z_1' + \cdots + z_m'\right) = \left(z_1 - z_1'\right) + \cdots + \left(z_m - z_m'\right).
    \]
    Observe that $x \in \ZB_\dgr^\Ffunc ((b,d))$ as $(z_i - z_i') \in \ZB_\dgr^\Ffunc((b,d))$ for each $i=1,\ldots,m$. Observe also that, by the description of $\tilde{N}_\dgr^{b,d}(\Ffunc)$ given in~\cref{lem: m tilde n tilde explicit}, $z_i' \in \tilde{N}_\dgr^{b,d}(\Ffunc)$ as $z_i'\in \Zfunc_\dgr(\Ffunc(b_i)) \subseteq \tilde{N}_\dgr^{b,d}(\Ffunc)$. So, $x-z = -(z_1' + \ldots + z_m') \in \tilde{N}_\dgr^{b,d}(\Ffunc)$. Thus, $\psi(x)= x + \tilde{N}_\dgr^{b,d}(\Ffunc) = z + \tilde{N}_\dgr^{b,d}(\Ffunc)$. Hence, $\psi$ is surjective.

    We now show that 
    \[
    W_{(b,d)}\subseteq\ker(\psi).
    \]
    Let $(a,c)<_\times (b,d)$ and let $x \in \ZB_\dgr^\Ffunc((a,c))$. If $a<b$, then, by the description of $\tilde{N}_\dgr^{b,d}(\Ffunc)$ given in~\cref{lem: m tilde n tilde explicit}, $x \in \tilde{N}_\dgr^{b,d}(\Ffunc)$. Therefore $x\in \ker(\psi)$. If $a=b$ and $c<d$, then let $b'$ be any element that is strictly smaller than $b$, i.e., $b'<b$. Then $x-0 \in \Bfunc_\dgr(\Ffunc(c))$ and $0 \in \Zfunc_\dgr(\Ffunc(b'))$. Thus, by~\cref{lem: m tilde n tilde explicit}, $x\in \tilde{N}_\dgr^{b,d}(\Ffunc)$. Hence $x\in \ker(\psi)$. So, $\ZB_\dgr^\Ffunc((a,c)) \subseteq \ker(\psi)$ for any $(a,c)<_\times (b,d)$. Thus, we conclude that $W_{(b,d)} \subseteq \ker(\psi)$. Therefore, we conclude that the map 
    \begin{align*}
        \varphi : \frac{\ZB_\dgr^\Ffunc((b,d))}{W_{(b,d)}} &\to \frac{\tilde{M}_\dgr^{b,d}(\Ffunc)}{\tilde{N}_\dgr^{b,d}(\Ffunc)} \\
        z + W_{(b,d)} &\mapsto z + \tilde{N}_\dgr^{b,d}(\Ffunc)
    \end{align*}
    is a surjection as it can be written as the following composition of surjections: 
    \[
        \frac{\ZB_\dgr^\Ffunc((b,d))}{W_{(b,d)}} \twoheadrightarrow \frac{\ZB_\dgr^\Ffunc((b,d))}{\ker(\psi)} \xrightarrow{\cong} \frac{\tilde{M}_\dgr^{b,d}(\Ffunc)}{\tilde{N}_\dgr^{b,d}(\Ffunc)}.
    \]

\end{proof}

\subsection{Stability of Grassmannian Persistence Diagrams}\label{subsec: stability of grassmannian pd}   
In this section, we establish the (edit distance) stability of Grassmannian persistence diagrams. We now define the category of filtrations over general posets, which is inspired by and extends the construction in~\cite[Definition 4.7]{edit}.

\begin{definition}[Category of filtrations over general posets]
    We define $\mpfil(K)$ to be the category where
    \begin{itemize}
        \item Objects are filtrations $\Ffunc : P \to \subcx(K)$ where $P$ is any finite metric poset,
        \item  A morphism from an object $\Ffunc : P \to \subcx (K)$ to another object $\Gfunc : Q \to \subcx(K)$ is given by a Galois connection $\ladj{f} : P \leftrightarrows S : \radj{f}$ such that $\Gfunc = \Ffunc \circ \radj{f}$.
    \end{itemize}
\end{definition}
\nomenclature[38]{$\mpfil(\cdot)$}{Category of filtrations of a simplicial complex over general posets}

\begin{definition}[Cost of a morphim in $\mpfil(K)$]
     The cost of a morphism $(\ladj{f}, \radj{f})$ in $\mpfil(K)$ is given by $\dis(\ladj{f})$, the distortion of the left adjoint.
\end{definition}

Recall the notion of edit distance from \cref{defn: edit dist}.

\begin{theorem}[Stability]\label{thm: stability of mp gpd}
    Let $\Ffunc$ and $\Gfunc$ be two filtrations in $\mpfil(K)$. Then, for any degree $\dgr\geq 0$, we have
    \[
    d_{\dtfnc\left(C_\dgr^K\right)}^E \left(\goi \left( \ZB_\dgr^\Ffunc \right), \goi \left(\ZB_\dgr^\Gfunc \right)\right) \leq d_{\mpfil(K)}^E (\Ffunc, \Gfunc).
    \]
\end{theorem}

\begin{proof}
    Observe that the assignment
    \[
    \Ffunc \mapsto \ZB_\dgr^\Ffunc
    \]
    is a functor from $\mpfil(K)$ to $\mon\left(C_\dgr^K \right)$ for every degree $\dgr\geq 0$. As a result of this functoriality, we obtain
    \[
    d_{\mon\left(C_\dgr^K\right)}^E \left(\ZB_\dgr^\Ffunc, \ZB_\dgr^\Gfunc\right) \leq d_{\mpfil(K)}^E (\Ffunc, \Gfunc).
    \]
    Thus, combining the above inequality with the one in~\cref{thm: goi functor and stable}, we obtain the desired inequality
    \[
    d_{\dtfnc\left(C_\dgr^K\right)}^E \left(\goi \left( \ZB_\dgr^\Ffunc \right), \goi \left(\ZB_\dgr^\Gfunc \right)\right) \leq d_{\mpfil(K)}^E (\Ffunc, \Gfunc).
    \]
\end{proof}

\begin{remark}[Non-triviality of our stability bounds]
A notion of edit distance between (birth-death induced) signed persistence diagrams was introduced in~\cite{edit}. However, the edit distance between two signed persistence diagrams can be trivial as noted in~\cite[Erratum, Example 10.1]{edit-arxiv}. 
To address this,~\cite[Erratum]{edit-arxiv} presents a modified construction that resolves this issue. Our isometry result,~\cref{thm: isometry of goi}, further extends these improvements by showing that the edit distance between Grassmannian persistence diagrams remains non-trivial and meaningful. Moreover, we establish a lower bound  (see~\cref{thm: edit between signed is lower bound for edit between grassmannian}) for this distance, building on the recent developments in~\cite{edit-arxiv}.
\end{remark}

\subsection{Grassmannian versus Signed Persistence Diagrams}\label{subsec: compare gpd to signed}

In this section, we compare the strength of our construction Grassmannian persistence diagrams with the concepts of signed persistence diagram, whether they are birth-death-induced or rank-induced. In particular, we show that our Grassmannian persistence diagrams recover (see~\cref{prop: goi to signed pd}) and is strictly stronger than signed persistence diagrams (see~\cref{ex: signed pd dont fully capture}).

We first recall the concepts of birth-death-induced signed persistence diagram; see \cite[Definition 8.1]{edit}\footnote{In \cite[Definition 8.1]{edit} the authors consider the more restrictive case when the underlying poset is a lattice.}, and rank-induced signed persistence diagram; see~\cite[Definition 1.5, Corollary 2.6, Equation 5.1]{botnan2022signed}.

\begin{definition}[Birth-death-induced signed persistence diagram]\label{defn: signed bd-pd}
    Let $P$ be any finite poset and let $\Ffunc : P\to \subcx(K)$ be a filtration. For any degree $\dgr\geq 0$, the \emph{degree-$\dgr$ birth-death-induced signed persistence diagram} of $\Ffunc$ is the M\"obius inverse of $\dim \left(\ZB_\dgr^\Ffunc\right) : \Seg(P) \to \Z$ with respect to the product order on $\Seg(P)$.
\end{definition}

\begin{definition}[Rank-induced signed persistence diagram]\label{defn: signed rank-pd}
    Let $P$ be any finite poset and let $\Ffunc : P\to \subcx(K)$ be a filtration. For any degree $\dgr\geq 0$, the \emph{degree-$\dgr$ rank-induced signed persistence diagram} of $\Ffunc$ is the M\"obius inverse of $\beta_\dgr^\Ffunc : \Seg(P) \to \Z$ (persistent Betti numbers; see~\cref{defn: persistent betti rank inv}) with respect to the reverse-inclusion order on $\Seg(P)$.
\end{definition}




\begin{remark}
    In the case of $1$-parameter filtrations, the birth-death induced signed persistence diagram and the rank-induced signed persistence diagram are in fact always nonnegative and they coincide 
    away from the diagonal as shown in~\cite[Section 9.1]{edit}. However, this equivalence breaks down in the multiparameter case, where the two constructions can yield different diagrams (see \cref{ex: signed pd dont fully capture} for an explicit example). In both cases, the diagrams admit an inductive computation based on the defining property of M\"obius inversion, as discussed in~\cref{rem: computing mob inverse}.

\end{remark}

\begin{example}[Grassmannian versus birth-death induced signed persistence diagrams; {\cref{ex: multiparameter gpd} continued}]\label{ex: multiparameter signeg-bd-pd}
    Let $P = \{1<2<3 \}\times \{1<2<3 \}$ be the $3$ by $3$ grid with the product order, as depicted in~\cref{fig: 3x3 filtration}. Let $\Ffunc$ be the filtration over $P$ depicted in~\cref{fig: 3x3 filtration} and  let $\dim \left(\ZB_1^\Ffunc \right)$ be the birth-death function of $\Ffunc$ in degree $1$; see~\cite[Definition 5.6]{edit}.
\begin{align*}
    \dim\left( \ZB_1^\Ffunc \right) : \Seg (P) &\to \Z \\
                    (p, p') &\mapsto \dim \left(\Zfunc_1 (\Ffunc(p)) \cap \Bfunc_1(\Ffunc(p')) \right).
\end{align*}    
The birth-death induced signed persistence diagram of $\Ffunc$ (see \cite[Definition 8.1]{edit})  is given by the classical M\"obius inversion of $\dim\left(\ZB_1^\Ffunc\right)$ with respect to the product order on  $\Seg(P)$. Let $\partial_{(\Seg(P),\prodord)} \left(\dim \left( \ZB_1^\Ffunc\right)\right)$ denote the birth-death induced signed persistence diagram of $\Ffunc$ in degree $1$ which is given explicitly by

    \[
        \partial_{(\Seg(P),\prodord)} \left(\dim \left( \ZB_1^\Ffunc\right)\right)(I) =\begin{cases}
            -1 & \text{ if } I \in \left\{ \begin{multlined}
                \left((1,2), (3,3)\right), \left((2,1), (3,3)\right), \vspace{-1em} \\  \left((2,2), (2,3)\right), \left((2,2), (3,2)\right)
            \end{multlined} \right\}\\
            1 & \text{ if } I \in \left\{ \begin{multlined}
                \left((1,2), (2,3)\right), \left((1,2), (3,2)\right), \vspace{-1em} \\  \left((2,1), (2,3)\right), \left((2,1), (3,2)\right), \left((2,2), (3,3)\right)             
            \end{multlined} \right\}\\
            0 &\text { otherwise.}
        \end{cases}
    \]

\end{example}

\begin{remark}[Comparison of supports]
While the Grassmannian persistence diagram in~\cref{ex: multiparameter gpd} is supported only on four segments, it still determines the birth-death induced signed persistence diagram (described in~\cref{ex: multiparameter signeg-bd-pd}) which is supported on nine segments. Contrary to the scenario illustrated in~\cref{ex: multiparameter gpd} and~\cref{ex: multiparameter signeg-bd-pd}, it is not always true that the support of a Grassmannian persistence diagram is a subset of the support of the corresponding birth-death induced signed persistence diagram. However, we anticipate that the size of the support of the Grassmannian persistence diagram will typically be smaller.
\end{remark}

\begin{remark}[Negative terms in signed persistence diagrams]
    Grassmannian persistence diagrams and birth-death induced signed persistence diagrams differ notably in their ability to capture linear dependencies between cycles. While Grassmannian persistence diagrams retain this information, signed persistence diagrams fail to do so. This inability to capture linear dependence contributes to the presence of negative terms in signed persistence diagrams. For instance, in~\cref{ex: multiparameter gpd}, the birth-death induced signed persistence diagram $\partial_{(\Seg(P),\prodord)} \left(\dim \left( \ZB_1^\Ffunc \right)\right)$ assigns $-1$ to the segment $((1,2),(3,3))$. This occurs because $\partial_{(\Seg(P),\prodord)} \left(\dim \left( \ZB_1^\Ffunc \right)\right))$ assigns $1$ to the segments $((1,2),(2,3))$ and $((1,2),(3,2))$ without capturing that the value $1$ for both segments represents the same cycles, namely $ab-ac+bc$.

\end{remark}

The birth-death-induced signed persistence diagram of a filtration $\Ffunc$ can always be recovered from its Grassmannian persistence diagram. This can be seen from the fact that Grassmannian persistence diagram is a monoidal M\"obius Inverse of the birth-death space associated with $\Ffunc$. Thus, the birth-death spaces can be recovered from the Grassmannian persistence diagram, allowing the computation of their dimension. Subsequently, M\"obius inversion can be applied to obtain the birth-death induced signed persistence diagram. In the following proposition, we give a more concrete way of obtaining birth-death induced signed persistence diagrams from  Grassmannian persistence diagrams.

\begin{proposition}[Obtaining  signed persistence diagrams from  Grassmannian diagrams]\label{prop: goi to signed pd}
    Let $\Ffrak : R \to \gr(V)$ be a monotone space function. Let $\dim \Ffrak : R \to \Z$ be given by $\dim \Ffrak (r) = \dim (\Ffrak(r))$ for every $r\in R$. Let $\partial_R (\dim \Ffrak)$ be the M\"obius inverse of $\dim \Ffrak$ and let $\goi (\Ffrak)$ be the Orthogonal Inverse of $\Ffrak$. Then, 
    \begin{equation}\label{eqn: recursive}
        \partial_R (\dim \Ffrak) (r) = \dim \left(\goi(\Ffrak)(r)\right) - \left ( \sum_{r'< r} \partial_R (\dim \Ffrak) (r') - \dim \sum_{r'<r} \goi (\Ffrak )(r')\right).
    \end{equation}

\end{proposition}

\begin{proof}
    Observe that we have the following equations
    \begin{align*}
        \partial_R (\dim \Ffrak) (r) &= \dim (\Ffrak (r)) - \sum_{r'<r} \partial_R (\dim \Ffrak) (r') \\
        \dim (\Ffrak(r)) &= \dim (\goi (\Ffrak)(r))+ \dim \sum_{r'<r} \goi (\Ffrak )(r') .
    \end{align*}
    Thus, we obtain
    \begin{align*}
        \partial_R (\dim \Ffrak) (r) &= \dim (\goi (\Ffrak)(r))+ \dim  \sum_{r'<r} \goi (\Ffrak )(r') - \sum_{r'<r} \partial_R (\dim \Ffrak) (r') \\
        &= \dim (\goi (\Ffrak)(r)) - \left ( \sum_{r'<r} \partial_R (\dim \Ffrak) (r') -  \dim  \sum_{r'<r} \goi (\Ffrak )(r') \right )
    \end{align*}
\end{proof}

Note that~\cref{prop: goi to signed pd} provides us with a recursive way of obtaining the birth-death induced signed persistence diagram of a filtration from its Grassmannian persistence diagram. Let $\Ffunc : P \to \subcx(K)$ be a filtration and let $\Ffrak := \ZB_\dgr^\Ffunc : \Seg(P) \to \gr\left( c_\dgr^K\right)$ be the $\dgr$-th birth death spaces of $\Ffunc$. For any minimal element $(m,m')\in \Seg(P)$, \cref{eqn: recursive} boils down to following.
\[
\partial_{(\Seg(P),\prodord)} (\dim \Ffrak) ((m,m')) = \dim \goi(\Ffrak)((m,m')).
\]

Then for an arbitrary $(p,p')\in \Seg(P)$, the right hand side of~\cref{eqn: recursive} only contain terms coming from the Grassmannian persistence diagram and values of birth-death induced signed persistence diagram at the segments which are strictly smaller than $(p,p')$, computed for during the recursive process.

\begin{remark}[Interpretability]\label{rem:interp}
    While Grassmannian persistence diagrams offer a straightforward interpretation in the sense of~\cref{prop: born and dead multiparameter} and~\cref{thm: completeness}, signed persistence diagrams fail to provide such a clear interpretation. The presence of negative terms is one factor contributing to the complexity of interpreting signed persistence diagrams. However, even positive terms can pose challenges in interpretation. For instance, in~\cref{ex: multiparameter gpd}, the birth-death induced signed persistence diagram $\partial_{(\Seg(P),\prodord)} \dim \left(\left( \ZB_1^\Ffunc \right)\right)$ assigns $1$ to the segment $((2,2),(3,3))$. Yet, this assignment does not capture the birth and death of a cycle for the segment $((2,2),(3,3))$. Moreover, the following example also permits concluding that, even if (either rank or birth-death induced) signed persistence diagram is non-negative everywhere, the number assigned to a segment $(b,d)$ does not necessarily correspond to the number of cycles or homology classes that are born at $b$ and die at $d$ as demonstrated in the following example.
\end{remark}

Observe that~\cref{prop: goi to signed pd} above indicates that Grassmannian persistence diagrams are stronger than birth-death induced signed persistence diagrams (and, hence, also stronger than rank induced persistence diagrams). The following example shows that Grassmannian persistence diagrams are, in fact, strictly stronger than signed persistence diagrams.

\begin{example}[Comparison of interpretability and strength]\label{ex: signed pd dont fully capture}
    Let $P = \{ p_1, p_2, p_3, p_4 \}$ be the poset with $p_1<p_3$, $p_2<p_3$, and $p_3<p_4$  depicted in~\cref{fig: filtrations for nonfully capturing}. Let $\Ffunc$ and $\Gfunc$ be the two filtrations over $P$  depicted in~\cref{fig: filtrations for nonfully capturing}. By~\cref{thm: completeness}, for the filtrations $\Ffunc$ and $\Gfunc$, we obtain the number of degree-$1$ homology classes that are born at $p_i$ and die at $p_j$ for each $(p_i,p_j)\in \Seg(P)$ as follows:
    \begin{align*}
        \dim \left(\goi \left( \ZB_1^\Ffunc  \right) \right)((p_1,p_4)) &=1 & \dim \left(\goi \left( \ZB_1^\Gfunc  \right) \right)((p_1,p_4)) &=1 \\
        \dim \left(\goi \left( \ZB_1^\Ffunc  \right) \right)((p_2,p_4)) &=1 & \dim \left(\goi \left( \ZB_1^\Gfunc  \right) \right)((p_2,p_4)) &=1 \\
        \dim \left(\goi \left( \ZB_1^\Ffunc \right) \right)((p_3,p_4))  &=2 & \dim \left(\goi \left( \ZB_1^\Gfunc  \right) \right)((p_3,p_4)) &=1. 
    \end{align*}
    In addition to the computation above, it is also clear from~\cref{fig: filtrations for nonfully capturing} that, while there is only one cycle in $\Gfunc$ that is born at $p_3$ and  dying at $p_4$, there are two linearly independent cycles in $\Ffunc$ that are born at $p_3$ and dying at $p_4$. However, the degree-$1$ birth-death induced signed persistence diagrams,~\cref{defn: signed bd-pd}, of $\Ffunc$ and $\Gfunc$ are equal (and non-negative). Hence, even when a birth-death signed persistence diagram is non-negative everywhere, the value determined by the signed persistence diagram on a segment $(b,d)\in \Seg(P)$ does not necessarily correspond to the number of cycles that are born at $b$ and die at $d$.

    Moreover, rank-induced signed persistence diagrams,~\cref{defn: signed rank-pd}, even when non-negative everywhere, also do not always capture the number of cycles that are born at $b$ and die at $d$, for  $(b, d)\in \Seg(P)$. To see this, observe that $\Ffunc$ and $\Gfunc$ have the same degree-$1$ rank invariants, and therefore the same rank-induced signed persistence diagrams (which are non-negative in this example). Below, we list the values of both the rank-induced and birth-death-induced (signed) persistence diagrams for the segments where these diagrams are nonzero:
    \begin{align*}
        \partial_{(\Seg(P),\prodord)} \left( \dim \left(\ZB_1^\Ffunc\right)\right) ((p_1,p_4)) &= \partial_{(\Seg(P),\prodord)} \left( \dim \left(\ZB_1^\Gfunc\right)\right) ((p_1,p_4)) =1 \\
        \partial_{(\Seg(P),\prodord)} \left( \dim \left(\ZB_1^\Ffunc\right)\right) ((p_2,p_4)) &= \partial_{(\Seg(P),\prodord)} \left( \dim \left(\ZB_1^\Gfunc\right)\right) ((p_2,p_4)) =1 \\
        \partial_{(\Seg(P),\prodord)} \left( \dim \left(\ZB_1^\Ffunc\right)\right) ((p_3,p_4)) &= \partial_{(\Seg(P),\prodord)} \left( \dim \left(\ZB_1^\Gfunc\right)\right) ((p_3,p_4)) =1
    \end{align*}

\begin{align*}
        \partial_{(\Seg(P),\leq_\supseteq)} \left( \beta_1^\Ffunc\right) ((p_1,p_3)) &= \partial_{(\Seg(P),\leq_\supseteq)} \left( \beta_1^\Ffunc\right) ((p_1,p_3)) =1 \\
        \partial_{(\Seg(P),\leq_\supseteq)} \left( \beta_1^\Ffunc\right) ((p_2,p_3)) &= \partial_{(\Seg(P),\leq_\supseteq)} \left( \beta_1^\Ffunc\right) ((p_2,p_3)) =1 \\
        \partial_{(\Seg(P),\leq_\supseteq)} \left( \beta_1^\Ffunc\right) ((p_3,p_3)) &= \partial_{(\Seg(P),\leq_\supseteq)} \left( \beta_1^\Ffunc\right) ((p_3,p_3)) =1
\end{align*}

\end{example}

We now revisit some of the key definitions from the modified construction in~\cite{edit-arxiv} and show that the (modified) edit distance between birth-death induced signed persistence diagrams of two filtrations is a lower bound for the edit distance between their Grassmannian persistence diagrams. Although their construction assumes the underlying poset is a finite lattice, we adapt their definitions to a more general setting by using an arbitrary finite poset.

\begin{definition}[Category of integral functions~{\cite[Definition 10.3]{edit-arxiv}}]
    We define the $\fnc$ to be the category where
    \begin{itemize}
        \item Objects are functions $\partial_R(\eta): R \to \Z$ where $R$ is any finite metric poset and $\eta : R \to \Z$ is an order-preserving function,
        \item A morphism from an object $\partial_R(\eta_1) : R_1 \to \Z$ to $\partial_Q(\eta_2) : R_2 \to \Z$ is given by a Galois connection $\ladj{f} : R_1 \leftrightarrows R_2 : \radj{f}$ such that $\left(\ladj{f}\right)_\sharp \left( \partial_{R_1}(\eta_1) \right) = \partial_{R_2}( \eta_2 )$. 
    \end{itemize}
\end{definition}

\begin{definition}[Cost of a morphism in $\fnc$]
    The cost of a morphism $(\ladj{f}, \radj{f})$ in $\fnc$ is given by $\dis(\ladj{f})$, the distortion of the left adjoint.
\end{definition}

In the following, we prove that the edit distance between birth-death induced signed persistence diagrams of two filtrations is a lower bound for the edit distance between their Grassmannian persistence diagrams. This demonstrates that Grassmannian persistence diagrams are at least as discriminative as birth-death induced signed persistence diagrams.
\begin{theorem}[Lower bound]\label{thm: edit between signed is lower bound for edit between grassmannian}
    Let $\Ffunc : P \to \subcx(K)$ and $\Gfunc : Q \to \subcx(K)$ be two filtrations in $\mpfil(K)$. Then, for any degree $\dgr\geq 0$, we have
    \[
    d_{\fnc}^E \Big( \partial_{(\Seg(P), \prodord)} \left( \dim \left( \ZB_\dgr^\Ffunc \right) \right), \partial_{(\Seg(Q), \prodord)} \left( \dim \left( \ZB_\dgr^\Gfunc \right) \right) \Big) \leq d_{\dtfnc\left(C_\dgr^K\right)}^E \Big(\goi \left( \ZB_\dgr^\Ffunc \right), \goi \left(\ZB_\dgr^\Gfunc \right)\Big).
    \]
\end{theorem}

\begin{remark}[Improved discriminating power]\label{rmk:kracht}
Observe that the filtrations depicted in~\cref{fig: filtrations for nonfully capturing} have the same (both rank and birth-death induced) signed persistence diagrams, as demonstrated in~\cref{ex: signed pd dont fully capture}. Consequently, the term on the left in~\cref{thm: edit between signed is lower bound for edit between grassmannian} equals to 0. On the other hand, the Grassmannian persistence diagrams of these filtrations are distinct and the edit distance between them remains positive. This indicates that the inequality in~\cref{thm: edit between signed is lower bound for edit between grassmannian} can indeed be strict. Hence, Grassmannian persistence diagrams are strictly more discriminative than signed persistence diagrams.
\end{remark}

To prove~\cref{thm: edit between signed is lower bound for edit between grassmannian}, we first revisit some constructions and results from from~\cite{edit-arxiv}.

\begin{definition}[Category of monotone integral functions~{\cite[Definition 5.5]{edit-arxiv}}]
    We define $\monfnc$ to be the category where
    \begin{itemize}
        \item Objects are order-preserving functions $\nu : R \to \Z$ where $R$ is any finite metric poset,
        \item A morphism from an object $\nu_1 : R_1 \to \Z$ to $\nu_2 : R_2 \to \Z$ is given by a Galois connection $\ladj{f} : R_1 \leftrightarrows R_2 : \radj{f}$ such that $\nu_2 = \left(\radj{f}\right)^\sharp(\nu_1) = \nu_1 \circ \radj{f}$.
    \end{itemize}
\end{definition}

\begin{definition}[Cost of a morphism in $\monfnc$]
    The cost of a morphism $(\ladj{f}, \radj{f})$ in $\monfnc$ is given by $\dis(\ladj{f})$, the distortion of the left adjoint.
\end{definition}

\begin{lemma}[{\cite[Lemma 10.9]{edit-arxiv}}]\label{lemma: edit-arxiv isometry}
    Let $\nu_1 : R_1 \to \Z$ and $\nu_2 : R_2 \to \Z$ be objects in $\monfnc$. Then,
    \[
    d_{\monfnc}^E(\nu_1, \nu_2) = d_{\fnc}^E (\partial_{R_1}(\nu_1), \partial_{R_2}(\nu_2)).
    \]
\end{lemma}

\begin{proof}[Proof of~\cref{thm: edit between signed is lower bound for edit between grassmannian}]
    Let $\Ffunc : P \to \subcx(K)$ and $\Gfunc : Q \to \subcx(K)$ be filrations in $\mpfil(K)$.
    Observe that 
    \begin{align*}
        \dim : \mon(C_\dgr^K) &\to \monfnc \\
                       \mathfrak{F} &\mapsto \dim(\mathfrak{F})
    \end{align*} 
    is a functor that preservers the cost of morphisms. Hence, for any $\Ffrak$ and $\Gfrak$ in $\mon\left(C_\dgr^K\right)$, we have that 
    \[
    d_{\monfnc}^E(\dim(\Ffrak), \dim(\Gfrak)) \leq d_{\mon\left( C_\dgr^K \right)} \left(\Ffrak, \Gfrak \right).
    \]
    In particular, we have
    \[
    d_{\monfnc}^E\left(\dim\left(\ZB_\dgr^\Ffunc\right), \dim\left(\ZB_\dgr^\Gfunc \right) \right)\leq d_{\mon\left( C_\dgr^K \right)} \left(\ZB_\dgr^\Ffunc, \ZB_\dgr^\Gfunc\right).
    \]
    Combining the inequality above with our isometry result (\cref{thm: isometry of goi}) and~\cref{lemma: edit-arxiv isometry}, we obtain the desired in equality:
    \begin{align*}
        d_{\fnc}^E \left( \partial_{(\Seg(P), \prodord)} \left( \dim \left( \ZB_\dgr^\Ffunc \right) \right), \partial_{(\Seg(Q), \prodord)} \left( \dim \left( \ZB_\dgr^\Gfunc \right) \right) \right) &= d_{\monfnc}^E \left( \dim \left( \ZB_\dgr^\Ffunc \right) , \dim \left( \ZB_\dgr^\Gfunc\right)\right) \\
        &\leq d_{\mon \left( C_\dgr^K \right)}^E \left( \ZB_\dgr^\Ffunc, \ZB_\dgr^\Gfunc \right) \\
        &= d_{\dtfnc \left( C_\dgr^K \right)}^E \left( \goi \left( \ZB_\dgr^\Ffunc\right)  , \goi \left( \ZB_\dgr^\Gfunc \right) \right).
    \end{align*}
\end{proof}

\section{M\"obius Homology and Grassmannian Persistence Diagrams}\label{sec: pers mob hom}

\medskip

In this section, we establish a direct connection with the notion of  Möbius Homology from~\cite{patel-skraba-mobius} which we apply in order to associate a family \{$H_s^\downarrow \ZB_\rho^F:\Seg(P)\to\Vec \}_{s \in \N}$ of modules to the $\rho$-th birth-death space of a given simplicial filtration $\Ffunc:P\to \subcx(K)$. This family is indexed by a nonnegative integer $s$, which we refer to as the \emph{stratum}. We prove that the stratum-$0$ persistent Möbius homology of a filtration coincides with the Grassmannian persistence diagram of the filtration.

\begin{theorem}\label{prop: grade 0 pers mob hom isomorphic to gpd}
Let $\Ffunc : P \to \subcx(K)$ be a filtration. Let $\dgr\geq 0$ be an integer. Then, for every $(b,d) \in \Seg(P)$, we have
    \[
    H_0^\downarrow \ZB_\dgr^\Ffunc ((b,d)) \cong \goi \left(ZB_\dgr^\Ffunc\right) ((b,d)).
    \]
\end{theorem}

\begin{remark}

In contrast to $H_0^\downarrow \ZB_\dgr^\Ffunc ((b,d))$, our proposed approach, $\goi \left(ZB_\dgr^\Ffunc\right) ((b,d))$, always yields a subspace of the $\dgr$-th chain space of $K$ that consists of cycles. This essential feature, enabled by our deliberate incorporation of an inner product structure, not only distinguishes our method but also serves as a cornerstone of our framework. More precisely, \cref{prop: born and dead multiparameter} utilizes this intrinsic feature to  establish the interpretability of Grassmannian persistence diagrams.

\end{remark}

The preceding theorem and \cref{prop: born and dead multiparameter} immediately yield the following corollary, which offers a novel interpretation of the stratum‑0 of persistent Möbius homology—a perspective that is absent from \cite{patel-skraba-mobius}.

\begin{corollary}\label{coro:MH-ex}
    Let $\Ffunc : P \to \subcx(K)$ be a filtration. Let $\dgr\geq 0$ be an integer. Then, for every $(b,d) \in \Seg(P)$, $\dim\left(H_0^\downarrow \ZB_\dgr^\Ffunc ((b,d))\right)$ equals the number of independent $\dgr$-cycles that are born at $b$ and die at $d$.

\end{corollary}

In~\cref{subsec: simplicial cosheaves,subsec: order cosheaf and mob hom of modules}, we recall the necessary background for constructing the stratum-$s$ M\"obius homology $H_s^\downarrow N$ of a module $N:R\to \Vec$.\footnote{Note that \cref{prop: grade 0 pers mob hom isomorphic to gpd} applies this construction to the $\rho$-th birth-death functor of a given filtration $\Ffunc$ so that the poset $R$ will be chosen as $\Seg(P)$ whereas $\ZB_\rho^\Ffunc$ will play the role of the module $N$.} There is a slight (but eventually immaterial) difference with the approach taken by Patel and Skraba that we explain in \cref{subsec: pers mob hom of modules}.

\begin{framed}
    In the following subsections, we will utilize three different simplicial complexes, denoted by $K$ and $L'\subseteq L$. Consistent with the notation used in previous section, the complex $K$ will denote a simplical complex that is filtered i.e., $\Ffunc : P \to \subcx(K)$ is a filtration. On the other hand, $L$ (and $L'$) will serve as the domain of simplicial cosheafs. Specifically, $L$ will become the order complex of $\Seg(P)$ when defining (persistent) Möbius homology in~\cref{defn: mob chain comp and mob hom,defn: pers mob hom of filtrations}.
\end{framed}

\subsection{Simplicial Cosheaves}\label{subsec: simplicial cosheaves}
Let $L$ be a finite oriented simplicial complex. Consider the poset category of $L$ where there is a unique morphism $\tau \to \sigma$ whenever $\tau$ is a coface of $\sigma$ (i.e., $\tau \geq \sigma$). We use the notation $\tau >_1 \sigma$ to mean $\tau$ is a coface of $\sigma$ and $\dim (\tau) = \dim(\sigma)+1$. 
 
\begin{definition}[Simplicial cosheaf]
    A \emph{simplicial cosheaf} over $L$ is any functor $$\underline{A} : L \to \Vec.$$
\end{definition}

\begin{definition}[$s$-th chain space of a simplicial cosheaf]
    For any stratum $s\geq 0$ (i.e. $s\in \N$), the \emph{$s$-th chain space} of a simplicial cosheaf $\underline{A} : L\to \Vec$ is 
    \[
    C_s(L, \underline{A}) := \bigoplus_{\sigma: \dim(\sigma) = s} \underline{A}(\sigma)
    \]
\end{definition}

\begin{definition}[Boundary operator between chain complexes of a simplicial cosheaf]
    The \emph{boundary operator} $\partial_s : C_s(L; \underline{A}) \to C_{s-1}(L; \underline{A})$ is defined summand-wisely as follows. For an $s$-simplex $\tau$, let 
    $$\iota_\tau : \underline{A}(\tau) \hookrightarrow C_s(L, \underline{A}) = \bigoplus_{\sigma: \dim(\sigma) = s} \underline{A}(\sigma)$$
    be the canonical inclusion of the summand $\underline{A}(\tau)$ into the direct sum $\bigoplus_{\sigma: \dim(\sigma) = s}\underline{A}(\sigma)$ and let 
    $$\pi_\tau : C_s(L, \underline{A}) = \bigoplus_{\sigma: \dim(\sigma) = s}\underline{A}(\sigma) \twoheadrightarrow \underline{A}(\tau)$$ be the canonical surjection from the direct sum to its summand $\underline{A}(\tau)$. The restriction of $\partial_s$ to $\tau$ is defined as
    \[
    \partial_s|_\tau := \sum_{\sigma: \tau>_1 \sigma} [\tau:\sigma]\cdot \big(\iota_\sigma \circ \underline{A}(\tau \geq \sigma)\circ \pi_\tau\big),
    \]
    where $[\tau:\sigma] =1$ if the restriction of the orientation of $\tau$ to $\sigma$ agrees with the orientation of $\sigma$ and  $[\tau: \sigma] = -1$ otherwise.
    Then, the  boundary operator is defined as the sum
    \[
    \partial_s := \sum_{\tau: \dim(\tau)=s} \partial_s|_\tau.
    \]
\end{definition}

The boundary operator $\partial_s : C_s(L; \underline{A}) \to C_{s-1}(L; \underline{A})$ satisfies $\partial_{s-1} \circ \partial_s = 0$. Hence, we obtain the chain complex

\[
\big(C_\bullet(L, \underline{A}),\partial_\bullet\big) : \cdots \xrightarrow[]{\partial_3} C_2(L, \underline{A}) \xrightarrow[]{\partial_2} C_1(L, \underline{A}) \xrightarrow[]{\partial_1} C_0(L, \underline{A}) \xrightarrow[]{\partial_0} 0.
\]

\begin{definition}[Cosheaf homology of a simplicial cosheaf]
    The $s$-th homology of the chain complex $(C_\bullet(L, \underline{A}), \partial_\bullet)$ is called the \emph{$s$-th cosheaf homology} of $\underline{A}$ and  is denoted by $H_s (L, \underline{A})$. 
\end{definition}

Let $L' \subseteq L$ be a subcomplex, and let $\underline{B} : L' \to \Vec$ be the restriction of the cosheaf $\underline{A}: L \to \Vec$ to $L'$. Then, we obtain the associated relative chain complex
\[
\big(C_\bullet(L, L'; \underline{A}),\partial_\bullet\big) : \cdots \xrightarrow[]{\partial_3} \frac{C_2(L; \underline{A})}{C_2(L'; \underline{B})} \xrightarrow[]{\partial_2} \frac{C_1(L; \underline{A})}{C_1(L'; \underline{B})} \xrightarrow[]{\partial_1} \frac{C_0(L; \underline{A})}{C_0(L'; \underline{B})} \xrightarrow[]{\partial_0} 0.
\]

\begin{definition}[Relative cosheaf homology]\label{defn: relative cosheaf homology}
    The \emph{$s$-th relative cosheaf homology} of $\underline{A}$ with respect to the subcomplex $L' \subseteq L$, denoted $H_s (L, L'; \underline{A})$ is the $s$-th homology of the chain complex $\big(C_\bullet(L, L'; \underline{A}), \partial_\bullet\big)$.
\end{definition}

\begin{figure}
\includegraphics[width=\linewidth]{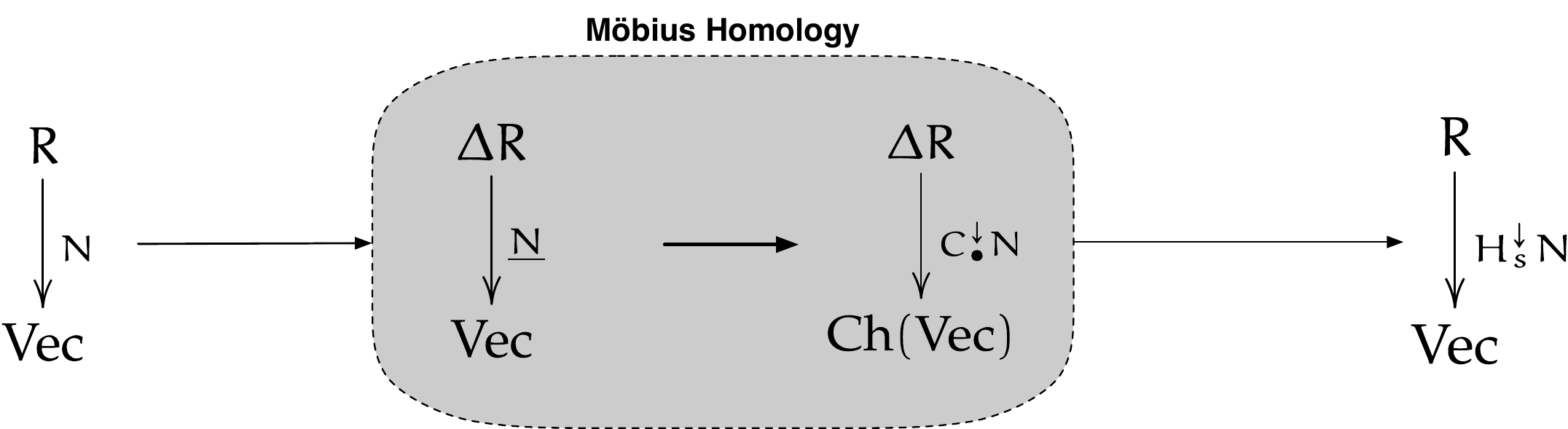}
\caption{M\"obius homology applies to modules $N:R\to \Vec{}$. See \cref{defn: mob chain comp and mob hom}.}\label{fig:mh}
\end{figure}
\subsection{The order Cosheaf and M\"obius Homology of Modules $N:R\to \Vec{}$}\label{subsec: order cosheaf and mob hom of modules}

A chain of length $n$ in a poset $R$ is a sequence $a_0 < a_1 < \cdots < a_n$ of $n+1$ distinct elements of $R$. A subchain of a chain $a_0 < a_1 < \cdots < a_n$ is chain obtained by deleting any number of its elements.

\begin{definition}[Order complex]
    Let $R$ be a finite poset. The \emph{order complex} of $R$, denoted $\Delta R$, is the simplicial complex whose $i$-simplices are chains of length $i$. A simplex $a_0 < a_1 < \cdots < a_i$ is a coface of a simplex $b_0 < b_1 < \cdots < b_j$ if the latter is a subchain of the former.
\end{definition}
\nomenclature[50]{$\Delta R$}{The order complex of a poset $R$}

Let $\min$, $\max$ $: \Delta R \to R$ be the functions that respectively assign to a simplex $\sigma$ of $\Delta R$ the minimal and maximal elements in the  chain corresponding to $\sigma$. 

\begin{definition}[Order cosheaf]
    Let $N : R \to \Vec$ be a module. The \emph{order cosheaf} of $N$ is the simplicial cosheaf $\underline{N} : \Delta R \to \Vec$ defined as the following composition
    \[
    \Delta R \xrightarrow[]{\min} R \xrightarrow[]{N} \Vec.
    \]
\end{definition}

\begin{definition}[Lower complex]
    Let $R$ be a finite poset. For any $r\in R$, the \emph{lower complex} of $r$ is the subcomplex $\Delta R_{\leq r} := \{ \sigma \in \Delta R \mid \max(\sigma) \leq r \}$ of $\Delta R$. The \emph{strict lower complex} of $r$ is the subcomplex $\Delta R_{\leq r} := \{ \sigma \in \Delta R \mid \max(\sigma) < r \}$ of $\Delta R$.
\end{definition}
\nomenclature[51]{$\Delta R_{\leq r}$}{The lower complex of $r\in R$.}
\nomenclature[52]{$\Delta R_{< r}$}{The strict lower complex of $r\in R$.}

\begin{definition}[M\"obius chain complex and M\"obius homology~{\cite[Definition 3.8]{patel-skraba-mobius}}]\label{defn: mob chain comp and mob hom}
    The \emph{M\"obius chain complex module} of $N : R \to \Vec$, denoted $C_\bullet^\downarrow N:R\to \mathrm{Ch}(\Vec{})$, is the functor defined as follows:
    \begin{align*}
        C_\bullet^\downarrow N(r) &:= C_\bullet \big(\Delta R_{\leq r}, \Delta R_{<r}; \underline{N}\big) \,\,\,\mbox{for $r\in R$}, \\
        &\mbox{and}\\
        C_\bullet^\downarrow N(r\leq r') &:= C_\bullet \big(\Delta R_{\leq r}, \Delta R_{<r}; \underline{N}\big) \to C_\bullet \big(\Delta R_{\leq r'}, \Delta R_{<r'}; \underline{N}\big)\,\,\,\mbox{for $r\leq r' \in R$}
    \end{align*}

    where the morphism on the r.h.s. is the one induced by the inclusion of pair of complexes $$\big(\Delta R_{\leq r}, \Delta R_{<r}\big) \hookrightarrow \big(\Delta R_{\leq r'}, \Delta R_{<r'}\big).$$ Applying the homology functor to $C_\bullet^\downarrow N$, we obtain another functor for every stratum $s\geq 0$, called the \emph{stratum-$s$  M\"obius homology module}\footnote{Or \emph{$s$-th stratum   M\"obius homology module}.} of $N$, denoted by $H_s^\downarrow N : R \to \Vec$. See  \cref{fig:mh}.
\end{definition}

\nomenclature[53]{$C_\bullet^\downarrow N$}{The M\"obius chain complex of a module $N$.}
\nomenclature[54]{$H_s^\downarrow N$}{The stratum-$s$ (or $s$-th stratum) M\"obius homology of the module $N$.}

\subsection{Persistent M\"obius Homology of Filtrations and  Proof of \cref{prop: grade 0 pers mob hom isomorphic to gpd}}\label{subsec: pers mob hom of filtrations and relation}

In this section,  we  prove \cref{prop: grade 0 pers mob hom isomorphic to gpd} --the main result in  \cref{sec: pers mob hom}-- stating that the stratum-$0$ persistent Möbius homology of a filtration produces vector spaces isomorphic to those determined by the Grassmannian persistence diagram of the filtration.

\begin{definition}[Persistent M\"obius homology of filtrations]\label{defn: pers mob hom of filtrations}
    Let $\Ffunc : P \to \subcx(K)$ be a filtration. We define the \emph{stratum-$s$, degree-$\dgr$ persistent M\"obius homology module} of $\Ffunc$ as the  $H_s^\downarrow \ZB_\dgr^\Ffunc : \Seg(P) \to \Vec$. See \cref{fig:mh-F}. We also refer to $H_s^\downarrow \ZB_\dgr^\Ffunc$ as the $(\rho,s)$-persistent M\"obius homology of $\Ffunc$. 
\end{definition}

\begin{figure}
\includegraphics[width=\linewidth]{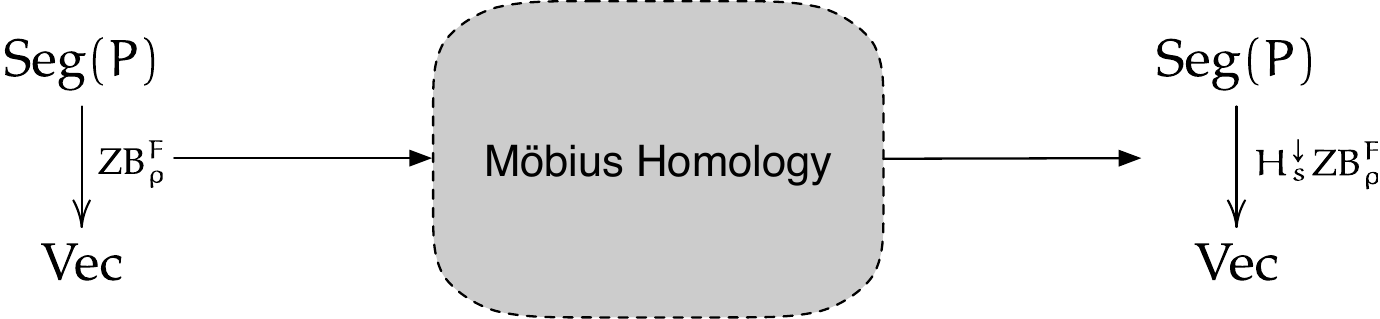}
\caption{The $(\rho,s)$-\emph{persistent} M\"obius homology of a filtration  $\Ffunc:P\to \subcx{(K)}$ is defined as the $s$-th M\"obius homology of the module $\ZB^F_\rho:\Seg(P)\to \Vec$. See \cref{defn: pers mob hom of filtrations}.}\label{fig:mh-F}
\end{figure}

\begin{proof}[Proof of \cref{prop: grade 0 pers mob hom isomorphic to gpd}]
    By~\cref{defn: mob chain comp and mob hom}, $H_0^\downarrow \ZB_\dgr^\Ffunc ((b,d))$ is the $0$-th homology of the relative chain complex $C_\bullet \left(\Delta \Seg(P)_{\leq (b,d)}, \Delta \Seg(P)_{< (b,d)}; \underline{\ZB_\dgr^\Ffunc} \right)$. Unraveling the definitions, observe that 
    \begin{enumerate}
        \item $C_0 := C_0 \left(\Delta \Seg(P)_{\leq (b,d)}, \Delta \Seg(P)_{< (b,d)}; \underline{\ZB_\dgr^\Ffunc} \right) = \frac{\oplus_{(a,c)\leq (b,d)}\ZB_\dgr^\Ffunc ((a,c))}{\oplus_{(a,c)< (b,d)}\ZB_\dgr^\Ffunc ((a,c))} \cong \ZB_\dgr^\Ffunc ((b,d))$,
        \item $C_1 := C_1 \left(\Delta \Seg(P)_{\leq (b,d)}, \Delta \Seg(P)_{< (b,d)}; \underline{\ZB_\dgr^\Ffunc} \right) \\ = \frac{\left(\oplus_{(a,c)<(b,d)} \ZB_\dgr^\Ffunc((a,c))\right) \oplus \left(\oplus_{I < J <(b,d)} \ZB_\dgr^\Ffunc(I)\right)}{\left(\oplus_{I < J <(b,d)} \ZB_\dgr^\Ffunc(I)\right)} \\ \cong \oplus_{(a,c)<(b,d)} \ZB_\dgr^\Ffunc((a,c))$
        \item The boundary operator $C_1 \cong \oplus_{(a,c)<(b,d)} \ZB_\dgr^\Ffunc((a,c))  \xrightarrow[]{\partial_1} C_0 \cong \ZB_\dgr^\Ffunc ((b,d))$ is given by component-wise inclusions.
    \end{enumerate}
    Hence, we obtain that $\Ima(\partial_1) = \sum_{(a,c)<(b,d)} \ZB_\dgr^\Ffunc((a,c))$. Therefore,
    \begin{align*}
        H_0^\downarrow \ZB_\dgr^\Ffunc ((b,d)) = \frac{C_0}{\Ima(\partial_1)} &= \frac{\ZB_\dgr^\Ffunc ((b,d))}{\sum_{(a,c)<(b,d)} \ZB_\dgr^\Ffunc ((a,c))} \\ &\cong \ZB_\dgr^\Ffunc ((b,d)) \cap \left(\sum_{(a,c)<(b,d)} \ZB_\dgr^\Ffunc ((a,c))\right)^\perp \\ &= \goi \left(\ZB_\dgr^\Ffunc\right) ((b,c)),        
    \end{align*}
    where the last equality follows from the definition of $\goi$ (\cref{defn: goi}). 
\end{proof}

\section{Discussion}\label{sec:disc}

We introduced the concept of Orthogonal Inversion and demonstrated its utility in deriving easily interpretable persistence diagrams for multiparameter filtrations. Orthogonal Inversion  critically relies on the inner product structure on $V$ enabling it to capture significantly finer information than integer-valued persistence diagrams. We believe that, as a consequence, Grassmannian persistence diagrams have the potential to contribute to a range of practical applications and data analysis tasks. 
In this section, we outline several key considerations that may motivate further research in this area.
 
\paragraph{Stability.} When comparing two filtrations and their respective Grassmannian persistence diagrams (\cref{subsec: stability of grassmannian pd}), we are required that there exists a fixed simplicial complex $K$ that each filtration is defined over. It is  natural to seek  a framework that can handle filtrations over different vertex sets. 

\paragraph{Orthomodular lattices.} While the motivation behind the concept of orthogonal inversions primarily stems from its applications in TDA (as in~\cref{subsec: gpd for multiparameter}), there is an inherent interest in broadening the utility of orthogonal inversions beyond the scope of TDA by potentially utilizing the notion of Orthomodular Inversion developed in~\cref{sec: orthomodular inversion sec}.

\paragraph{Functoriality.} We have extended the functoriality of classical M\"obius inversion (see~\cite[Proposition 6.6]{edit} and~\cite[Proposition 5.9]{gal-conn}) by demonstrating that assigning a monoidal M\"obius inverse to a monotone-space function via Orthogonal Inversion retains functoriality; see~\cref{thm: goi functor and stable}. However, since monoidal M\"obius inverses are not necessarily unique, a natural question arises: is the process of assigning the set of all monoidal M\"obius inverses of a given function itself functorial? 

\paragraph{Persistent M\"obius Homology.} We have shown that the stratum-$0$ persistent Möbius homology of a filtration produces spaces that are isomorphic to those determined by the Grassmannian persistence diagram of the filtration.  A natural next step is to explore the potential construction of higher order Grassmannian persistence diagrams that could be related to persistent Möbius homology of the filtration in strata 1 and above.

\phantomsection
\addcontentsline{toc}{section}{References}
\bibliographystyle{alpha}
\bibliography{ref}{}

\appendix

\section{Grothendieck Group Completion}\label{appendix:details}
The details below pertain to the discussion on page \pageref{pg:mi}. \medskip

Let $(\mathcal{M}, + , 0)$ be a commutative monoid. Consider the equivalence relation $\sim$ defined on $\mathcal{M}\times \mathcal{M}$ given by
\[
(m_1, n_1) \sim (m_2, n_2) \iff \text{ there exists } k\in \mathcal{M} \text{ such that } m_1 + n_2 + k = m_2 + n_1 +k.
\]
We denote by $[(m_1,n_1)]$ the equivalence class containing $(m_1, n_1)$. Let $\kappa (\mathcal{M}) := \mathcal{M} \times \mathcal{M} / \sim$ be the set of equivalence classes of $\sim$. $\kappa(\mathcal{M})$ inherits the binary operation of $\mathcal{M}$
\[
+ : \kappa(\mathcal{M}) \times \kappa (\mathcal{M}) \to \kappa(\mathcal{M})
\]
by applying it component-wisely  
\begin{align*}
        [(m_1, n_1)] + [(m_2, n_2)] := [(m_1+m_2, n_1 +n_2)].
\end{align*}
The tuple $(\kappa(\mathcal{M}), +, [(0,0)])$ determines an abelian group, called 
the~\emph{Grothendieck group completion of $\mathcal{M}$}. Observe that there is a canonical morphism 
\begin{align*}
    \varphi_\mathcal{M} : \mathcal{M} &\to \kappa (\mathcal{M}) \\
    m           &\mapsto [(m,0)].
\end{align*}

\begin{definition}[Absorbing element]
    An element $\infty_\mathcal{M} \in \mathcal{M}$ is called an~\emph{absorbing element} if $m+\infty_\mathcal{M} = \infty_\mathcal{M}$ for every $m\in \mathcal{M}$.
\end{definition}

\begin{proposition}\label{prop: absorbing implies trivial}
    Let $\mathcal{M}$ be a commutative monoid with an absorbing element $\infty_\mathcal{M}$. Then, the Grothendieck group completion of $\mathcal{M}$ is the trivial group.
\end{proposition}

\begin{proof}
    Let $(m_1, n_1), (m_2, n_2) \in \mathcal{M} \times \mathcal{M}$. Observe that $(m_1, n_1) \sim (m_2, n_2)$ because
    \[
    m_1 + n_2 + \infty_\mathcal{M} = \infty_\mathcal{M} = m_2 + n_1 + \infty_\mathcal{M}.
    \]
    As $(m_1, n_1), (m_2, n_2) \in \mathcal{M} \times \mathcal{M}$ were arbitrary, we conclude that there is only one equivalence class. Namely, $\kappa(\mathcal{M}) = \{ [(0,0)] \}$.
\end{proof}

\begin{corollary}
    The Grothendieck group completion of $\gr(V)$ is the trivial group.
\end{corollary}

\begin{proof}
    $V \in \gr(V)$ is an absorbing element. The result follows from~\cref{prop: absorbing implies trivial}.
\end{proof}

\section{An isomorphism result}\label{subsec: pers mob hom of modules}

In contrast with our approach (i.e. \cref{defn: pers mob hom of filtrations}), in~\cite{patel-skraba-mobius}, Patel and Skraba do not directly consider simplicial filtrations $\Ffunc:P\to \subcx(K)$ and,  instead, define  persistent Möbius homology of  modules $M:P\to \Vec$ by first considering a free presentation $\Phi:X\Rightarrow M$ of $M$ which they use to induce an algebraic proxy $\ZB_\Phi$  for the unavailable birth-death space $\ZB^\Ffunc_\rho$ of $\Ffunc$.  In this section, we verify that our approach, consisting of applying M\"obius homology to $\ZB^\Ffunc_\rho$ (as depicted in \cref{fig:mh-F}), agrees with the result of applying M\"obius homology to an algebraic proxy $\ZB_\Phi$, where $\Phi$ is an arbitrary free resolution of $M:=H_\rho(F)$.

Before presenting our result, we highlight some of the key insights from~\cite{patel-skraba-mobius} as well as our strategy:

\begin{enumerate}
    \item The definition of persistent M\"obius homology of a module $M : P \to \Vec$ initially depends on a choice of free presentation of $M$; see~\cite[Definition 5.8]{patel-skraba-mobius}.
    \item Patel and Skraba show that, while the choice of the free presentation  $\Phi$ may affect the algebraic proxy $\ZB_{\Phi}$, the persistent M\"obius homology spaces $H_s^\downarrow\ZB_{\Phi} ((b,d))$ are independent of $\Phi$ for $(b,d)\in \Seg(P)\setminus\diag(P)$; see~\cite[Section 5.1]{patel-skraba-mobius}.
    \item This result is established in~\cite[Proposition 5.12]{patel-skraba-mobius}, which states that, for any free presentation $\Phi$, $H_s^\downarrow\ZB_{\Phi}$ is isomorphic to an intrinsic object, namely, the M\"obius homology of the kernel module of $M$ (\cref{defn: kernel module}), on  non-diagonal segments. 
    \item Consequently, this suggests that persistent M\"obius homology of a module M could have been defined intrinsically through its kernel module, thus bypassing the need for a free presentation.
\end{enumerate}

Therefore, rather than focusing on the definition of persistent Möbius homology that is given via free presentations, we leverage this intrinsic characterization and demonstrate that our approach—applying Möbius homology directly to the birth-death space—naturally aligns with the one obtained from the kernel module. To be more precise, our main result in this section is the following.

\begin{theorem}\label{thm: mob hom ker is mob hom fil}
    \label{cor: mod or filtration gives the same}
    Let $\Ffunc : P \to \subcx(K)$ be a filtration. Let $\dgr\geq 0$ and $s\geq 0$ be any integers. Let $M:= H_\dgr(\Ffunc)$ be the persistence module obtained from $\Ffunc$ by applying the degree-$\dgr$ homology functor. Let $K_M$ be the kernel module of $M$. Then, for every $(b,d) \in \Seg(P)$ with $b<d$, we have $$H_s^\downarrow ZB_\dgr^\Ffunc ((b,d)) \cong H_s^\downarrow K_M ((b,d)).$$
\end{theorem}

We now recall from \cite{patel-skraba-mobius} the definition of the kernel module of a persistence module $M : P \to \Vec$, which will be a module defined on $\overline{P}^\times = (\Seg(P), \prodord)$ (recall the definitions of $\Seg(P)$ and the product order $\prodord$ in~\cref{parag: poset of int}), and then proceed to prove~\cref{thm: mob hom ker is mob hom fil}.

\begin{definition}[Kernel module~{\cite[Definition 5.11]{patel-skraba-mobius}}]\label{defn: kernel module}
    Let $M : P \to \Vec$ be a persistence module. The \emph{kernel module} of $M$, denoted 
    $$K_M : \overline{P}^\times \to \Vec$$ is defined as follows. For every segment $(p, p')$, where $p' \neq \infty$, $K_M$ assigns the kernel of the morphism $M(p\leq p')$. For segments of the form $(p, \infty)$, $K_M$ assigns the vector space $M(p)$. For every $(p_1, p_2)\leq (p_3, p_4) \in \Seg(P)$, $K_M$ assigns the appropriate composition of morphisms below, depending on whether $p_2$ or $p_4$ is $\infty$:
    \begin{center}
        \begin{tikzcd}
\ker(M(p_1\leq p_2)) \arrow[d, hook] \arrow[rr, dashed] &  & \ker(M(p_3\leq p_4)) \arrow[d, hook] \\
M(p_1) \arrow[rr, "M(p_1\leq p_3)"]                     &  & M(p_3)                              
\end{tikzcd}
    \end{center}
\end{definition}

We now present the proof of~\cref{cor: mod or filtration gives the same}. This proof is similar in structure to the approach adopted in~\cite[Proposition 5.12]{patel-skraba-mobius}.

\begin{proof} [Proof of \cref{cor: mod or filtration gives the same}]
    Let $D_\dgr^\Ffunc : \overline{P}^\times \to \Vec$ be the module that assigns to every segment $(b,d)$ the birth-death space $\ZB_\dgr^\Ffunc ((b,b))$ and to every pair $(b_1, d_1) \leq (b_2, d_2) \in \Seg(P)$ the canonical inclusion $\ZB_\dgr^\Ffunc(b_1, b_1) \hookrightarrow \ZB_\dgr^\Ffunc(b_2, b_2)$. As explained in \cite[Lemma 5.14]{patel-skraba-mobius}, $H_s^\downarrow D_\dgr^\Ffunc ((b,d)) = 0$, for every segment $(b,d) \in \Seg(P)$ with $b<d$, and for every $s\geq 0$. 

    Observe that the modules $D_\dgr^\Ffunc$, $\ZB_\dgr^\Ffunc$, and $K_M$ fit into a short exact sequence:
    \[
    0 \to D_\dgr^\Ffunc \to \ZB_\dgr^\Ffunc \to K_M \to 0.
    \]
    For every $(b,d)\in \Seg(P)$, this short exact sequence of modules gives rise to a short exact sequence of M\"obius chain complexes: 
    \[
    0 \to C_\bullet^\downarrow D_\dgr^\Ffunc (b,d) \to C_\bullet^\downarrow \ZB_\dgr^\Ffunc ((b,d)) \to C_\bullet^\downarrow K_M ((b,d)) \to 0,
    \]
    which, in turn, gives rise to a long exact sequence of M\"obius homology vector spaces:
    \begin{center}
        \begin{tikzcd}
\cdots \arrow[r] & {H_{s+1}^\downarrow K_M ((b,d))} \arrow[r]             & {H_{s}^\downarrow D_\dgr^\Ffunc ((b,d))} \arrow[r]   & {H_{s}^\downarrow \ZB_\dgr^\Ffunc ((b,d))} \arrow[d, "\cong" {anchor=south, rotate=270}] \\
\cdots           & {H_{s-1}^\downarrow \ZB_\dgr^\Ffunc ((b,d))} \arrow[l] & {H_{s-1}^\downarrow D_\dgr^\Ffunc ((b,d))} \arrow[l] & {H_{s}^\downarrow K_M ((b,d))} \arrow[l]                    
\end{tikzcd}
    \end{center}

Since both $H_{s}^\downarrow D_\dgr^\Ffunc ((b,d))$ and $H_{s}^\downarrow D_\dgr^\Ffunc ((b,d))$ are zero for $b<d$, we obtain that the vertical arrow in the above diagram has to be an isomorphism.
    \end{proof}

\end{document}